\documentclass[12pt]{amsart}
\usepackage{amsmath, amsthm, amssymb}
\usepackage{fullpage}
\usepackage{xcolor}
\usepackage{hyperref}
\usepackage{soul}
\usepackage{enumitem}

\newcommand{\blue}[1]{{\color{black}#1}}

\newtheorem{theorem}{Theorem}
\newtheorem{lemma}{Lemma}
\newtheorem{proposition}[lemma]{Proposition}
\newtheorem{corollary}[lemma]{Corollary}
\newtheorem{definition}[lemma]{Definition}
\newtheorem{remark}[lemma]{Remark}
\newtheorem{conjecture}{Conjecture}
\numberwithin{lemma}{section}

\numberwithin{equation}{section}

\newcommand{\R}{{\mathbb R}}
\newcommand{\C}{{\mathbb C}}
\newcommand{\N}{{\mathbb N}}

\renewcommand{\R}{\mathbb R}

\newcommand{\bM}{\mathbf M}
\newcommand{\bP}{\mathbf P}

\newcommand{\bR}{\mathbf R}

\newcommand{\bI}{\mathbf I}
\newcommand{\bJ}{\mathbf J}
\newcommand{\bK}{\mathbf K}

\newcommand{\bu}{{\bar u}}
\newcommand{\bv}{{\bar v}}
\newcommand{\bfu}{{\mathbf u}}

\newcommand{\Ns}{N^\sharp}

\renewcommand{\S}{{\mathbf S}}

\newcommand{\step}{2}
\newcommand{\la}{\langle}
\newcommand{\ra}{\rangle}

\newcommand{\ol}{\overline}
\newcommand{\ms}{M^\sharp}
\newcommand{\ps}{P^\sharp}

\newcommand{\calR}{\mathcal{R}}

\begin{document}

\title{Global solutions for cubic quasilinear Schr\"odinger flows in two and higher dimensions}

\author{Mihaela Ifrim}
\address{Department of Mathematics, University of Wisconsin, Madison}
\email{ifrim@wisc.edu}

\author{ Daniel Tataru}
\address{Department of Mathematics, University of California at Berkeley}
\email{tataru@math.berkeley.edu}

\begin{abstract}

In recent work the authors proposed a 
broad global well-posedness conjecture
for cubic defocusing dispersive equations in one space dimension, and then proved 
this conjecture in two cases, namely for
one dimensional semilinear and quasilinear Schr\"odinger flows. 

Inspired by the circle of ideas developed 
in the proof of the above conjecture, in this paper we expand the reach of these 
methods to higher dimensional quasilinear cubic Schr\"odinger flows. The study of this class of problems, in all dimensions, was initiated  in pioneering work of Kenig-Ponce-Vega for localized initial data, and then continued by Marzuola-Metcalfe-Tataru (MMT) for initial data in Sobolev spaces. 

The outcomes of this work are (i) a new, potentially sharp  local well-posedness result
in low regularity Sobolev spaces, one 
derivative below MMT and just one half derivative above scaling, (ii) a small data global well-posedness and scattering result at the same regularity level, the first result of its kind at least in two 
space dimensions, and (iii) a new way to think about this class of problems, which, we believe, will  become the standard approach in the future.

\bigskip

\end{abstract}


\maketitle

\setcounter{tocdepth}{1}
\tableofcontents


\section{Introduction}

This article is devoted to the study of both local and global well-posedness for quasilinear cubic Schr\"odinger flows in two and higher space dimensions, following the authors' recent work in \cite{IT-qnls} on  the similar problem in one space dimension. 

 The starting point for the work in \cite{IT-qnls} was in turn the \emph{global well-posedness conjecture for 1D cubic defocusing dispersive flows}, also introduced by the authors in \cite{IT-global} a year earlier. This conjecture was first proved in \cite{IT-global} for semilinear Schr\"odinger flows, and then in \cite{IT-qnls} for quasilinear Schr\"odinger flows. Further, the approach in \cite{IT-qnls} also led to a major improvement to the local 
well-posedness theory, up to the optimal threshold.

The cubic higher dimensional problem we consider here has received considerable interest over the years, beginning with the first local well-posedness results of Kenig-Ponce-Vega~\cite{KPV}, for regular and localized initial data, which was expanded to a full well-posedness theory  for small data in lower regularity Sobolev spaces  by Marzuola-Metcalfe-Tataru~\cite{MMT2}, and later for large data in \cite{MMT3}.
One key element of the latter papers was the introduction of the local energy spaces. This was instrumental in allowing the study of the problem for nonlocalized
data, i.e., in standard translation invariant Sobolev spaces. 

On the other hand, the approach introduced  in \cite{IT-qnls} for the one dimensional 
problem departs radically from \cite{MMT2},
and is instead based on 
bilinear estimates via an interaction Morawetz analysis, and normal forms ideas
implemented via modified energies. Our aim here is to extend the reach of these ideas
to higher dimensional problems, leading 
to a drastic expansion of both local 
and global results. Precisely, our goals in this paper are fourfold:
\begin{enumerate}[label=\roman*)]
    \item to produce a novel, better approach for the study of both local and global dynamics in higher dimensional problems. 
    \item to propose a higher dimensional version of our global well-posedness
    conjecture, particularly  in the 
    most interesting case of two-dimensional problems.
    \item to sharply improve the local well-posedness theory for cubic quasilinear Schr\"odinger flows
    in higher dimension, one full
derivative below the earlier results of \cite{MMT2} and only a half derivative above scaling, which may well be optimal  depending on the dimension.
    \item to prove the global well-posedness conjecture for cubic quasilinear Schr\"odinger flows with nonlocalized data, at  the same Sobolev regularity as the local well-posedness result. This is 
    the first result of its kind in the most interesting case of 2D quasilinear flows.
    
\end{enumerate}

\subsection{ The global well-posedness conjectures}
As noted earlier, the present work is motivated in part by a broad 
conjecture formulated by the authors in a recent paper \cite{IT-global}. Our global well-posedness (GWP) conjecture, which applies to both semilinear and quasilinear 1D problems, is as follows:

\begin{conjecture}[Non-localized data defocusing GWP conjecture]\label{c:nld}
One dimensional dispersive flows on the real line with cubic defocusing nonlinearities and small initial data have global in time, scattering solutions.
\end{conjecture}

The main result of \cite{IT-global} asserts that this conjecture is true first in the semilinear 1D Schr\"odinger setting. This was the first global in time well-posedness result of its type. 

Next, in \cite{IT-qnls}, the authors established the above conjecture 
for 1D quasilinear Schr\"{o}dinger flows, which represented the first validation of the conjecture in a quasilinear setting. Precisely, we proved  that if the problem has \emph{phase rotation symmetry} and is \emph{conservative}  and \emph{defocusing}, then  small data in Sobolev spaces yields global, scattering solutions.   Notably, scattering here was interpreted in a weak sense, to mean that the solution satisfies global $L^6_{t,x}$ Strichartz estimates and bilinear $L^2_{t,x}$ bounds. This is due to the strong nonlinear effects that preclude any classical scattering for the problem.

The defocusing condition for the nonlinearity is essential for the global result in 1D. In the focusing case, the existence of small solitons generally prevents global, scattering solutions. Nevertheless, in another recent paper the authors have conjectured that long-time solutions can be  obtained on a potentially optimal time scale:

\begin{conjecture}[Non-localized data long-time well-posedness conjecture]\label{c:nld-focusing}
One-dimensional \\ dispersive flows on the real line with cubic conservative nonlinearities and  initial data of size $\epsilon \ll 1$ have  
long time solutions on the $\epsilon^{-8}$ time scale.
\end{conjecture}

The conservative assumption here is very natural. It is heuristically aimed at preventing nonlinear ode blowup of solutions with wave packet localization.
\medskip

In this article we are interested in the 
counterpart of this conjecture in two and higher space dimensions. Compared to our earlier one-dimensional work, one difference is that here we have more dispersion, which is reflected in the fact that the $L^4_{t,x}$ space-time norm of the solutions, rather than the $L^6_{t,x}$ one,
plays the leading role.
 In particular, in the semilinear case  the small data problem can be approached 
directly using Strichartz estimates, and global well-posedness and scattering follow. The focusing/defocusing character of  the problem only plays a role in the large data case, see e.g., Dodson's results in \cite{D,D-focus}.

Thus, the interesting class of open problems remains the quasilinear one, for which Strichartz estimates are no longer readily available.  For such problems, we formulate here the following conjecture:

\begin{conjecture}[2D non-localized data  GWP conjecture]\label{c:nld2}
Two and higher dimensional quasilinear dispersive flows with cubic  nonlinearities and small initial data have global in time, scattering solutions.
\end{conjecture}

\blue{To be clear, this is for evolutions in $\R^n$, $n \geq 2$, and applies to flows which are 
invariant with respect to translations.} Here scattering is interpreted in the classical fashion, with the only caveat that, due to the nature of quasilinear  
well-posedness, the scattering state can only be expected to depend continuously in the initial data in the strong  topology.

While for convenience we have stated this 
conjecture in all dimensions, the most interesting case by far is the two-dimensional case, where, at this point, no small data global well-posedness result exists. This is in part due to the fact that in two space dimensions $L^4_{t,x}$ is a sharp Strichartz norm, whereas in higher  dimensions there is some room. Indeed, 
in dimension three and higher global well-posedness has been proved for very specific models, using a combination of energy estimates and constant coefficient Strichartz estimates, properly interpolated, see e.g., \cite{HLT}.

In this article we will prove this conjecture for all cubic quasilinear Schr\"odinger flows.  In dimension three and higher, and even in dimension two at higher regularity, no additional structural assumptions are required. However, in order to reach the 
local well-posedness Sobolev threshold in dimension two, we require the flow to be conservative.

\subsection{ Quasilinear Schr\"odinger flows: local solutions}

The most general form for a  quasilinear Schr\"odinger flow  in $n$ space dimensions is
\begin{equation}\tag{DQNLS}
\label{dqnls}
\left\{ \begin{array}{l}
i u_t + g^{jk}(u,\partial_x u ) \partial_j \partial_k  u = 
N(u,\partial_x u) , \quad u:
\R \times \R^n \to \C \\ \\
u(0,x) = u_0 (x),
\end{array} \right. 
\end{equation}
with a metric $g_{jk}$ which is a real valued symmetric matrix and a complex valued source term $N$, both of which are smooth functions of their arguments. 
Here smoothness is interpreted in the real sense. 
But if $g$ and $N$ are analytic, then they can also be thought of as separately (complex) analytic functions of $u$ and $\bu$. This is the interpretation we use in the present paper.

In parallel we consider the seemingly simpler problem
\begin{equation}\tag{QNLS}
\label{qnls}
\left\{ \begin{array}{l}
i u_t + g^{jk}(u) \partial_j \partial_k u = 
N(u,\partial_x u) , \quad u:
\R \times \R \to \C \\ \\
u(0,x) = u_0 (x),
\end{array} \right. 
\end{equation}
where $N$ is  at most quadratic in $\partial u$. This can be seen as the differentiated form of \eqref{dqnls}. We will state our results for both
of these flows, but, in order to keep the exposition focused on the important issues, we will provide complete arguments only for \eqref{qnls}.

In the present work we make the key assumption that 
\begin{equation}\label{elliptic}
g(0) = I_n,    
\end{equation}
which can be easily generalized to any positive definite matrix. But there has also been considerable interest in the 
ultra-hyperbolic case, where $g(0)$ is a symmetric nondegenerate matrix of indefinite case.

The first question we consider here is the local well-posedness question in Sobolev spaces.
This is not at all a straightforward question,
and is in particular it is  much harder than the corresponding question for nonlinear wave equations. One primary reason for this is the infinite speed of propagation, which introduces  short-term growth effects  that  are only visible on long time scales for hyperbolic problems. 
This goes back to the work of Takeuchi~\cite{MR1191488}, Mizohata~\cite{MR0860041} and Ichinose~\cite{MR0948533}.

One effect of this is that in problems with quadratic nonlinearities $N$, one generally does not have local well-posedness in any $H^s$ space,
at least not without making additional decay assumptions at infinity. This happens
even for semilinear equations where the metric $g$ is simply the Euclidean one. 

Historically, the first problems of this type that were studied were the semilinear ones, starting with the work of Kenig-Ponce-Vega~\cite{MR1230709} for small data, Hayashi-Ozawa~\cite{MR1255899} and then Chihara~\cite{MR1344627} for large data, also using ideas of Doi~\cite{MR1284428}.
Finally, Kenig-Ponce-Vega~\cite{MR1660933} also considered the ultra-hyperbolic case, where the Laplacian is replaced by a nondegenerate second-order operator with indefinite signature. All the above results are for localized data, and the first result in translation invariant Sobolev spaces was due to Bejenaru-Tataru~\cite{MR2443925}.

Moving on to the full quasilinear case,
the first problems that were considered there were the one-dimensional ones, 
for which we refer the reader to the work 
of de Bouard-Hayashi-Naumkin-Saut~\cite{BNPS}, Colin~\cite{Co}, Poppenberg~\cite{Pop} and Lin-Ponce~\cite{Lin-Ponce}, leading up to the authors' recent article \cite{IT-qnls}.

The first local well-posedness result for the quasilinear problem in higher dimension goes back to work of Kenig-Ponce-Vega~\cite{KPV}, where
these equations were studied for  regular and localized initial data. Similar results in the ultra-hyperbolic case were then obtained by Kenig-Ponce-Vega-Rovlung \cite{KPRV1,KPRV2}. The next step was 
to study the local well-posedness question in  translation invariant Sobolev spaces
and at lower regularity; this was accomplished 
by Marzuola-Metcalfe-Tataru, first for small data  in \cite{MMT1,MMT2}, and then for large data in \cite{MMT3}; the latter 
article considers only the elliptic case,
while the ultra-hyperbolic case was only recently studied by Pineau-Taylor~\cite{2023arXiv231019221P},
also related to the work of Jeong-Oh~\cite{2024arXiv240206278J}.
Another very recent work of 
Shao-Zhou~\cite{2023arXiv231102556S} provides a nice alternate approach for the 
results of \cite{KPV,MMT2}.

One key distinction made in these last papers was between quadratic and cubic nonlinearities. It is only in the latter case that the local well-posedness 
is proved in classical $H^s$ Sobolev spaces.
Another distinction is between small and large data; in the latter case, large energy growth may occur on arbitrarily short time scales, and a nontrapping condition is also required.

For work on other related models we refer the reader to Colin~\cite{MR1886962}, Chemin-Salort~\cite{MR3336355}, as well as the work of Huang-Tataru~\cite{HT1,HT2}, followed by  the  small data global result of Huang-Li-Tataru~\cite{HLT}; the last three articles are concerned with the skew mean curvature flow, which can be formulated as a quasilinear Schr\"odinger evolution when represented in a suitable gauge.

In order to work in $H^s$ Sobolev spaces,
in the present article we restrict our attention to the cubic case. By this, we mean

\begin{definition}
 We say that the equation \eqref{qnls}/\eqref{dqnls} is \emph{cubic} if $g$ is at least quadratic and $N$ is at least cubic at zero.   
\end{definition}
 
 For reference and easy comparison, we state the earlier local well-posedness results here, specialized to cubic nonlinearities:

\begin{theorem}[cubic nonlinearities \cite{MMT2,MMT3}]\label{t:regular}
 The nD cubic problem \eqref{qnls} is locally well-posed 
for nontrapping data in $H^s$ for $s > \frac{n+3}2$, and the cubic problem \eqref{dqnls} is locally well-posed 
in $H^s$ for $s > \frac{n+5}2$.
\end{theorem}

The nontrapping condition above is needed for large data only, as it is automatically satisfied for small data.
In the 1D case, this result was improved in \cite{IT-qnls} up to the natural, optimal Sobolev threshold:

\begin{theorem}[1D cubic nonlinearities \cite{IT-qnls}]\label{t:local1}
 The 1D cubic problem \eqref{qnls} is locally well-posed for small initial data in $H^s$ for $s > 1$, and the 1D cubic problem \eqref{dqnls} is locally well-posed in $H^s$ for $s > 2$.
\end{theorem}

This is not only a major improvement over Theorem~\ref{t:regular}, it is also likely sharp for generic problems of this type, see the heuristic arguments in \cite{IT-conjecture}. 
\bigskip

We now turn our attention to the first objective of this paper, which is to study the same local well-posedness problem in two and higher space dimensions. Our results are as follows:

\begin{theorem}\label{t:local2}
Let $n = 2$. The 2D cubic problem \eqref{qnls} is locally well-posed for small data in $H^s$ for $s >  \frac32$, and the 2D cubic problem \eqref{dqnls} is locally well-posed for small data in $H^s$ for $s >  \frac52$.
\end{theorem}

\begin{theorem}\label{t:local3+}
Let $n \geq 3$. The nD cubic problem \eqref{qnls} is locally well-posed for small data in $H^s$ for $s >  \frac{n+1}2$, and the nD cubic problem \eqref{dqnls} is locally well-posed for small data in $H^s$ for $s >  \frac{n+3}2$.
\end{theorem}

Here well-posedness is interpreted in the Hadamard sense, and in effect in an enhanced Hadamard form, which includes
:
\begin{itemize}
    \item existence of solutions in  $C([0,T]; H^s)$.
    \item uniqueness of regular solutions.
    \item uniqueness of rough solutions as 
    uniform limits of regular solutions.
    \item continuous dependence on the initial data.
    \item Lipschitz dependence of the solutions on the initial data in a weaker topology ($L^2$)
    \item higher regularity: more regular data yields more regular solutions.
\end{itemize}

These results do not require any structure for the nonlinearity, other than that it is cubic. While the 
formulation is identical in two and higher dimensions,
we have stated the two results separately in order
to emphasize the fact the proofs are different, with the
two dimensional case being the more difficult one.
 
To better understand their significance, one may compare both of these results with the leading order scaling for these problems, which corresponds to the exponents $s_c = \frac{n}2$ 
for \eqref{qnls}, respectively $s_c = \frac{n+2}2$ for \eqref{qnls}.
So our results are only $1/2$ derivative above scaling, which is similar to the earlier one dimensional result. However, this similarity at the level of the results does not translate into a similarity at the level of the proofs.

Concerning the optimality of the local well-posedness results in the above theorems, we propose the following 
\begin{conjecture}
The range of $s$ in Theorems~\ref{t:local2},\ref{t:local3+}    is
generically sharp, except possibly for the endpoint.
\end{conjecture}

 This is however due to a different reason than in the one dimensional case, where the main obstruction appears to come from the self-interaction of a single wave-packet, see \cite{IT-conjecture}. Instead, in the higher dimensional setting the obstruction seems related to trapping, which generically should fail for $s$ below our thresholds even in the small data case.

In a related vein, we also propose

 \begin{conjecture}
The results in Theorems~\ref{t:local2},\ref{t:local3+} also hold for large data, under an additional nontrapping assumption.
\end{conjecture}

One main difference between various dimensions arises at the level of the Strichartz estimates. In one dimension we have a better toolbox for short-time Strichartz, see \cite{IT-qnls}. On the other hand 
in dimension three and higher the $L^4_{t,x}$ Strichartz bound
is obtained directly from the bilinear analysis.
By far the most difficult case is the two dimensional one, for which the 2D toolbox is much smaller than in the 1D case, while the bilinear estimates are weaker than in higher dimension.

The starting point for the proof of our results is provided by the earlier results stated  in Theorem~\ref{t:regular}, which, say for \eqref{qnls}, apply for initial data in $H^\sigma$ with $\sigma > \frac{n+3}2$. The lifespan of these solutions a priori depends on the $H^{\sigma}$ size of the initial data, which  can be arbitrarily large.
However, once we prove suitable \emph{a-priori} bounds on these solutions, we will be able to extend the lifespan so that it depends only on the $H^{s}$ size of the data.  

Once we have regular solutions on uniform time scales, we produce the rough, $H^{s}$ solutions as unique uniform limits of regular solutions. 
This requires careful estimates for the linearized 
flow, which we do at the $L^2$ level.

\subsection{ Quasilinear Schr\"odinger flows:  global solutions}
Our objective here is to consider the global well-posedness for the 2D problem for initial data which is small in $H^s$, in the spirit of Conjecture~\ref{c:nld2}.

This 1D counterpart of this conjecture was proved in \cite{IT-qnls}, and we recall this first. 
In contrast to the local well-posedness result, refined long-time and  global in time results in 1D require some natural structural assumptions on the equations, which will be described next. 

\begin{definition}
We say that the equation \eqref{qnls}/\eqref{dqnls}
has  \emph{phase rotation} symmetry if it is 
invariant with respect to the transformation 
$u \to u e^{i\theta}$ for $\theta \in \R$.
\end{definition}

One consequence of this assumption is that 
the Taylor series of the nonlinearity has only
odd terms, which are multilinear expressions 
in $u$ and $\bar u$ (and their derivatives) with one more $u$ factor. In fact, it is sufficient to assume that this holds for the cubic terms in the equation.

To state the next assumptions it is convenient to write the equations in a semilinear way. Assuming the phase rotation symmetry, and also harmlessly assuming that $g(0) = 1$, this semilinear form is
\begin{equation}\label{eq-cubic}
i u_t + \partial^2_x u = C(u,\bar u, u) + \text{higher order terms},
\end{equation}
where $C$ is a translation invariant trilinear form with symbol $c(\xi_1,\xi_2,\xi_3)$. We refer the reader to Section~\ref{s:notations} for  a description of our notations.

\begin{definition}
We say that the 1D equation \eqref{qnls}/\eqref{dqnls}
is \emph{conservative} if  the cubic component of the nonlinearity satisfies 
\[
c(\xi,\xi,\xi), \ \partial_{\xi_j}
c(\xi,\xi,\xi) \in \R .
\]
\end{definition}
It is also 
immediately satisfied if the problem admits a coercive conservation law, but, conversely, it does not imply the existence of such a conservation law.

One final necessary assumption in 1D is that the problem is defocusing:

\begin{definition}
We say that the equation \eqref{qnls}/\eqref{dqnls}
is \emph{defocusing} if  the cubic component of the nonlinearity satisfies 
\[
c(\xi,\xi,\xi) \gtrsim 1+ \xi^2.
\]
\end{definition}

Now we can state the one dimensional result:

\begin{theorem}\label{t:global}
 Consider a 1D \eqref{qnls}/\eqref{dqnls} problem, which we assume to be cubic, phase rotation invariant and defocusing.
Assume that  the initial data $u_0$ is small in $H^s$,
where $s > 1$ for \eqref{qnls} and $s > 2$ for \eqref{dqnls}.
Then the solutions are global in time, and satisfy 
\begin{equation}
\| u\|_{L^\infty_t H^s_x} \lesssim \| u_0\|_{H^s_x}.
\end{equation}
\end{theorem}

Now we turn our attention  to the higher dimensional problem, which differs from the 1D case in two key ways:

\begin{enumerate}
    \item The dispersive decay is stronger in nD, where 
    cubic terms are even perturbative in the semilinear case; this makes the problem easier, and in particular it allows us to discard the defocusing assumption.

    \item There are many more cubic resonant interactions in nD;
this makes the problem harder, and potentially it mostly disables 
normal form/modified energy type arguments.
\end{enumerate}

The first feature above allows us to weaken the hypothesis of the theorem, compared to the $1D$ case, dispensing with both the phase rotation invariance, the conservative assumption and the defocusing condition. Without phase rotation the expansion in 
\eqref{eq-cubic} acquires extra terms,
\begin{equation}\label{eq-cubic2}
i u_t + \Delta u = C(u,\bar u, u)  
+ \text{ non rotation invariant cubic terms} + 
\text{higher order terms},
\end{equation}
where all wave-packet self-interactions are captured by the $C$ 
term. This term does not play any role in higher dimension $n \geq 3$, but we do need to consider its properties in dimension $n = 2$. 
There we introduce the following more  relaxed  conservative condition:
\begin{definition}\label{d:conservative2}
We say that the 2D equation \eqref{qnls}/\eqref{dqnls}
is \emph{conservative} if  the symbol of the cubic, phase rotation invariant  component of the nonlinearity satisfies 
\[
c(\xi,\xi,\xi) \in \R, \qquad \xi \in \R.
\]
\end{definition}

After these preliminaries we are ready to state our global results.
We begin with the higher dimensional case $n \geq 3$, 
where we work at the same regularity level as in the local well-posedness result in Theorem~\ref{t:local3+}: 

\begin{theorem}\label{t:global3}
Let $n \geq 3$.
 a) Assume that the equation \eqref{qnls} is cubic, and that the initial data $u_0$ is small in $H^s$, with $s > \frac{n+1}2$,
 \begin{equation}
 \|u_0\|_{H^s} \leq \epsilon \ll 1.   
 \end{equation}
Then the solutions are global in time,  satisfy the uniform bound 
\begin{equation}
\| u\|_{L^\infty_t H^s_x} \lesssim \epsilon,
\end{equation}
and scatter at infinity, in the sense that
\begin{equation}
u_\infty = \lim_{t \to \infty} e^{-it \Delta} u(t)     \qquad \text{in } H^{s}
\end{equation}
exists and has a continuous dependence on the initial data in the same topology.

b) The same result holds for  the equation \eqref{dqnls} for $s > \frac{n+3}2$. 
\end{theorem}

Next we have a similar result in two space dimensions, but for a more restrictive class of exponents, namely \blue{$1/2$} derivative above the local well-posedness threshold in Theorem~\ref{t:local2}:

\begin{theorem}\label{t:global2}
Let $n = 2$.
 a) Assume that the equation \eqref{qnls} is cubic, and that the initial data $u_0$ is small in $H^s$, with $s \geq \frac{n+1}2+\blue{\frac12}$,
 \begin{equation}
 \|u_0\|_{H^s} \leq \epsilon \ll 1.   
 \end{equation}
Then the solutions are global in time, satisfy the uniform bound
\begin{equation}
\| u\|_{L^\infty_t H^s_x} \lesssim \epsilon,
\end{equation}
and scatter at infinity, in the sense that
\begin{equation}
u_\infty = \lim_{t \to \infty} e^{-it \Delta} u(t)     \qquad \text{in } H^{s}
\end{equation}
exists and has a continuous dependence on the initial data in the same topology.

b) The same result holds for  the equation \eqref{dqnls} for $s \geq \frac{n+3}2 +\blue{\frac12}$. 
\end{theorem}

Finally, in the low regularity case in dimension $n=2$,
under an additional conservative assumption, we are able to bring the global well-posedness result 
to the same regularity level as the local well-posedness result:

\begin{theorem}\label{t:global2c}
Let $n = 2$.
 a) Assume that the equation \eqref{qnls} is cubic and conservative, and that the initial data $u_0$ is small in $H^s$, with $s > \frac{n+1}2$,
 \begin{equation}
 \|u_0\|_{H^s} \leq \epsilon \ll 1.   
 \end{equation}
Then the solutions are global in time, satisfy the uniform bound
\begin{equation}
\| u\|_{L^\infty_t H^s_x} \lesssim \epsilon,
\end{equation}
and scatter at infinity, in the sense that
\begin{equation}
u_\infty = \lim_{t\to \infty} e^{-it \Delta} u(t)     \qquad \text{in } H^{s}
\end{equation}
exists and has a continuous dependence on the initial data in the same topology.

b) The same result holds for  the equation \eqref{dqnls} for $s > \frac{n+3}2 $. 
\end{theorem}

As noted earlier, these results represent the first global well-posedness results  for general cubic quasilinear Schr\"odinger equations with small initial data in Sobolev spaces
in dimension $n \geq 2$.  Indeed, all prior  global results are of following two types:
\begin{enumerate}
    \item Problems with initial data which is not only small but also \emph{smooth and localized}, see for instance  Ginibre-Hayashi~\cite{MR1340853}, Hayashi-Naumkin~\cite{MR1856255} and 
    Bernal-V\'{\i}lchis-Hayashi-Naumkin.
    \cite{MR2811056}.
    \item Specific problems with additional structure in higher dimension $n \geq 3$, see e.g. Huang-Li-Tataru~\cite{HLT}.
\end{enumerate}

The contrast is even more stark in dimension $n=2$, where 
as far as we are aware there is no global result known for any 
cubic quasilinear dispersive problem with small but non-localized initial data. 

We continue with two remarks about the scattering result:

\begin{remark}
One should compare the results on scattering with 
our earlier 1D results in \cite{IT-qnls}. In the 1D
case, the cubic nonlinear effects are too strong,  \blue{preventing} any standard scattering; even for localized data, the best one could hope fore is some form of modified scattering.  By contrast, in two and higher dimension we 
have full classical scattering.
\end{remark}

\begin{remark}
A natural question here is to study the regularity of the 
wave operator $u_0 \to u_\infty$. Our analysis implies that it has properties akin to the (quasilinear) Hadamard well-posedness:
\begin{itemize}
    \item it is continuous in $H^s$
    \item it is Lipschitz continuous\footnote{When restricted to small $H^s$ data.} in a weaker topology 
    (e.g. $L^2$), and in effect close to the identity in the Lipschitz topology.
\end{itemize}
    
\end{remark}

We note that the above results are stated here in a short form, but  the proofs show in effect that the solutions satisfy several types of estimates:

\begin{enumerate}[label=(\roman*)]
\item{Bilinear $L^2_{t,x}$ bounds}, which can be stated in a balanced form
\begin{equation} \label{intro-bi-bal}
   \| |D|^\frac{3-n}2 |\la D\ra^{s-\frac14} u|^2\|_{L^2_{t,x}} \lesssim \epsilon^2 ,  
\end{equation}  
and in an imbalanced form, using standard paraproduct notation,
\begin{equation}\label{intro-bi-unbal}
   \| T_{\partial u} \bu\|_{L^2_t H^{s+\frac12}_x} \lesssim \epsilon^2  .
  \end{equation} 

    \item Lossless $L^{4}_{t,x}$ Strichartz bounds for $n \geq 3$:
  \begin{equation}\label{intro-Str3}
   \| \la D\ra^{s-\frac{2-n}4} u\|_{L^4_{t,x}} \lesssim \epsilon  . 
  \end{equation}  

    \item Lossless $L^{4}_t L^8_x$ Strichartz bounds for $n = 2$
  \begin{equation}\label{intro-Str2}
   \| \la D\ra^{s-\frac14} u\|_{L^4_t L^8_x} \lesssim \epsilon .  
  \end{equation}  
   \item Full Strichartz bounds with derivative loss:
  \begin{equation}\label{intro-Str-full}
   \| \la D\ra^{s-\frac{2}{p}} u\|_{L^{p}_t L^{q}_x} \lesssim \epsilon
   \end{equation}  
for any pair $(p,q)$ of sharp Strichartz exponents.
\end{enumerate}

\medskip

The reader is referred to Theorem~\ref{t:local-fe3} and Theorem~\ref{t:local-fe2} for a more complete, frequency envelope based description of the results.
One may  also compare the $L^4_{t,x}$  bounds above  with the $L^6_{t,x}$ bounds which  play the leading role in the 1D case, where $6$ is the sharp Strichartz exponent.

\bigskip
One may also frame our results within the broader question of obtaining long time solutions for one or two dimensional dispersive flows with quadratic/cubic nonlinearities, which has attracted much of attention in recent years. One can distinguish two different but closely related types of results that have emerged, as well as several successful approaches.

On the one hand, \emph{normal form methods} have been developed 
in order to extend the lifespan of solutions, beginning with 
\cite{Shatah} in the late '80's. Somewhat later, around 2000, the \emph{ I-method}, introduced in \cite{I-method} brought forth the idea of constructing  better almost conserved quantities. These two ideas serve well in the study of semilinear flows, where it was later understood that they are connected ~\cite{Bourgain-nf}.

Neither of these techniques can be directly applied to quasilinear problems. To address this problem, it was discovered in the work of the authors and collaborators \cite{BH}, \cite{IT-g} that one could adapt the normal form method to quasilinear problems by constructing energies that simultaneously capture both the quasilinear and normal form structures. This idea has been called the \emph{modified energy method}, and can also be seen in some sense as a quasilinear adaptation of the I-method. Other alternate approaches, also in the quasilinear setting, are provided by the \emph{ flow method } of  Hunter-Ifrim~\cite{hi}, where a better nonlinear normal form transformation is constructed using a well-chosen auxiliary flow, and by the paradiagonalization method of Alazard and-Delort \cite{AD}, where a paradifferential symmetrization is applied instead.

Further, the question of obtaining scattering, global in time solutions for one or two dimensional dispersive flows with quadratic/cubic nonlinearities  has also been extensively studied in the last two decades  for a number of models, under the assumption that the initial data is both \emph{small} and \emph{localized};  for a few  examples among many, see for instance \cite{HN,HN1,LS,KP,IT-NLS} in one dimension, and \cite{GMS} in two dimensions.
The nonlinearities in the one-dimensional models are primarily cubic, though the analysis has also been extended via normal form and modified energy methods to problems which also have nonresonant quadratic interactions; several such examples are \cite{AD,IT-g,D,IT-c,LLS}, see also further references therein, as well as the authors' expository paper \cite{IT-packet}.
On the other hand some quadratic nonlinearities can also be handled, see for instance \cite{GMS}, based on the idea of \emph{space-time} resonances.

Comparing the above class of results where the initial data
is both small and localized with the present results, without any localization assumption, it is clear that in the latter case the problem becomes much more difficult, because the absence of localization allows for far stronger nonlinear interactions over long time-scales. Another twist in the one dimensional setting was that one also needs to distinguish between the focusing and defocusing case, as was made clear in our  in our earlier  semilinear work in \cite{IT-global} and \cite{IT-focusing}, and then quasilinear work in \cite{IT-qnls}.

\subsection{Outline of the paper} \
While the short form of the results provided in the introduction
represents a good starting point, from the perspective of both 
proving these results and of understanding the global in time dispersive properties of the solutions, it becomes essential 
to have an appropriate family of quantitative bounds for the solutions. As briefly indicated with the bounds \eqref{intro-bi-bal}-\eqref{intro-Str-full}, these bounds are roughly of two types:

\begin{description}
    \item[Strichartz estimates] these represent the more classical bounds one can obtain, and are well-known in the constant coefficient case, though proving such bounds in quasilinear 
    settings becomes anything but straightforward.
    \item[Bilinear $L^2_{t,x}$ estimates] these play the leading role here, and should be seen as transversality bounds, i.e., essentially capturing the bilinear interaction of two transversal waves.   Again these are easily proved in linear, constant coefficient settings, but are not at all straightforward in any quasilinear context.
\end{description}

As it is often the case in the study of quasilinear problems,
proving such estimates requires having a considerable amount of information about the solutions to start with, which is why the entire proof of the results is wrapped within a carefully designed bootstrap argument, where, to start with, one assumes that a weaker version of such estimates already holds, not only at the linear level but also at the bilinear level. 

While the bootstrap requirements make it a necessity that the 
main line of our argument cannot  be closed until the last section, 
we are still able to organize the proofs in a modular fashion,
with the only proviso that the bootstrap assumptions are imposed at each step, but not fully recovered until the very end. These
modules are as follows:

\medskip

\textbf{I. Littlewood-Paley theory and frequency envelopes.} 
While the estimates \eqref{intro-bi-bal}-\eqref{intro-Str-full}
serve to give the reader an idea of our overall setup, for our proofs it is far better for our proofs to have a more precise accounting of these estimates relative to the size of the frequency.. This is  achieved first by using a Littlewood-Paley decomposition of the solutions, 
and second, by using the language of frequency envelopes in order
to track the size of the dyadic pieces in both linear and bilinear bounds. This is described in Section~\ref{s:boot}, where we also provide sharper, frequency envelope versions of our estimates 
for the solutions, as well as the setup of the global bootstrap argument. 

\medskip

\textbf{II. The paradifferential formalism.} As it is standard
in the study of nonlinear evolutions, in order to prove local well-posedness we need good estimates not only for the solution itself, but also for the corresponding linearized equation. As it turns out, rather then study each of these two equations separately, it is better to  seek a common denominator; this is exactly be the associated linear paradifferential flow, which will indeed play
the leading role.   In Section~\ref{s:lin} we introduce all these related flows, and recast both the full equation and the linearized equation as paradifferential flows with source terms. These source terms should ideally play a perturbative role, and indeed they do with two provisos: (i) that we have a suitable notion of ``perturbative" which is sufficiently broad, and (ii) that in the most difficult case of two dimensional flows at low regularity  we give special consideration to the cubic balanced terms.

\medskip

\textbf{III. Conservation laws  in density-flux form.}
Efficient energy estimates must necessarily play an important role in our analysis, but for the proof of the bilinear estimates we need to consider these estimates in a  frequency localized manner,
and, more importantly, as local conservation laws in density-flux form. This is discussed primarily in Section~\ref{s:df}, and 
refined further in the last section for conservative 2D flows.

\medskip

\textbf{IV. Interaction Morawetz bounds.} These represent the primary tool in the proof of the crucial bilinear $L^2_{t,x}$ bounds, and originate in work of the I-team in \cite{MR2415387}. However our starting point is provided by the work of Planchon-Vega~\cite{PV} for the constant coefficient case. We modify and adapt this for our purposes
in Section~\ref{s:para}, following the ideas developed in our prior one dimensional work in \cite{IT-qnls}. However, compared to one dimension, the higher dimensional case turns out to be more difficult due to a certain lack of coercivity, which considerably increases the burden of  
estimating ``perturbative" contributions.

\medskip

\textbf{V. Strichartz estimates in the paradifferential world.}
A direct proof of lossless Strichartz estimates for the paradifferential was possible in \cite{IT-qnls} in the 1D case locally in time, but it becomes a daunting task in higher dimensions. This is partially remedied by the stronger bilinear bounds in the balanced case, which in particular give an essential scale invariant $L^4_{t,x}$ bound in dimension three and higher. But in 2D this is no longer the case, so in Section~\ref{s:Str} we instead develop a strategy to prove the Strichartz estimates for a full range of exponents, globally in time, but with a loss of derivatives.

\medskip

\textbf{VI. Rough solutions as limits of smooth solutions.}
This is by now a standard step in obtaining rough solutions for quasilinear flows, which is best carried out under the umbrella 
of frequency envelopes. We do this in Section~\ref{s:rough}, 
following the path described in our earlier expository paper \cite{IT-primer}.

\medskip

\textbf{VII. The 2D global result at low regularity.}
While cubic balanced terms can be seen as perturbative in dimension three and higher, this is no longer the case in two space dimensions, particularly at low regularity. This is where our ``conservative" assumption plays an important role, which is investigated in the last section of the paper.

\subsection{Acknowledgements} 
The first author was supported by the Sloan Foundation, and by an NSF CAREER grant DMS-1845037. The second author was supported by the NSF grant DMS-2054975 as well as, by a Simons Investigator grant from the Simons Foundation. The authors also thank Ben Pineau for carefully reading the manuscript and pointing out a number of typos/corrections.

 \section{Notations and preliminaries}
\label{s:notations}

 Here we introduce the usual Littlewood-Paley decomposition, 
and discuss some standard facts for constant coefficient linear 
Schr\"odinger flows, including the Strichartz and bilinear $L^2_{t,x}$ estimates, as well as the notion of cubic resonant interactions.
Finally, we describe our notations for multilinear forms and their symbols.

\subsection{ The Littlewood-Paley decomposition}
We use a standard 
 Littlewood-Paley decomposition in frequency: let $\psi$ be a bump function adapted to $[-2,2]$ and equal to $1$ on
$[-1,1]$.  We define the Littlewood-Paley operators $P_{k}$, as well as
and
$P_{\leq } = P_{<k+1}$ for  $k \in \N$ by defining
\[
\widehat{P_{\leq k} f}(\xi) := \psi(2^{-k} \xi) \hat f(\xi)
\]
with $P_1:= P_{\leq 1}$
 and $P_k := P_{\leq k} - P_{\leq k-1}$ for $\lambda \geq 2$.  All the operators $Pk$,
$P_{\leq k}$ are bounded on all translation-invariant Banach spaces, thanks to Minkowski's inequality.  We also define $P_{>k} := P_{\geq k-1}
:= 1 - P_{\leq k}$.

Thus
\begin{equation*}
1=\sum_{k \in \mathbf{N}}P_{k},
\end{equation*}
where the multipliers $P_k$ have smooth symbols localized at frequency $2^k$. Correspondingly, our solution $u$ will be decomposed as
\[
u = \sum_{k \in  \N} u_k, \qquad \mbox{ where } \quad  u_k := P_k u. 
\]
The main estimates we will establish for our solution $u$
will be linear and bilinear estimates for the functions $u_k$.

 In the proof of the local well-posedness result in Section~\ref{s:rough} it will be convenient to switch from a discrete to a continuous Littlewood-Paley decomposition. There we think of $k$ 
as a real nonnegative parameter and define 
\[
P_k u := \frac{d}{dk} P_{< k} u,
\]
and our Littlewood-Paley decomposition for $u$
becomes
\begin{equation}\label{lp-cont}
u = u_0 +\int_0^\infty u_k \, dk.
\end{equation}

To shorten some calculations we will also use Greek indices
for Littlewood-Paley projectors, e.g. $P_\lambda$, $P_\mu$
where we take $\lambda,\mu \in 2^{\N}$. Correspondingly, our Littlewood-Paley decomposition of a function $u$ will read
\[
u = \sum_{\lambda \in  2^\N} u_\lambda, \qquad u_\lambda = P_\lambda u. 
\]

\subsection{Strichartz and bilinear $L^2_{t,x}$ bounds}
Here we begin by recalling  the  Strichartz inequalities, 
which apply to solutions to the inhomogeneous linear Schr\"odinger equation:
\begin{equation}\label{bo-lin-inhom}
(\partial_t + \Delta)u = f, \qquad u(0) = u_0.
\end{equation}
These are $L^p_t L^q_x$ bounds which are used as measures of the dispersive effect. The allowed exponents $(p,q)$ depend on the dimension:

\begin{definition}
The pair $(p,q)$ is a standard Strichartz exponent in $n$ space dimensions if
\begin{equation}
\frac{2}{p}+\frac{n}q = \frac{2}{n}, \qquad 2 \leq p,q \leq \infty    \end{equation}
with the only exception of the forbidden endpoint $(2,\infty)$ 
in dimension $n=2$.
\end{definition}
With this definition, the Strichartz estimates in the $L^2$ setting are summarized in the following

\begin{lemma}
Assume that $u$ solves \eqref{bo-lin-inhom} in $[0,T] \times \R$. Then the following estimate holds for all sharp Strichartz pairs $(p,q), (p_1,q_1)$:
\begin{equation}
\label{strichartz}
\| u\|_{L^p_t L^q_x} \lesssim \|u_0 \|_{L^2} + \|f\|_{L^{p'_1}_t L^{q'_1}_x}.
\end{equation}
\end{lemma}

 For further details and references we direct the reader to \cite{Ke-Ta}, where the final endpoint $p=2$ was proved.

 For practical purposes it is useful to place all 
 these estimates under a single umbrella. In dimension three and higher we have access to both end-points, and it is most efficient 
 to  define the Strichartz space $S$ associated to the $L^2$ flow by
\[
S = L^\infty_t L^2_x \cap L^2_t L^{\frac{2n}{n-2}}_x,
\]
as well as its dual 
\[
S' = L^1_t L^2_x + L^{2} _t L^{\frac{2n}{n+2}}_x .
\]

Then the Strichartz estimates can be summarized all as 
\begin{equation}
\label{strichartz3}
\| u\|_{S} \lesssim \|u_0 \|_{L^2} + \|f\|_{S'}.
\end{equation}

In two space dimensions we lose access to the $(2,\infty)$ endpoint
so the above definition of the Strichartz space $S$ no longer applies. Instead, it is convenient to use the $U^p_\Delta$ and $V^p_\Delta$ spaces. These were introduced in unpublished work of the second author~\cite{T-unpublished}, see also \cite{KT} and \cite{HTT} for some of the first applications of these spaces. In this setting the 
Strichartz estimates can be summarized as
\begin{equation}
\label{strichartz2}
\| u\|_{V^2_\Delta} \lesssim \|u_0 \|_{L^2} + \|f\|_{DV^2_\Delta},
\end{equation}
where the transition to the Strichartz bounds is provided by the embeddings
\begin{equation}\label{UV-embed}
V^2_\Delta \subset U^p_\Delta \subset L^p_t L^q_x,
\qquad L^{p'}_t L^{q'}_x \subset DV^{p'}_\Delta \subset DV^2_\Delta.
\end{equation}

\medskip
The last property of the linear Schr\"odinger equation which is important here is the bilinear $L^2$ estimate, which is as follows:

\begin{lemma}
\label{l:bi}
Let $u^1$, $u^2$ be two solutions to the inhomogeneous Schr\"odinger equation with data 
$u^1_0$, $u^2_0$ and inhomogeneous terms $f^1$ and $f^2$. Assume 
that $u^1$ and $u^2$ are frequency localized in balls $B_1$, $B_2$
with radius $\lambda_1 \lesssim  \lambda_2$ so that  
\[
d(B_1,B_2) \gtrsim \lambda_2.
\]
 Then we have 
\begin{equation}
\label{bi-di}
\| u^1 u^2\|_{L^2_{t,x}} \lesssim \lambda_1^{\frac{n-1}{2}} \lambda_2^{-\frac12} 
( \|u_0^1 \|_{L^2_x} + \|f^1\|_{S'}) ( \|u_0^2 \|_{L^2_x} + \|f^2\|_{S'}).
\end{equation}
\end{lemma}

This is a classical bound, which is most easily proved for solutions to the homogeneous
equation, where it can be seen as a convolution estimate in the Fourier space. We 
provide this bound here for orientation only, as our Schr\"odinger evolutions have variable coefficients. Nevertheless, it provides us with a guide of what we could hope for in our setting.

\subsection{ Resonant analysis}
\label{s:resonances}
In this subsection we aim to classify the trilinear interactions associated to the linear 
constant coefficient flow; for simplicity, the reader may think 
of the context of an equation of the form
\[
(i\partial_t + \Delta) u = C(\bfu,\bfu,\bfu),
\]
where $C$ is a translation invariant trilinear form and we have used the notation $\bfu = \{ u, \bu\}$ in order to allow for both $u$ and $\bu$ in each of the arguments of $C$. For clarity, we note that this includes expressions of the form
\[
C(u,\bu,u), \quad C(u,u,u), \quad C(u,\bu,\bu), \quad C(\bu,\bu,\bu).
\]
Within these choices we distinguish the first one,  namely $C(u,\bu,u)$, 
as the only one with \emph{phase rotation symmetry}.

Given three input frequencies $\xi^1, \xi^2,\xi^3$ for 
our cubic nonlinearity, the output will be at frequency 
\begin{equation}\label{freq}
\xi^4 = \pm \xi^1 \pm \xi^2 \pm \xi^3,
\end{equation}
where the signs are chosen depending on whether we use the input $u$ or $\bu$. We call such an interaction {balanced} if all $\xi^j$'s are of comparable size, and \emph{unbalanced} otherwise.
This is a \emph{resonant interaction} if and only if we have a similar relation for the associated time frequencies, namely 
\[
|\xi^4|^2 = \pm |\xi^1|^2 \pm |\xi^2|^2 \pm |\xi^3|^2.
\]

\medskip 

On the other hand from the perspective of bilinear estimates we 
seek to distinguish the \emph{transversal interactions}, which are 
defined as those for which no more than two of the frequencies $\xi^j$ coincide. We seek to classify interactions from this perspective:

\begin{description}
\item[(i) With phase rotation symmetry]
Here the relation \eqref{freq} can be expressed in a more symmetric fashion as 
\[
\Delta^4 \xi = 0, \qquad \Delta^4 \xi = \xi^1-\xi^2+\xi^3-\xi^4 .
\]
This is a resonant interaction if and only if we have a similar relation for the associated time frequencies, namely 
\[
\Delta^4 \xi^2 = 0, \qquad \Delta^4 \xi^2 = |\xi^1|^2-|\xi^2|^2+|\xi^3|^2-|\xi^4|^2 .
\]
Hence, we define the resonant set in a symmetric fashion as 
\[
\calR := \{ \Delta^4 \xi = 0, \ \Delta^4 \xi^2 = 0\}.
\]
In 1D it is easily seen that this set may be characterized as
\[
\calR = \{ \{\xi^1,\xi^3\} = \{\xi^2,\xi^4\}\}.
\]
However, in higher dimension this is no longer the case, and the resonant set 
$\calR$ consists of all quadruples $\xi^1$, $\xi^2$, $\xi^3$
and $\xi^4$ which are the vertices of a rectangle, in this order.
Of these, the only non-transversal interactions are given by the diagonal set 
\[
\calR_2 = \{\xi^1=\xi^2=\xi^3 = \xi^4\}.
\]
We will refer to these interactions as the \emph{doubly resonant interactions}.

\item[(ii) Without phase rotation symmetry] Here we concentrate  our attention
on describing the worst case scenario, which corresponds to interactions which are both resonant and non-transversal.
The first requirement corresponds to frequencies where
\[
 \xi^1 \pm \xi^2 \pm \xi^3 \pm \xi^4 = 0,
\]
and 
\[
|\xi^1|^2 \pm |\xi^2|^2 \pm |\xi_3|^2 \pm |\xi^4|^2 = 0.
\]
The second requirement is that three of the frequencies are equal.
Without any loss in generality assume that $\xi^1 = \xi^2 = \xi^3:=\xi$.
Depending on the choice of signs and excluding the case of phase rotation symmetry, we obtain one of the following systems:
\[
\xi^4 = \pm 3 \xi, \qquad |\xi^4|^2 = \pm 3 |\xi|^2,
\]
with matched signs, or
\[
\xi^4 = - \xi, \qquad |\xi^4|^2 = - |\xi|^2.
\]
In both cases the only solution is $\xi = 0$. Hence we have established the following

\medskip 

\begin{lemma} Except for the case $(0,0,0,0)$, all cubic interactions 
without phase rotation symmetry are either nonresonant or transversal.
\end{lemma}

\end{description}
We will use the above classification to separate the cubic interactions using smooth frequency cut-offs as 
\[
C(\bfu,\bfu,\bfu) = C^{res}(u,\bu,u) + C^{0}(\bfu,\bfu,\bfu) + 
C^{nr}(\bfu,\bfu,\bfu) + C^{tr}(\bfu,\bfu,\bfu),
\]
where
\begin{itemize}
    \item $C^{res}(u,\bu,u)$ contains balanced interactions with phase rotation symmetry
    \item $C^{0}(\bfu,\bfu,\bfu)$ contains only low frequency interactions
    \item $C^{nr}(\bfu,\bfu,\bfu)$ contains only balanced nonresonant interactions
  \item $C^{tr}(\bfu,\bfu,\bfu)$ contains only transversal interactions, balanced or unbalanced.
\end{itemize}

\subsection{ Translation invariant multilinear forms}

As  our problem and many of our estimates are  invariant with respect to translations, it is natural  that translation invariant multilinear forms play an important role in the analysis.

The simplest examples of multilinear forms are multipliers,
which are convolution operators,
\[
m(D) u (x) = \int K(y) u(x-y) \,dy, \quad \hat K = m.
\]
More generally, we will denote by $L$  any convolution operator 
\[
Lu (x) = \int K(y) u^y(x)\, dy,
\]
where $K$ is an integrable kernel, or 
 a bounded measure, with a universal bound and where
\[
u^y(x) := u(x-y),
\]
is just a shorthand notation for translations of $u$ by $y$.

\medskip

Throughout the paper we will use multilinear operators,  and denote by $L$ any multilinear form 
\[
L(u_1,\cdots, u_k) (x) = 
\int K(y_1, \cdots,y_k) u_1^{y_1}(x) \cdots u_k^{y_k}(x) \, dy,
\]
where, again $K$ is assumed to have a kernel which is integrable, or more generally,  a bounded measure (we allow for the latter in order to be able to include products here).
Multilinear forms satisfy the same bounds as 
corresponding products do, as long as we work in translation invariant Sobolev spaces. Such a form may be equivalently described via its symbol $a(\xi^1, \cdots \xi^k)$, as
\[
\widehat{L(u_1,\cdots, u_k)}(\xi) = (2\pi)^{-\frac{n(k-1)}2}\int_{\xi^1+ \cdots \xi^k = \xi} 
a(\xi^1, \cdots \xi^k) \hat u_1(\xi^1) \cdots \hat u_k(\xi^k) d \xi^1 \cdots d \xi^{k-1},
\]
where the symbol $a$ is the inverse Fourier transform of the kernel $K$.

\medskip

A special role in our analysis is played 
by multilinear forms which exhibit phase rotation symmetry, where the arguments $u$ and $\bar u$ are alternating. For such forms we borrow the notations from our earlier paper \cite{IT-global}.

Precisely, for an integer $k \geq 2$, we will use translation invariant $k$-linear  forms 
\[
(\mathcal D(\R^n))^{k} \ni (u_1, \cdots, u_{k}) \to     Q(u_1,\bu_2,\cdots) \in \mathcal D'(\R^n),
\]
where the nonconjugated and conjugated entries are alternating.

Such a form is uniquely described by its symbol $q(\xi_1,\xi_2, \cdots,\xi_{k})$
via
\[
\begin{aligned}
Q(u_1,\bu_2,\cdots)(x) = (2\pi)^{-nk} & 
\int e^{i(x-x_1)\xi^1} e^{-i(x-x_2)\xi^2}
\cdots 
q(\xi^1,\cdots,\xi^{k})
\\ & \qquad 
u_1(x_1) \bu_2(x_2) \cdots  
dx_1 \cdots dx_{k}\, d\xi^1\cdots d\xi^k,
\end{aligned}
\]
or equivalently on the Fourier side
\[
\mathcal F Q(u_1,\bu_2,\cdots)(\xi)
= (2\pi)^{-\frac{n(k-1)}2} \int_{D}
q(\xi^1,\cdots,\xi^{k})
\hat u_1(\xi^1) \bar{\hat u}_2(\xi^2) \cdots  \,
d\xi^1 \cdots d\xi^{k-1},
\]
where, with alternating signs, 
\[
D := \{ \xi = \xi^1-\xi^2 + \cdots \}.
\]

They can also be described via their kernel
\[
Q(u_1,\bu_2,\cdots)(x) =  
\int K(x-x_1,\cdots,x-x_{k})
u_1(x_1) \bu_2(x_2) \cdots  
dx_1 \cdots dx_{k},
\]
where $K$ is defined in terms of the  
Fourier transform  of $q$
\[
K(x_1,x_2,\cdots,x_{k}) = 
(2\pi)^{-\frac{kn}2} \hat q(-x_1,x_2,\cdots,(-1)^k x_{k}).
\]

These notations are 
convenient but slightly nonstandard because of the alternation of complex conjugates. Another important remark is that, for $k$-linear forms, the cases of odd $k$, respectively even $k$ play different roles here, as follows:

\medskip

i) The $2k+1$ multilinear forms will be thought of as functions, e.g. those which appear 
in some of our evolution equations.

\medskip

ii) The $2k$ multilinear forms will be thought of as densities, e.g. which appear 
in some of our density-flux pairs.

\medskip
Correspondingly,  
to each $2k$-linear form $Q$ we will associate
a $2k$-linear functional $\mathbf{Q}$ defined by 
\[
\mathbf{Q}(u_1,\cdots,u_{2k}) := \int_\R Q(u_1,\cdots,\bu_{2k})(x)\, dx,
\]
which takes real or complex values.
This may be alternatively expressed 
on the Fourier side as 
\[
\mathbf{Q}(u_1,\cdots,u_{2k}) = (2\pi)^{n(1-k)} \int_{D_0}
q(\xi^1,\cdots,\xi^{2k})
\hat u_1(\xi^1) \bar{\hat u}_2(\xi^2) \cdots  
\bar{\hat u}_{2k}(\xi^{2k})\,d\xi^1 \cdots d\xi^{2k-1},
\]
where, with alternating signs, the diagonal $D_0$ is given by
\[
D_0 = \{ 0 = \xi^1-\xi^2 + \cdots \}.
\]
Note that in order to define the multilinear functional $\mathbf{Q}$ we only need to know the symbol $q$ on $D_0$.

\section{A frequency envelope formulation of the results}
\label{s:boot}

For expository purposes, the main results of this paper, namely Theorem~\ref{t:local2}, Theorem~\ref{t:local3+}, Theorem~\ref{t:global3}, Theorem~\ref{t:global2} and Theorem~\ref{t:global2c}
are stated in a simplified form in the introduction. However, the full results that we prove provide a much more detailed picture, which gives  a full family of bilinear $L^2_{t,x}$ bounds 
for the solutions. Furthermore, the proofs of the results are complex bootstrap arguments relative to all these bounds, both linear and bilinear. For these reasons, it is important to have a good setup for both the results and for the bootstrap assumptions. An elegant way to do this is to use the 
language of frequency envelopes. Our goals in this section are 
\begin{enumerate}[label=(\roman*)]
\item to define frequency envelopes, 
\item to provide a more accurate, frequency envelope
version of our main results, 
\item  to provide the bootstrap assumptions in the proofs of each of the theorems, 
\item to outline the continuity argument allowing us to use 
these bootstrap assumptions.
\end{enumerate}

\subsection{Frequency envelopes}
Before stating one of the main
theorems of this paper, we revisit the \emph{frequency envelope}
notion. This elegant and useful tool will streamline the exposition of our results, and one should think at it as a bookkeeping device that is meant to efficiently track the evolution of the energy of the solutions between dyadic energy shells.  

Following Tao's paper \cite{Tao-BO} we say that a sequence $c_{k}\in l^2$
is an $L^2$ frequency envelope for $\phi \in L^2$ if
\begin{itemize}
\item[i)] $\sum_{k=0}^{\infty}c_k^2 < \infty$;\\
\item[ii)] it is slowly varying, 
\[
c_j /c_k \leq 2^{\delta \vert j-k\vert}, \qquad j,k \in \N
\]
with a small universal constant $\delta$;
\\
\item[iii)] it bounds the dyadic norms of $\phi$, namely $\Vert P_{k}\phi \Vert_{L^2} \leq c_k$. 
\end{itemize}
Given a frequency envelope $c_k$ we define 
\[
 c_{\leq k} := (\sum_{j \leq k} c_j^2)^\frac12, \qquad  c_{\geq k} := (\sum_{j \geq k} c_j^2)^\frac12.
\]
 In  practice we may choose our envelopes in a minimal fashion so that we have the equivalence
\[
\sum_{k=0}^{\infty}c_k^2  \approx \|u\|_{L^2}^2,
\]
and so that we also have 
\[
c_0 \approx \|u\|_{L^2}.
\]
The same applies with respect to any $H^s$ norm.

\begin{remark}\label{r:unbal-fe}
Another useful variation is to weaken the slowly varying assumption to
\[
2^{- \delta \vert j-k\vert} \leq    c_j /c_k \leq 2^{C \vert j-k\vert}, \qquad j < k,
\]
where $C$ is a fixed but possibly large constant. All the results in this paper are compatible with this choice.
This offers the extra flexibility of providing higher regularity results by the same arguments.
\end{remark}

\subsection{ The frequency envelope form of the results}

The  goal of this section is twofold: (i) restate the bounds needed for our  main results in Theorem~\ref{t:local2}, Theorem~\ref{t:local3+}, Theorem~\ref{t:global3}, 
Theorem~\ref{t:global2} and Theorem~\ref{t:global2c} in the frequency envelope setting,  and (ii) to set up the bootstrap argument for the proof of these bounds. 

The set-up for the bootstrap is most conveniently described using  the language  of frequency envelopes. This was originally introduced in the context of dyadic Littlewood-Paley decompositions in work of Tao, see for example\cite{Tao-WM}. For another interesting general frequency envelope setup one can see \cite{IT-BO}, which was the first place to use it to bootstrap bilinear bounds.
\bigskip

To start with, we assume that the initial data has small size in $H^s$,
\[
\| u_0\|_{H^s} \lesssim \epsilon.
\]
We consider a dyadic frequency decomposition for the initial  data, 
\[
u_0 = \sum_{\lambda \in 2^\N} u_{0,\lambda}.
\]
Then we place the initial data components under an admissible  frequency envelope,
\[
\|u_{0,\lambda}\|_{H^s} \leq \epsilon c_\lambda, \qquad c \in \ell^2,
\]
where the envelope $\{c_\lambda\}$ is not too large,
\begin{equation*}
\| c\|_{\ell^2} \approx 1.    
\end{equation*}

Our goal will be to establish similar frequency envelope bounds for the solution. 
We now state the frequency envelope bounds that will be the subject of the next theorems. These are as follows:

\begin{enumerate}[label=(\roman*)]
\item Uniform frequency envelope bound:
\begin{equation}\label{uk-ee}
\| u_\lambda \|_{L^\infty_t L^2_x} \lesssim \epsilon c_\lambda \lambda^{-s}
\end{equation}
\item Unbalanced bilinear $L^2_{t,x}$ bounds:
\begin{equation} \label{uab-bi-unbal}
\| u_\lambda \bu_\mu^{x_0}  \|_{L^2_{t,x}} \lesssim \epsilon^2 c_\lambda c_\mu \lambda^{-s-\frac12} \mu^{-s+\frac{n-1}2} 
, \qquad \mu \ll \lambda
\end{equation}
\item Balanced bilinear $L^2_{t,x}$ bound 
\begin{equation} \label{uab-bi-bal}
\| |D|^\frac{3-n}2 (u_\lambda \bu_\mu^{x_0})  \|_{L^2_{t,x}} \lesssim \epsilon^2 c_\lambda c_\mu \lambda^{-2s+\frac12}   (1+ \lambda |x_0|)
, \qquad \lambda \approx \mu 
\end{equation}
\item Strichartz bounds with a loss of derivatives, for any sharp Strichartz exponents $(p,q)$:
\begin{equation}\label{uk-Str}
\| u_\lambda\|_{L^p_t L^q_x} \lesssim \epsilon c_\lambda \lambda^{-s+ \frac{2}p}  . 
\end{equation}
\end{enumerate}

We remark that, as a consequence of the balanced 
bilinear bound \eqref{uab-bi-bal}, one also obtains
the linear bounds 
\begin{equation}\label{uk-l4-3d}
\| u_\lambda\|_{L^4_{t,x}} \lesssim \epsilon c_\lambda \lambda^{-s+ \frac{n-2}4}, \qquad n \geq 3  ,  
\end{equation}
respectively
\begin{equation}\label{uk-l4-2d}
\| u_\lambda\|_{L^4_t L^8_x} \lesssim \epsilon c_\lambda \lambda^{-s+ \frac14  }, \qquad n = 2  . 
\end{equation}
We also remark on a special case of the Strichartz bound
\eqref{uk-Str} which has a $1/2$ derivative loss and will play a leading role in the sequel in the proof of the global results in two dimensions:
\begin{equation}\label{uk-Str2}
\| u_\lambda\|_{L^4_{t,x}} \lesssim \epsilon c_\lambda \lambda^{-s+ \frac12}, \qquad n = 2.  
\end{equation}

Here we distinguish between the  2D  case and higher dimensions. The 2D bounds bear a closer resemblance to the 1D case studied in \cite{IT-qnls}; this reflects the fact that there is less dispersion in 2D, and in particular the $L^4_{t,x}$ Strichartz bound lies on the sharp Strichartz line and cannot be obtained  directly 
using interaction Morawetz based tools. By contrast, in three and higher dimension we are able to prove the $L^4_{t,x}$  bound directly from interaction Morawetz.

We also remark on the need to add translations to the bilinear 
$L^2_{t,x}$ estimates. This is because, unlike the linear bounds
\eqref{uk-ee} and \eqref{uk-Str} which are inherently invariant with respect to translations, bilinear estimates are not invariant 
with respect to separate translations for the two factors.
Hence, in particular, they cannot be directly transferred to 
nonlocal bilinear forms. However, if we allow translations, then
one immediate corollary of \eqref{uab-bi-bal} is that for any bilinear form $L$ with smooth and bounded symbol we have 
\begin{equation} \label{uab-bi-boot-trans}
\| |D|_x^{\frac{3-n}2} L(u_{\lambda} \bu_{\lambda})  \|_{L^2_{t,x}} \lesssim  \epsilon^2 c_{\lambda^2} \lambda^{-2s-\frac12}.
\end{equation}
This is essentially the only way we will use this translation invariance in our proofs. We also remark on the $(1+\lambda|x_0|)$ 
factor, which appears in the balanced bilinear bounds but not in the 
unbalanced ones. Heuristically this is because in the unbalanced
case transversality remains valid even when one of the metrics gets translated; on the other hand, in the balanced case we also estimate interactions of near parallel waves, and there transversality is very sensitive to changes of the metric, so there is a price to pay for translations. 
\medskip

We are now ready to state our local in time results, beginning  with the case of three and higher space dimensions:

\begin{theorem}\label{t:local-fe3}
a)  Let $n \geq 3$, $s > \frac{n+1}2$,  $\epsilon \ll 1$ and $T > 0$. Consider the equation \eqref{qnls} with cubic nonlinearity and let $u \in C[0,T;H^s]$ be a smooth solution with initial data $u_0$  which has $H^s$ size at most $\epsilon$.
 Let $\{ \epsilon  c_\lambda\}$ be a frequency envelope for the initial data  in $H^s$. Then the solution $u$ satisfies  the bounds \eqref{uk-ee}, \eqref{uab-bi-unbal}, \eqref{uab-bi-bal} and  \eqref{uk-Str}   uniformly with respect to $x_0 \in \R$.

b) The same result holds for \eqref{dqnls} but with $s > \frac{n+3}2$.
\end{theorem}

We continue with the case of two space dimensions, where we provide several variants of the result:

\begin{theorem}\label{t:local-fe2}
a) Let $n = 2$, $s > \frac32$, $\epsilon \ll 1$ and $T > 0$. Consider the equation \eqref{qnls} with cubic nonlinearity, and for which  one of the following three conditions holds:
\begin{itemize}
\item[(i)] either $s \geq 2$ and $T$ is arbitrary  large,
\item[(ii)] or $s >\frac32$ and $T \ll \epsilon^{-6}$,
\item[(iii)] or $s > \frac32$, the problem is conservative and $T$ is arbitrarily large.
\end{itemize}
Let $u \in C[0,T; H^s]$ be a smooth solution with initial data $u_0$  which has $H^s$ size at most $\epsilon$.
 Let $ \{\epsilon c_\lambda\}$ be a frequency envelope for the initial data  in $H^s$. Then the solution $u$ satisfies  the bounds \eqref{uk-ee}, \eqref{uab-bi-unbal},  \eqref{uab-bi-bal} and \eqref{uk-Str}  uniformly with respect to $x_0 \in \R$.

b) The same result holds for \eqref{dqnls} but with $s$ increased by one.

\end{theorem}

\medskip

The results in Theorem~\ref{t:local-fe3} and Theorem~\ref{t:local-fe2} [(i), (iii)] apply independently of the size of $T$, so they yield the global in time solutions in Theorems~\ref{t:global3}, \ref{t:global2},  \ref{t:global2c}, given an appropriate local well-posedness result.

 We continue with a theorem that applies to the linearized  equation, which will be essential in the proof of all of our  well-posedness results:

\begin{theorem}\label{t:linearize-fe}
Consider the cubic equation \eqref{qnls} for $n \geq 2$. 
Let $u$ be a solution as in Theorem~\ref{t:local-fe2} or Theorem~\ref{t:local-fe3}. Let $v$ be a solution 
to the linearized equation  around $u$ with $L^2$ initial data 
$v_0$ and  with frequency envelope $d_\lambda$. Then $v$ satisfies 
the following bounds:
\begin{enumerate}[label=(\roman*)]
\item Uniform frequency envelope bound:
\begin{equation}\label{uk-ee-lin}
\| v_\lambda\|_{L^\infty_t L^2_x} \lesssim \epsilon d_\lambda 
\end{equation}
\item Balanced bilinear $(v,v)$-$L^2$ bound:
\begin{equation} \label{vvab-bi-bal-lin}
\| |D|^\blue{\frac{3-n}2}(v_\lambda \bv_\mu^{x_0})  \|_{L^2_{t,x}} \lesssim d_\lambda d_\mu  \blue{\lambda^\frac12} (1+ \lambda |x_0|)
, \qquad \lambda\approx \mu 
\end{equation}
\item Unbalanced bilinear $(v,v)$-$L^2$ bound:
\begin{equation} \label{vvab-bi-unbal-lin}
\| v_\lambda \bv_\mu^{x_0}  \|_{L^2_{t,x}} \lesssim d_\lambda d_\mu \mu^\blue{\frac{n-1}2} \lambda^{-\frac12}
, \qquad \mu \ll \lambda
\end{equation}
\item Balanced bilinear $(u,v)$-$L^2$ bound:
\begin{equation} \label{uvab-bi-bal-lin}
\| |D|^\blue{\frac{3-n}2}(u_\lambda \bv_\mu^{x_0})  \|_{L^2_{t,x}} \lesssim c_\lambda d_\mu  \lambda^{-s\blue{+\frac12}}(1+ \lambda |x_0|)
, \qquad \lambda\approx \mu 
\end{equation}
\item Unbalanced bilinear $(u,v)$-$L^2$ bound:
\begin{equation} \label{uvab-bi-unbal-lin}
\| u_\lambda \bv_\mu^{x_0}  \|_{L^2_{t,x}} \lesssim c_\lambda d_\mu \lambda^{-s} \frac{\min\{\lambda,\mu\}^\blue{\frac{n-1}2}}{(\lambda+\mu)^\frac12}
, \qquad \mu \not\approx \lambda .
\end{equation}
\end{enumerate}
\end{theorem}

This theorem in particular yields $L^2$ well-posedness for the linearized equation.

\subsection{ The bootstrap hypotheses}

To prove the above theorems, we make a bootstrap assumption where we assume the same bounds but with a worse constant $C$, as follows:

\begin{enumerate}[label=(\roman*)]
\item Uniform frequency envelope bound:
\begin{equation}\label{uk-ee-boot}
\| u_\lambda \|_{L^\infty_t L^2_x} \leq C\epsilon c_\lambda \lambda^{-s}
\end{equation}
\item Unbalanced bilinear $L^2_{t,x}$ bounds:
\begin{equation} \label{uab-bi-unbal-boot}
\| u_\lambda \bu_\mu^{x_0}  \|_{L^2_{t,x}} \leq C^2 \epsilon^2 c_\lambda c_\mu \lambda^{-s-\frac12} \mu^{-s+\frac{n-1}2}  
, \qquad \mu \ll \lambda
\end{equation}
\item Balanced bilinear $L^2_{t,x}$ bound:
\begin{equation} \label{uab-bi-bal-boot}
\| |D_x|^{\frac{3-n}2} (u_\lambda \bu_\mu^{x_0})  \|_{L^2_{t,x}} \leq C^2\epsilon^2 c_\lambda c_\mu \lambda^{-2s+\frac{n-1}2} \mu^{-s} \lambda^\frac12 (1+ \lambda |x_0|)
, \qquad \lambda \approx \mu .
\end{equation}
\end{enumerate}
We note that in dimension $n\geq 3$ the estimate \eqref{uab-bi-bal-boot} implies an $L^4_{t,x}$ bound for $u_\lambda$, but that is no longer the case in dimension $n = 2$. For later use in the study of global well-posedness in two dimensions, we will add to the 
list above an $L^4_{t,x}$ bootstrap bound for $u_\lambda$, which corresponds to the estimate \eqref{uk-Str}:
\begin{enumerate}[resume]
    \item $L^4_{t,x}$ Strichartz bound for $n = 2$:
    \begin{equation}\label{uk-Str2-boot}
\| u_\lambda\|_{L^4_{t,x}} \lesssim C\epsilon c_\lambda \lambda^{-s+ \frac12}, \qquad n = 2.  
\end{equation}
\end{enumerate}

Then we seek to improve the constant in these bounds. The gain will come from the fact that the $C$'s will always come paired with extra $\epsilon$'s. Precisely, a continuity argument shows that 

\begin{proposition}\label{p:boot}
a) It suffices to prove Theorems~\ref{t:local-fe3}, \ref{t:local-fe2}(ii) under the bootstrap assumptions \eqref{uk-ee-boot}, \eqref{uab-bi-bal-boot} and
\eqref{uab-bi-unbal-boot}. 

b) It suffices to prove Theorem \ref{t:local-fe2}(i)(iii) under the bootstrap assumptions \eqref{uk-ee-boot}, \eqref{uab-bi-bal-boot}, \eqref{uab-bi-unbal-boot}, \eqref{uk-Str2-boot}.
\end{proposition}

We remark that the full Strichartz bounds \eqref{uk-Str}
are not part of this bootstrap loop. Similar bootstrap assumptions are made for $v$ in the proof of Theorem~\ref{t:linearize-fe} in Section~\ref{s:lin-proof}.

\section{A paradifferential/resonant  expansion of the equation}
\label{s:lin}

 In this section we consider the equation for the frequency localized portions 
of the solution.  We expand this equation in a dual fashion,
separating two principal components:

\begin{enumerate}
\item
The paradifferential part, which accounts for the quasilinear character of the problem. This is essential for the proofs of  all the results in this paper.

\item The doubly resonant part, which accounts for the 
nonperturbative semilinear part of the nonlinearity. This is only needed for the global results in Theorem~\ref{t:global2c} in two dimensions. 
\end{enumerate}

As a first approximation, the paradifferential part can be viewed as simply obtained by truncating the coefficients in the principal part to lower frequencies,
\begin{equation}\label{paraT}
i \partial_t v + \partial_j T_{g^{jk}(u)} \partial_k v 
= f.
\end{equation}
While this might suffice for the local well-posedness results  stated in Theorems~\ref{t:local2}  and \ref{t:local3+}, it is not precise enough for all the long time results, as the low frequencies of $g(u)$  may also  include some doubly resonant $ high\times high$ quadratic contributions. This leads us to 
make a better choice in the frequency truncation of the metric, and correspondingly in the source term $f$.  

Algebraically it will be easier to deal with  the
corresponding frequency $\lambda$ evolution, which may be taken as
\begin{equation}\label{para}
i \partial_t v_\lambda + \partial_x g_{[<\lambda]} \partial_x v_\lambda 
= f_\lambda, \qquad v_\lambda(0) = v_{0,\lambda},
\end{equation}
where the truncated metric is defined as 
\begin{equation}
g_{[<\lambda]} := P_{<\lambda} g (u_{\ll \lambda}).
\end{equation}
 It is important to note that we carefully  localize  $u$ to lower frequencies, 
rather than $g(u)$ directly; this is essential later on in order to exclude 
the doubly resonant interactions from the paradifferrential flow. On the other
hand, using the second order elliptic operator in divergence form
is more of a convenience, since  the corresponding commutator terms
play  a perturbative role.
\bigskip

One may express both the full equation and the linearized equation in terms 
of the paradifferential flow, in the form
\begin{equation}\label{para-full}
i \partial_t u_\lambda + \partial_x g_{[<\lambda]} \partial_x u_\lambda 
= N_\lambda(u), \qquad v_\lambda(0) = u_{0,\lambda},
\end{equation}
respectively
\begin{equation}\label{para-lin}
i \partial_t v_\lambda + \partial_x g_{[<\lambda]} \partial_x v_\lambda 
= N^{lin}_\lambda v, \qquad v_\lambda(0) = v_{0,\lambda},
\end{equation}

The source terms $N_\lambda(u)$ and $N^{lin}_\lambda$
 will be written explicitly later when needed. 
They are expected to play a perturbative role for the short-time results.

\bigskip

For the 2D global results there is one additional portion of the nonlinearity which plays a nonperturbative role, namely the balanced 
cubic part, which we separate from $N_\lambda$, writing instead
 the frequency localized evolution \eqref{para-full} in the form
\begin{equation}\label{full-lpara0}
i \partial_t u_\lambda + \partial_x g_{[<\lambda]} \partial_x u_{\lambda} = C_\lambda(\bfu,\bfu, \bfu) + F_\lambda(u) ,
\end{equation}
where we recall that the notation $\textbf{u}$ stands for  $\mathbf{u} =\left\{ u, \bar{u}\right\}$.
Following the discussion in Section~\ref{s:resonances}, 
if $\lambda \gg 1$ then the balanced cubic part is further decomposed into 
\[
C_\lambda(\bfu,\bfu, \bfu) = C_\lambda^{res}(u,\bu,u) + 
C_\lambda^{nr}(\bfu,\bfu,\bfu) + C_\lambda^{tr}(\bfu,\bfu,\bfu).
\]
Here the last term contains only transversal interactions
and may be harmlessly included in $F_\lambda$.
Correspondingly, we arrive at the final expansion 
\begin{equation}\label{full-lpara}
i \partial_t u_\lambda + \partial_x g_{[<\lambda]} \partial_x u_{\lambda} = C_\lambda^{res}(u,\bu,u) +  C^{nr}_\lambda(\bfu,\bfu, \bfu) + F_\lambda(u), 
\end{equation}
where the first two components on the right represent terms as follows:
\begin{itemize}
    \item $C_\lambda^{res}(u,\bu,u)$ contains balanced cubic terms
    with phase rotation symmetry, including the doubly resonant interactions
    \item  $C^{nr}_\lambda(\bfu,\bfu, \bfu)$ contains balanced nonresonant terms.
\end{itemize}
To explicitely make a choice for $C_\lambda^{res}$ we introduce a symmetric symbol $c_{diag}$, smooth on the corresponding dyadic scale
so that
\[
c_{diag}(\xi_1,\xi_2,\xi_3) := \left\{
\begin{aligned}
1 & \qquad \sum |\xi_i - \xi_j| \ll \la \xi_1\ra + \la \xi_2 \ra 
+ \la \xi_3 \ra 
\\
0 & \qquad  \sum |\xi_i - \xi_j| \gtrsim \la \xi_1\ra + \la \xi_2 \ra 
+ \la \xi_3 \ra \ .
\end{aligned}
\right.
\]
Then cubic doubly resonant part of the nonlinearity is defined in terms of the symbol $c(\cdot \, , \, \cdot \, , \, \cdot)$ of the form in \eqref{eq-cubic} as 
\begin{equation}\label{defcl}
c_\lambda(\xi_1,\xi_2,\xi_3) := p_\lambda(\xi_1-\xi_2+\xi_3) \, c_{diag}(\xi_1,\xi_2,\xi_3) c(\xi_1,\xi_2,\xi_3).
\end{equation}
In particular, one sees that the conservative condition implies 
that this symbol is real on the diagonal,
\[
\Im c_\lambda^{res}(\xi,\xi,\xi) = 0 , \qquad \xi \in \R.
\]

\medskip

Finally, we also briefly discuss $F_\lambda(u)$, which contains the remaining source terms in the paradifferential equation for $u_\lambda$, which can be classified as follows:

\begin{enumerate}[label=(\roman*)]
\item cubic $ high \times high$ terms of the form 
\[
\mu^2 P_\lambda L(\bfu_{\leq \mu}, \bu_\mu,  u_\mu) \qquad \mu \gtrsim \lambda.
\]

\item cubic $low \times high$ commutator terms of the form
\[
L(\bfu_{< \lambda}, \partial_x \bfu_{<\lambda}, \partial_x u_\lambda).
\]

\item cubic terms with only transversal interactions of the form
\[
\lambda^2 L(u_\lambda,u_\lambda,u_\lambda),
\qquad 
\lambda^2 L(u_\lambda,\bu_\lambda,\bu_\lambda),
\qquad 
\lambda^2 L(\bu_\lambda,\bu_\lambda,\bu_\lambda),
\qquad \lambda \gg 1
\]

 \item all cubic terms at frequency $\lambda = 1$,
 
\item quartic and higher terms of the form 
\[
\mu^2 L( \bfu_{\leq \mu}, \bfu_{\leq \mu}, \bu_\mu,  u_\mu) \qquad \mu \gtrsim \lambda .
\]

\end{enumerate}
In our analysis the $F_\lambda$ terms will play a perturbative role, though establishing that is not at all immediate, and  is instead part of the challenge.

\section{Interaction Morawetz estimates: the constant coefficient case} \label{s:im-cc}

Our main tool in the proof of the bilinear $L^2_{t,x}$ bounds
is provided  by the interaction Morawetz estimates, which in turn are based on density-flux identities for the mass and momentum. In this 
section we first review the constant coefficient setting, which is broadly based on \cite{PV}. However,  our choices for the fluxes are different,  customized to fit  the nonlinear setting, akin to the  work in the one dimensional case in \cite{IT-global}, \cite{IT-qnls}.

\subsection{Conservation laws in the flat case}
For clarity, we begin our discussion with the linear Schr\"odinger equation
\begin{equation}\label{eq:flat}
i u_t + \Delta u = 0, \qquad u(0) = u_0.
\end{equation}

For this we consider the following  conserved quantities, the mass
\[
\bM(u) = \int |u|^2 \,dx,
\]
and the momentum 
\[
\bP_j(u) = 2 \int \Im (\bar u \partial_j u) \,dx.
\]

To these quantities we associate corresponding densities
\[
M(u) = |u|^2, \qquad P_j(u) = i ( \bar u \partial_j u - u \partial_j \bar u).
\]

The densities here are not uniquely determined, and our choices are motivated by the conservation law computations
\begin{align}
    &\label{df-lin}
\partial_t M(u) = \partial_j P_j(u),\\
\
\
& \partial_t P_j (u) = \partial_m E_{jm}(u),
\end{align}
where we choose
\[
E_{jm}(u) :=  \partial_m u \partial_j \bar u +  \partial_j u \partial_m \bar u
- v \partial_j  \partial_m \bar v - \bar v 
\partial_j \partial_m v.
\]
The reader may compare this with \cite{PV}, where a different choice is made.

\medskip

The symbols of these densities viewed as bilinear forms  are
\[
m(\xi,\eta) = 1, \qquad p_j(\xi,\eta) = -(\xi_j+\eta_j),
\]
respectively
\[
\qquad e_{jm}(\xi,\eta) = (\xi_j+\eta_j)(\xi_m +\eta_m).
\]


\subsection{Interaction Morawetz identities in the flat case}
We define the interaction Morawetz functional as
\begin{equation}\label{Ia-sharp-def-vv-flat}
\bI(u,v) :=   \iint a_j(x-y) (M(u)(x) P_j(v) (y) -  
P_j(u)(x) M(v) (y)) \, dx dy,
\end{equation}
with the bounded weights $a_j$ to be chosen later.
Assuming that $u,v$ solve the equation \eqref{eq:flat}, the time derivative of $\bI(u,v)$ is 
\begin{equation}\label{I-vv-flat}
\frac{d}{dt} \bI(u,v) =  \bJ^4(u,v) ,
\end{equation}
where
\[
\bJ^4(u,v) = \int  \partial_m a_j(x-y) \left(M(u) E_{jm}(v) - P_j(u) P_m(v) + M(v) E_{jm}(u)
-  P_j(v) P_m(u)\right)\, dxdy.
\]
Now we consider the choice for $a_j$. To get a good sign
and symmetry for $\bJ^4$, it is convenient to take
\[
a_j (x) = \partial_j a, \qquad \partial_m a_j = a_{mj},
\]
where $a$ is a well-chosen convex function.

For the quadrilinear form $J^4_{mj}$ which is the coefficient of $a_{jm}$ in $\bJ^4$ we write the symbol
\[
\begin{aligned}
J^4_{mj} = & \ (\xi^1_m + \xi^2_m)(\xi^1_j +\xi^2_j)
+ (\xi^3_m + \xi^4_m)(\xi^3_j +\xi^4_j) - (\xi^1_j +\xi^2_j) (\xi^3_m + \xi^4_m) -  (\xi^1_m +\xi^2_m) (\xi^3_j + \xi^4_j)
\\
= & \ (\xi^1_j+\xi^2_j -\xi^3_j -\xi^4_j)(\xi^1_m+\xi^2_m -\xi^3_m -\xi^4_m)
\\
= & \ 2(\xi^1_j - \xi^4_j)(\xi^3_m - \xi^2_m) + 2(\xi^1_m - \xi^4_m)(\xi^3_j - \xi^2_j) \qquad (\text{mod } \Delta^4 \xi),
\end{aligned}
\]
where terms with $\Delta^4 \xi$ factors do not contribute to the integral $\bJ^4_{mj}$.

So for instance if $a(x) = x^2$ 
then we get 
\[
J^4(\xi^1,\xi^2,\xi^3,\xi^4)= 8 ( \xi^1-\xi^4)\cdot (\xi^3-\xi^2) \qquad (\text{mod } \Delta^4 \xi),
\]
which yields the positive form
\[
\bJ^4(u,v) = 8 \| \partial (u \bar v)\|_{L^2}^2 .
\]
Unfortunately, we cannot use this, because for this choice the $a_j$'s are not bounded. 

More generally, the above formula for the symbol of $J^4_{mj}$ directly yields
\begin{equation}\label{J4-flat}
\bJ^4(u,v) = 2 \int a_{jm}(x-y) F_j \bar F_m\,  dx dy,
\end{equation}
where
\begin{equation}\label{Fj-flat}
F_{j}(x,y) = u(x) \partial_j \bar v(y) + 
\partial_j u(x) \bar v(y).
\end{equation}
This is clearly nonnegative definite if $a$ is convex.

\medskip

The special case when $u = v$ is also interesting.
Using the identity
\[
| u(x) \partial_j \bar v(y) + 
\partial_j u(x) \bar v(y)|^2 
= | u(x) \partial_j  v(y) -
\partial_j u(x) v(y)|^2 + 
\partial_j |u(x)|^2
 \partial_j |v(y)|^2,
\]
and its bilinear version, setting $u = v$ we obtain the inequality (see also \cite{PV})
\begin{equation}\label{J4-flat-u=v}
\bJ^4(u,u) \geq \int a_{jm}(x-y) \partial_j |u(x)|^2
 \partial_m |u(y)|^2 \, dx dy.
\end{equation}

For the function $a$ we will use $a(x) = |x|$,
in which case we have the Hessian
\[
D^2 a(x) = \frac{1}{|x|} (I - \frac{x}{|x|} \otimes \frac{x}{|x|}).
\]
To complete our computations we use the identity
\begin{equation}\label{IM-clean}
    \int a_{jm}(x-y) \partial_j |u(x)|^2
 \partial_m |u(y)|^2 \, dx dy = c_n \| |D|^\frac{3-n}2 |u|^2 \|_{L^2}^2,
\end{equation}
which is easily verified by interpreting $a_{jm}$ in \eqref{J4-flat-u=v} as kernels for appropriate multipliers. This finally  leads to 
\begin{equation}\label{J4-flat-u=v+}
\bJ^4(u,u) \gtrsim  \| |D|^\frac{3-n}2 |u|^2 \|_{L^2}^2.
\end{equation}
We remark, however, that the two expressions above are not equivalent, which is harmless here but will generate considerable difficulties in the nonlinear case.

We are also interested in frequency localized versions of these objects. 
Given a dyadic frequency $\lambda$, we start with a symmetric  symbol $a_\lambda(\xi,\eta)$, in the sense  that
\[
a_\lambda(\eta,\xi) = \ol{a_\lambda(\xi,\eta)},
\]
and  localized at frequency $\lambda$. Then we define an associated weighted mass density by
\[
M_\lambda(u) := A_\lambda(u,\bar u).
\]
We also define corresponding momentum symbols $p_{j,\lambda}$  by 
\[
p_{j,\lambda}(\xi,\eta) := -(\xi_j+\eta_j)\, a_\lambda(\xi,\eta).
\]
Then a direct computation yields the density-flux relations
\[
\frac{d}{dt} M_\lambda(u,\bar u) = \partial_j P_{j,\lambda}(u,\bar u), \qquad \frac{d}{dt} P_{j,\lambda}(u,\bar u) = \partial_m E_{mj,\lambda}(u,\bar u),
\]
\blue{with 
\[
\qquad e_{jm,\lambda}(\xi,\eta) := a_\lambda(\xi,\eta) (\xi_j+\eta_j)(\xi_m +\eta_m).
\]}

\section{Density-flux relations for the paradifferential flow}
\label{s:df}

In this section  we develop the density-flux identities 
for the paradifferential and the nonlinear flows, leaving the full 
interaction Morawetz analysis for the following section.

\subsection{ Density-flux identities for the paradifferential problem}
\
Now that we have the flat case as a model, we turn our attention to the paradifferential flow \eqref{para} for which we consider  solutions $v_\lambda$. 
For these we have 
\[
\begin{aligned}
\frac{d}{dt} M(v_\lambda) = &\ 2 \Re ( \partial_t v_\lambda \cdot \bar v_\lambda) 
\\
= &\ 2 \Im (\partial_j g^{jl}_{[<\lambda]} \partial_l v_\lambda   \cdot \bar v_\lambda) 
+ 2 \Im ( f_\lambda \bar v_\lambda)
\\
= & \  2 \partial_j [g^{jl}_{[<\lambda]} \Im (\partial_l v_\lambda   \cdot \bar v_\lambda)] 
+ 2 \Im ( f_\lambda \bar v_\lambda).
\end{aligned}
\]

Defining the covariant momenta as 
\[
P^j (v_\lambda) = g^{jl}_{[<\lambda]} P_l(v_\lambda),
\]
we rewrite this in the form
\begin{equation}\label{dens-flux-param}
\frac{d}{dt} M(v_\lambda) =  \partial_j [ P^j(v_\lambda)] + F^{para}_{\lambda,m}, \qquad F^{para}_{\lambda,m} = 2 \Im ( f_\lambda \bar v_\lambda).
\end{equation}

Similarly, one computes, with $g:= g_{[<\lambda]}$ and $v:= v_\lambda$,
\begin{equation*}
\begin{aligned}
\frac{d}{dt} P_j(v_\lambda) =  & \ 
2\Re ( \partial_k g^{kl} \partial_l v \partial_j \bv)
- 2\Re ( v \partial_j \partial_k g^{kl} \partial_l \bv) - 2\Re (f \partial_j \bv)
+ 2 \Re ( v \partial_j \bar f)
\\ = & \
2 \partial_k \Re (  g^{kl} \partial_l v \partial_j \bv - v
 \partial_j g^{kl} \partial_l    \bv)
- 2\Re (  g^{kl} \partial_l v \partial_j \partial_k \bv) +  2\Re ( \partial_k  v \partial_j g^{kl} \partial_l \bv)
\\
&\ - 2\Re (f \partial_j \bv)
+ 2 \Re ( v \partial_j \bar f)
\\ = & \
2 \partial_k \Re (  g^{kl} \partial_l v \partial_j \bv -
v \partial_j  g^{kl} \partial_l   \bv)
+  2\Re ( \partial_k  v (\partial_j g^{kl}) \partial_l \bar v)
- 2\Re (f \partial_j \bv)
+ 2 \Re ( v \partial_j \bar f)
\\ = & \ \partial_k [ E^k_j(v_\lambda)] + F^{para}_{j,\lambda,p},
\end{aligned}
\end{equation*}
where
\begin{equation}
 E^k_j(v_\lambda) :=    2  \Re (  g^{kl} \partial_l v \partial_j \bv -
v \partial_j  g^{kl} \partial_l   \bv),
\end{equation}
and
\begin{equation}
    F^{para}_{j,\lambda,p} :=  2\Re ( \partial_k  v (\partial_j g^{kl}) \partial_l \bar v)
- 2\Re (f \partial_j \bv)
+ 2 \Re ( v \partial_j \bar f).
\end{equation}

Now we switch to the covariant versions, 
\[
P^j(v_\lambda) = g^{jk}_{[<\lambda]} P_k(v_\lambda),
\]
for which we similarly write
\begin{equation}\label{dens-flux-parap}
\begin{aligned}
\frac{d}{dt} P^j(v_\lambda)
= & \ \partial_k [ E^{kj}(v_\lambda)] + F^{para,j}_{\lambda,p},
\end{aligned}
\end{equation}
where for $E^{kj}$ we have simply raised indices tensorially,
\begin{equation} \label{Ekj-para}
 E^{kj}(v_\lambda) =    2  \Re (   \partial^k v_\lambda \partial^j \bv_\lambda -
v_\lambda \partial^j  \partial^k   \bv_\lambda),
\end{equation}
but for the source term we get additional commutator terms,
\begin{equation}
    F^{j,para}_{\lambda,p} =   G^{j,para}_{\lambda,p}
- 2\Re (f_\lambda \partial^j \bv_\lambda)
+ 2 \Re ( v_\lambda \partial^j \bar f_\lambda),
\end{equation}
where the quadratic term $G^{j,para}_{\lambda,p}$
is given by 
\begin{equation}\label{G4}
   G^{j,para}_{\lambda,p} =  \partial_k v_\lambda   (\partial^j g^{kl}_{[<\lambda]}) \partial_l \bar v_{\lambda} + (\partial_t g^{jk}_{[<\lambda]}) P_k(v_\lambda) - (\partial_k g^{kl}_{[<\lambda]})E^j_l(v_\lambda).
\end{equation}
Fortunately its exact form is less important, as it will be perturbatively estimated in $L^1_{x,t}$ later on. Here we simply note 
that we can schematically write it in the form
\begin{equation}
   G^{j,para}_{\lambda,p}  = h(\bfu_{<\lambda}) \bfu_{<\lambda} \partial \bfu_{<\lambda} \partial v_{\lambda}    \partial \bv_{\lambda},
\end{equation}
which is all that will be needed.
\bigskip

\subsection{Nonlinear density-flux identities for the frequency localized mass and momentum} \label{s:df-2c}

The interaction Morawetz identities \eqref{dens-flux-param}, \eqref{dens-flux-parap} for the paradifferential equation suffice for the proof of our results in dimension $n \geq 3$. However, a finer analysis is needed in the most interesting case $n=2$, and particularly for the proof of the global result in Theorem~\ref{t:global2c}. This is the aim of this subsection.
We note that the analysis below is only needed at large frequencies, where
the $L^4_{t,x}$ Strichartz bound loses too much to allow us to treat cubic terms perturbatively.

 We begin with a simpler computation for the mass density of $u_\lambda$, using the equation \eqref{full-lpara}. We have 
\[
\begin{aligned}
\frac{d}{dt} M_\lambda(u) = &\ 2 \Re ( \partial_t u \cdot \bar u) 
\\
= &\ 2 \Im (\partial_j g_{[<\lambda]}^{jk} \partial_k u_\lambda   \cdot \bar u_\lambda) 
+ 2 \Im ( (C_\lambda^{res}(u,\bu, u)+  C_\lambda^{nr}(\bfu,\bfu, \bfu) + F_\lambda(\bfu)) \bar u_\lambda)
\\
= & \  2 \partial_j [g_{[<\lambda]}^{jk} \Im (\partial_k u_\lambda   \cdot \bar u_\lambda)] 
+ 2 \Im ( (C_\lambda^{res}(u,\bu, u)+  C_\lambda^{nr}(\bfu,\bfu, \bfu) + F_\lambda(\bfu)) \bar u_\lambda),
\end{aligned}
\]
which we rewrite in the form
\begin{equation}\label{dens-flux-m0}
\frac{d}{dt} M_\lambda(u) =  \partial_j [ P_\lambda^j(u)] 
+ C^{4,res}_{\lambda,m}(u,\bu,u,\bu) + C^{4,nr}_{\lambda,m}(\bfu,\bfu,\bfu,\bfu) +  F^4_{\lambda,m}(\bfu),
\end{equation}
where 
\[
C^{4,res}_{\lambda,m} (u,\bu,u,\bu) := 2 \Im ( (C_\lambda^{res}(u,\bu, u)  \bu_\lambda), 
\qquad 
C^{4,nr}_{\lambda,m}(\bfu,\bfu,\bfu,\bfu) := 2 \Im ( C_\lambda^{nr}(\bfu,\bfu, \bfu)\bu_\lambda)),
\]
\[
 F^4_{\lambda,m} (\bfu) := 2 \Im ( (F_\lambda(\bfu)  \bu_\lambda).
\]
The three source terms in \eqref{dens-flux-m0} will be treated differently in the last section.

The term $C^{4,res}_{\lambda,m}$ contains only interactions  which are localized at frequency $\lambda$, but include the parallel interactions where all frequencies are equal. Under the conservative assumption in Definition~\ref{d:conservative2}, we obtain the key symbol property 
\begin{equation}\label{c4-cancel}
   c^4_{\lambda,m}(\xi,\xi,\xi,\xi) = 0 .
\end{equation}
One should compare this with the 1D case in  \cite{IT-global},
\cite{IT-qnls}, where a stronger conservative condition
guaranteed a second order vanishing for $c^4_{\lambda,m}$ on the diagonal.
This in turn allowed us to remove this term via a quartic correction to the mass density, and a quartic correction to the mass flux, modulo perturbative errors. In the 2D case there are 
too many resonant interactions in order to allow for such a strategy. Instead, we will interpret the above vanishing condition as allowing us to get a symbol representation
\[
c^{4,res}_{\lambda,m}(\xi^1,\xi^2,\xi^3,\xi^4) = i(\xi_{even}-\xi_{odd}) \cdot r^{4}_{\lambda,m}(\xi^1,\xi^2,\xi^3,\xi^4) + i\Delta^4 \xi \cdot q^{4,res}_{\lambda,m}(\xi^1,\xi^2,\xi^3,\xi^4)
\]
where the symbols $r^{4}_{\lambda,m}$ and $q^{4,res}_{\lambda,m}$ are localized at frequency $\lambda$, have size $\lambda$ and are smooth on the $\lambda$ scale. 
At the multilinear form level this roughly yields a representation of the form
\[
C^{4,res}_{\lambda,m} (u,\bu,u,\bu) \approx \lambda L(\partial |u_\lambda|^2, u_\lambda, \bu_\lambda)
+ \partial Q^{4,res}_{\lambda,m}(u,\bu,u,\bu),
\]
which in turn will be used in order to treat it perturbatively, by separating it into an $L^1_{t,x}$ component and a flux correction.

The term $C^{4,nr}_{\lambda,m}$ contains a larger range of interactions  which are also localized at frequency $\lambda$, and are nonresonant. This yields a symbol representation
\[
c^{4,nr}_{\lambda,m}(\xi^1,\xi^2,\xi^3,\xi^4) = i \Delta^4 \xi^2 \cdot b^{4,nr}_{\lambda,m}(\xi^1,\xi^2,\xi^3,\xi^4) + i \Delta^4 \xi \cdot q^{4,nr}_{\lambda,m}(\xi^1,\xi^2,\xi^3,\xi^4)
\]
where the symbols $b^{4,nr}_{\lambda,m}$ and $q^{4,nr}_{\lambda,m}$ are localized at frequency $\lambda$, have size $1$, respectively $\lambda$ and are smooth on the $\lambda$ scale. 

This will allow us to remove it using both density and flux corrections
\[
C^{4,nr}_{\lambda,m} (\bfu,\bfu,\bfu,\bfu) \approx 
\partial_t B^{4,nr}_{\lambda,m}(\bfu,\bfu,\bfu,\bfu)
+  \partial_x Q^{4,nr}_{\lambda,m}(\bfu,\bfu,\bfu,\bfu) + F^{6,nr}_{\lambda,m}(\bfu)
\]
with a perturbative source term $F^{6,nr}_{\lambda,m}(\bfu)$ containing six-linear and higher order
terms arising from the nonlinear terms in the equation in the expansion of the time derivative of $B^{4,nr}_{\lambda,m}$.

Finally, $F^4_{\lambda,m}$ is at least quartic and only involves 
double transversal interactions, and will be estimated in $L^1_{t,x}$ and treated perturbatively.

\blue{Summarizing the above discussion, we arrive at a modified density-flux relation for the 
localized mass of the form
\begin{equation}\label{dens-flux-m0+}
\begin{aligned}
\frac{d}{dt} ( M_\lambda(u) - B^{4,nr}_{\lambda,m}(\bfu,\bfu,\bfu,\bfu))  = & \   \partial_j [ P_\lambda^j(u) +Q^{4,nr,j}_{\lambda,m}(\bfu,\bfu,\bfu,\bfu) + Q^{4,res,j}_{\lambda,m}(\bfu,\bfu,\bfu,\bfu)] 
\\ & \ 
+ \lambda L(\partial |u_\lambda|^2, u_\lambda, \bu_\lambda) +  F^4_{\lambda,m}(\bfu) + F^{6,nr}_{\lambda,m}(\bfu).
\end{aligned}
\end{equation}

Similarly, we have a corresponding  
 density-flux relation for the momentum
\begin{equation}\label{dens-flux-p}
\frac{d}{dt} P^j_{\lambda}(u) =  \partial_k  E^{jk}_{\lambda}(u) 
+ C^{j,4,res}_{\lambda,p}(u,\bu,u,\bu)  +  C^{4,nr}_{\lambda,p}(\bfu,\bfu,\bfu,\bfu)+
 F^{j,4}_{\lambda,p}(\bfu)
\end{equation}
with a cancellation relation akin to \eqref{c4-cancel} and a similar set of density and flux corrections for the middle terms on the right. This yields a modified density-flux relation for the localized momentum which is similar to \eqref{dens-flux-m0+},
\begin{equation}\label{dens-flux-p0+}
\begin{aligned}
\frac{d}{dt} ( P^j_\lambda(u) \!- \!B^{4,nr}_{\lambda,p}(\bfu,\bfu,\bfu,\bfu))  = & \   \partial_k [ E^{kj}_\lambda(u) +Q^{4,nr,kj}_{\lambda,p}(\bfu,\bfu,\bfu,\bfu) + Q^{4,res,kj}_{\lambda,p}(\bfu,\bfu,\bfu,\bfu)] 
\\ & \ 
+ \lambda^2 L^j(\partial |u_\lambda|^2, u_\lambda, \bu_\lambda) +  F^{4,j}_{\lambda,p}(\bfu) + F^{6,nr,j}_{\lambda,p}(\bfu).
\end{aligned}
\end{equation}
 The modified density flux relations \eqref{dens-flux-m0+} and \eqref{dens-flux-p0+}
 will play a leading role in the study of the two dimensional problem at low regularity
 in the last section of the paper.}


\section{Interaction Morawetz bounds for the linear paradifferential flow}
\label{s:para}

In this section,  as well as in  the next one, we study the linear paradifferential equation
associated to a solution $u$ to the quasilinear Schr\"odinger flow \eqref{qnls}. Here we prove bilinear $L^2_{t,x}$ bounds and $L^2$ well-posedness.
On the one hand, this is a key step in  the proof of both the local well-posedness results in Theorem~\ref{t:local2}, Theorem~\ref{t:local3+}
and the global results in Theorem~\ref{t:global3}, Theorem~\ref{t:global2}. On the other hand,
an enhanced version of these arguments will be used later 
in order to prove estimates for the full equation, and, in particular, to prove our final global result in Theorem~\ref{t:global2c}. In the next section we prove Strichartz estimates for the paradifferential equation.

To start with, we consider a one parameter family of frequency localized functions $v_\lambda$ that solve the inhomogeneous linear paradifferential equations \eqref{para}.  For now we assume no connection between these functions. Later we will apply these bounds in the case when  $v_\lambda=P_\lambda u$ or $v_\lambda = P_\lambda v$ where $v$ solves the linearized equation.

For these functions we have the conservation laws
\begin{equation}\label{df-v-m}
\partial_t M(v_\lambda) = \partial_k  P^k(v_\lambda)+ F^{para}_{\lambda,m},
\end{equation}
respectively
\begin{equation}\label{df-v-p}
\partial_t P^j(v_\lambda) = \partial_k  E^{kj} (v_\lambda) + F^{j,para}_{p,\lambda},
\end{equation}
where the source terms are 
\begin{equation}\label{f-v-m}
F^4_{\lambda,m} := 2 \Im M(f_\lambda,\bv_\lambda)
=
2\Im ( f_\lambda \bv_\lambda) ,  
\end{equation}
respectively 
\begin{equation}\label{f-v-p}
F^{j,para}_{p,\lambda}:=  G^{j,para}_{p,\lambda}
- 2\Re (f_\lambda \partial^j \bv_\lambda)
+ 2 \Re ( v_\lambda \partial^j \bar f_\lambda)
\end{equation}
with $G^{j,para}_{p,\lambda}$ given by \eqref{G4}.

We will measure each $v_\lambda$ based on the size of the initial data
and of the source term $f_\lambda$. However, following an idea introduced in \cite{IT-qnls}, rather than measure $f_\lambda$ 
directly, we will instead measure its interaction with $v_\lambda$ and its translates. So we denote
\begin{equation}\label{dl}
d_\lambda^2 := \sup_{\mu,\nu \approx \lambda} \| v_\mu(0)\|_{L^2}^2 +  \sup_{x_0 \in \R} \| v_\mu f_\nu^{x_0} \|_{L^1_{t,x}}.
\end{equation}
Using the more implicit parameter $d_\lambda$ allows for a larger range of source terms, which is absolutely essential in the present work.

Our main result here asserts that we can obtain energy and bilinear $L^2$ bounds for $v_\lambda$:

\begin{theorem}\label{t:para}
Assume that $u$ solves \eqref{qnls} in a time interval $[0,T]$,
and satisfies the bounds \eqref{uk-ee},\eqref{uab-bi-unbal}
and \eqref{uab-bi-bal} in the same time interval for some $s > \frac{n+1}2$. If $n=2$, we also assume that $T \lesssim \epsilon^{-6}$. Assume $v_\lambda$ solve \eqref{para}, and let $d_\lambda$ be as in \eqref{dl}. Then the following bounds hold for the functions $v_\lambda$ uniformly in $x_0 \in \R$:

\begin{enumerate}
    \item Uniform energy bounds:
\begin{equation}\label{v-ee}
\| v_\lambda \|_{L^\infty_t L^2_x} \lesssim d_\lambda 
\end{equation}

\item Balanced bilinear $(u,v)$-$L^2$ bound:
\begin{equation} \label{uv-bi-bal}
\| |D|^\frac{3-n}2 (v_\lambda \bu_\mu^{x_0})  \|_{L^2_{t,x}} \lesssim \epsilon d_{\lambda} c_\mu \lambda^{-s+\frac12} (1+ \lambda |x_0|),
\qquad \mu \approx \lambda
\end{equation}

 \item Unbalanced bilinear $(u,v)$-$L^2$ bounds:
\begin{equation} \label{uv-bi-unbal}
\| v_\lambda \bu_\mu^{x_0}  \|_{L^2_{t,x}} \lesssim \epsilon d_{\lambda} c_\mu \mu^{-s} \frac{\min\{\lambda,\mu\}^\frac{n-1}2}{(\lambda+\mu)^{\frac12}} 
 \qquad \mu \not \approx \lambda 
\end{equation}

\item Balanced bilinear $(v,v)$-$L^2$ bound:
\begin{equation} \label{vv-bi-bal}
\| |D|^\frac{3-n}2 (v_\lambda \bv_\mu^{x_0})  \|_{L^2_{t,x}} \lesssim  d_{\lambda} d_\mu \lambda^\frac12 (1+ \lambda |x_0|),
\qquad \mu \approx \lambda
\end{equation}

 \item Unbalanced bilinear $(v,v)$-$L^2$ bound:
\begin{equation} \label{vv-bi-unbal}
\| v_\lambda \bv_\mu^{x_0}  \|_{L^2_{t,x}} \lesssim  d_{\lambda} d_\mu\frac{\min\{\lambda,\mu\}^\frac{n-1}2}{(\lambda+\mu)^{\frac12}} 
 \qquad \mu \not \approx \lambda .
\end{equation}
\end{enumerate}
\end{theorem}

Before proceeding to the proof, we remark that the unbalanced bounds
are more robust, and do not require any relation between the metric in the $v_\lambda$ equation and the one in the $v_\mu$ equation:

\begin{corollary}\label{c:para}
The unbalanced bounds \eqref{uv-bi-unbal} and \eqref{vv-bi-unbal}
are still valid if the equation \eqref{para} for $v_\lambda$ 
is replaced by the linear Schr\"odinger equation   
\[
(i \partial_t+\Delta) v_\lambda = f_\lambda.
\]
\end{corollary}

This should be seen as a corollary of the proof of the theorem, and not of the theorem itself. \blue{ Indeed, the unbalanced case of the proof of the theorem applies in this case without any change.  The above corollary applies in particular if $v_\lambda$ is an $L^2$
solution to the homogeneous Schro\"dinger equation. That leads us to a follow-up
corollary:

\begin{corollary}\label{c:para1}
The unbalanced bounds \eqref{uv-bi-unbal} and \eqref{vv-bi-unbal}
are still valid if $v_\lambda \in U^2_\Delta$, with $d_\lambda = \|v_\lambda \|_{U^2_\Delta}$.
\end{corollary}

Here it suffices to consider the case when $v_\lambda$ is an $U^2_\Delta$ atom.
But the $U^2_\Delta$ atoms are concatenations of $L^2$ solution to the homogeneous Schro\"dinger equation with $\ell^2$ summation, so the bilinear $L^2_{t,x}$ bounds reduce to the case  $v_\lambda$ is an $L^2$ solution to the homogeneous Schro\"dinger equation,
which is included in Corollary~\ref{c:para}.
}

\bigskip

We now return to the proof of the theorem.

\begin{proof}[Proof of Theorem~\ref{t:para}]
 In order to treat the
dyadic bilinear $(u,v)$ and $(v,v)$ bounds at the same time,
we write the equation for $u_\lambda$ in the form 
\eqref{para-full} where we estimate favourably the source term $N_\lambda$, and prove 
the following estimate:

\begin{proposition}\label{p:N-lambda}
 a) Let $n \geq 3$ and $s > \frac{n+1}{2}$. Assume that the function $u$ satisfies   the bounds \eqref{uk-ee}, \eqref{uab-bi-unbal}  
 and \eqref{uab-bi-bal}
 in a time interval $[0,T]$. Then for $\epsilon$ small enough, the functions $N_\lambda(u)$ in \eqref{para-full} satisfy 
 \begin{equation}\label{good-nl}
\| N_\lambda(\bfu) \bu_\lambda^{x_0}\|_{L^1_{t,x}} \lesssim \epsilon^4 c_\lambda^2 \lambda^{-2s}.
 \end{equation}

 b) Let $n = 2$ and $s > \frac{n+1}{2}$. Then \eqref{good-nl}
 holds under the additional assumption $T \leq \epsilon^{-6}$.
 Furthermore, for any $T$ we have the partial bound
 \begin{equation}\label{good-fl}
\| F_\lambda(\bfu) \bu_\lambda^{x_0}\|_{L^1_{t,x}} \lesssim \epsilon^4 c_\lambda^2 \lambda^{-2s}.
 \end{equation}
\end{proposition}

\begin{proof}  We start the proof without discriminating between the spatial dimensions $n=2$ and $n\geq 3$; we will do so when the estimates require it.  The bound in \eqref{good-nl} relies on a careful analysis of the frequency localized source term $N_{\lambda}(\bfu)$.
The key features of the multilinear terms contained in $N_\lambda(u)$ are :
\begin{itemize}
\item  they are at least cubic
\item they have exactly two derivatives
\item the highest frequency is at least $\lambda$.
\item the two highest frequencies have to be comparable when 
either
\begin{itemize}
    \item the two derivatives 
apply to the highest frequency or 
\item  the highest frequency is $\gg \lambda$.
\end{itemize}
\end{itemize}
This leads to a large number of cases which are also informed by the comments in Section 4 which discuss a potential partition of $N_{\lambda}$, and which we  next organize  efficiently. 
For this, recall that 
\[
N_{\lambda}(\bfu) =  C_\lambda(\bfu,\bfu, \bfu) + F_\lambda(\bfu),
\]
where the entries of $C_{\lambda}$ are localized around frequency $\lambda$ (but not strictly localized), and where $F_{\lambda}$ contain not only the remaining cubic source terms but also higher order terms. The cubic terms in $F_{\lambda}$ are either unbalanced, or balanced transversal, and do not necessarily have the phase rotation symmetry property.  We begin discussing all cases, one at a time:

\medskip

\noindent \textbf{A.} The highest input frequency is $\mu \gg \lambda$. In this case we must have at least two frequency $\mu$ factors. We separate this case further.

\noindent A1.  Terms which have at least one lower frequency $\mu_1 \ll \mu$. For brevity we consider a typical term of the form 
\[
N_\lambda^{A1} =  \bfu_{\mu_1} \partial^2_x \bfu_\mu^2 ,
\]
which can be expressed as a multilinear expression in the form
\[
N_\lambda^{A1} =  \mu^2 L(\bfu_{\mu_1}, \bfu_\mu, \bfu_{\mu}).
\]
Here the estimate goes as follows
\[
\| N_{\lambda}^{A1} \bar u_\lambda^{x_0} \|_{L^1_{t,x}} 
\lesssim \mu^2 \|L( \bfu_{\mu},  \bfu_\lambda^{x_0}) \|_{L^2_{t,x}} 
\| L(\bfu_{\mu_1}, \bfu_\mu) \|_{L^2_{t,x}} \lesssim 
\epsilon^4 c_\lambda c_{\mu_1} c_\mu^2 \,  \mu^{1-2s} \lambda^{-s+\frac{n-1}{2}} 
\mu_1^{-s+\frac{n-1}{2}}.
\]
Here we have arranged the factors in order to obtain two unbalanced pairs, to which we then applied the unbalanced bilinear $L^2_{t,x}$ bounds in \eqref{uab-bi-unbal}. The summation with respect to $\mu_1$ and $\mu$ is straightforward. One last observation is regarding the complex conjugates which do not matter in obtaining this bound.

\medskip

\noindent A2. Terms with exactly three $\mu$ frequencies. Here the  three $\mu$ frequencies must add up to a frequency $\ll \mu$, so at least two must be $\mu$ separated, and the above argument still applies.

\medskip

\noindent A3. Terms with at least four $\mu$ frequencies, e.g.
\[
N_\lambda^{A3} =  \bfu_\mu^3 \partial_{\mu}^2 \bfu_\mu ,
\]
which we can re-express as 
\[
N_\lambda^{A3} =  \mu^2 L( \bfu_\mu,  \bfu_\mu,\bfu_\mu,\bfu_\mu).
\]
To get the desired estimate we group the two unbalanced terms and use \eqref{uab-bi-unbal}, 
\[
\begin{aligned}
\| N_{\lambda}^{A3} \bar u_\lambda^{x_0} \|_{L^1_{t,x}} 
&\lesssim \mu^2\| L(\bfu_{\mu} , u_\lambda^{x_0}) \|_{L^2_{t,x}}  \| L(\bfu_{\mu} , \bfu_\mu, \bfu_{\mu}) \|_{L^2_{t,x}}\\
&\lesssim \epsilon^2c_{\lambda}c_{\mu} \,\mu^{-s+\frac{3}{2}}\lambda^{-s+\frac{n-1}{2}} \| L(\bfu_{\mu} , \bfu_\mu, \bfu_{\mu}) \|_{L^2_{t,x}}.
\end{aligned}
\]
Then for the remaining factor we need to discuss separate ways of obtaining the bounds when $n=2$, respectively $n\geq 3$. We begin with the seemingly easier case $n\geq 3$, and 
bound the remaining cubic terms using the $L^4_{t,x} $ Strichartz bound \eqref{uk-l4-3d}  for two copies of $u_{\mu}$ and Bernstein's inequality  for the last $u_{\mu}$
\[
\| L(\bfu_{\mu} , \bfu_\mu, \bfu_{\mu}) \|_{L^2_{t,x}} \lesssim \Vert u_\mu\Vert_{L^4_{t,x}}^2\Vert u_{\mu}\Vert_{L^{\infty}_{t,x}}\lesssim \epsilon^2 c_{\mu}^2\,  \mu^{-s+\frac{n-2}{2}} \mu^{\frac{n}{2}}\Vert u_{\mu}\Vert_{L^\infty_tL^2_x}\lesssim \epsilon^3 c_{\mu}^3\mu^{-2s+n-1}.
\]
We summarize the bound we obtained in  the $n\geq 3$ case below
\[
\| N_{\lambda}^{A3} \bar u_\lambda^{x_0} \|_{L^1_{t,x}} 
\lesssim \epsilon^5 c_{\lambda}c_{\mu}^4 \, \mu^{-3s+n+\frac{1}{2}} \lambda^{-s+\frac{n-1}{2}},
\]
which suffices.

It remains to bound the same expression when  $n=2$:
\[
\| N_{\lambda}^{A3} \bar u_\lambda^{x_0} \|_{L^1_{t,x}} 
\lesssim \mu^2\| L(\bfu_{\mu} , u_\lambda^{x_0}) \|_{L^2_{t,x}}  \| L(\bfu_{\mu} , \bfu_\mu, \bfu_{\mu}) \|_{L^2_{t,x}}\lesssim \epsilon^2c_{\lambda}c_{\mu} \,\mu^{-s+\frac{3}{2}}\lambda^{-s+\frac{1}{2}} \| L(\bfu_{\mu} , \bfu_\mu, \bfu_{\mu}) \|_{L^2_{t,x}}. 
\]
For the last trilinear term we rely on the Strichartz bound \eqref{uk-l4-2d} as well on the Bernstein's inequality
\[
\| L(\bfu_{\mu} , \bfu_\mu, \bfu_{\mu}) \|_{L^2_{t,x}}\lesssim \Vert u_{\mu}\Vert_{L^4_tL^8_x}^2 \Vert u_{\mu}\Vert_{L^{\infty}_tL^{4}_x}\lesssim  \epsilon^3c_{\mu}^3\, \mu^{-3s+1}.
\]

Hence, when we are in spatial dimension $n=2$ we get
\[
\| N_{\lambda}^{A3} \bar u_\lambda^{x_0} \|_{L^1_{t,x}} 
\lesssim \mu^2\| L(\bfu_{\mu} , u_\lambda^{x_0}) \|_{L^2_{t,x}}  \| L(\bfu_{\mu} , \bfu_\mu, \bfu_{\mu}) \|_{L^2_{t,x}}\lesssim 
\epsilon^5c_{\lambda}c_{\mu}^4 \,\mu^{-4s+2}\lambda^{-s+\frac{1}{2}},
\]
which still suffices.

\bigskip

\noindent \textbf{B}. The highest frequency is comparable to $\lambda$. We also 
split this case  further.

\medskip

\noindent B1. Terms in $F_{\lambda}(\bfu)$ which have exactly one frequency $\lambda$ factor. Such terms arise only from commuting the 
metric $g_{[<\lambda]}$ with a derivative or a frequency $\lambda$ 
projector (see the similar computation in \cite{IT-qnls}),  
so we have at most one derivative on the frequency $\lambda$ factor. Thus we are led to consider expressions of the form
\[
N_\lambda^{B1} =  \partial_x ( \bfu_{\lambda_1} 
\bfu_{ \lambda_2}) \partial_x u_{\lambda} + \partial_x  \bfu_{\lambda_1} 
\partial_x \bfu_{\lambda_2} u_{ \lambda} , \quad  \lambda_2 \leq \lambda_1 \ll \lambda,
\]
which will be useful to write as
\[
N_\lambda^{B1} = \lambda \lambda_1 L (\bfu_{\lambda}, \bfu_{\lambda_1} ,\bfu_{ \lambda_2}) +\lambda_1 \lambda_2 L (\bfu_{\lambda}, \bfu_{\lambda_1} ,\bfu_{ \lambda_2}) .
\]

We again take advantage of the unbalanced bilinear Strichartz bound \eqref{uab-bi-unbal}
\[
\begin{aligned}
\| N_{\lambda}^{B1} \bar u_\lambda^{x_0} \|_{L^1_{t,x}} 
&\lesssim (\lambda_1\lambda +\lambda_1\lambda_2)\, \| u_{\lambda_1}  u_\lambda^{x_0} \|_{L^2_{t,x}} 
\| u_{\lambda_2} u_\lambda\|_{L^2_{t,x}} \\
&\lesssim (\lambda_1\lambda +\lambda_1\lambda_2)\, \epsilon^4 c_{\lambda}^2 c_{\lambda_1}c_{\lambda_2}\lambda^{-2s-1}\lambda_1^{-s+\frac{n-1}{2}} \lambda_2^{-s+\frac{n-1}{2}}\\
&\lesssim  \epsilon^4 c_{\lambda}^2 c_{\lambda_1}c_{\lambda_2}\lambda^{-2s}\lambda_1^{1-s+\frac{n-1}{2}} \lambda_2^{-s+\frac{n-1}{2}}.
\end{aligned}
\]

\medskip

\noindent B2. We have exactly two frequency $\lambda$ factors. Here the two $\lambda$ frequencies must add up to $O(\lambda)$ so  they must be $\lambda$ separated, and the above argument still applies.

\medskip

 \noindent B3.  We have at least three factors at frequency $\lambda$.
This is the most difficult case, which includes in particular the cubic balanced terms in 2D. We first dispense with everything else.

\medskip

 \noindent B3a. $n \geq 3$. Here we can simply use four $L^4$ bounds:
\[
\begin{aligned}
\| N_{\lambda}^{B3} \bar u_\lambda^{x_0} \|_{L^1_{t,x}} 
&\lesssim \lambda^2 \Vert u_{\lambda}  \Vert_{L^4_{t,x}}^4 \lesssim \epsilon^4c_{\lambda}^4\lambda^{-4s+n}.
\end{aligned}
\]

 \medskip

\noindent B3b. $n=2$. Here we separate into several cases:

\medskip

\noindent B3b(i).
We have at least four factors at frequency $\lambda$. This includes quartic and higher order terms in $F_{\lambda}(\bfu)$. Such terms are given, for example, by
\[
N_{\lambda}^{B3} =   \partial_x^2  u_\lambda  \,  \bfu_\lambda^3 .
\]
 We reinterpret this term as a multilinear form
\[
N_{\lambda}^{B3} =\lambda^2 L( \bfu_\lambda, \bfu_{\lambda},  \bfu_\lambda, \bfu_{\lambda}).
\]
Since we are in the case $n=2$, this leads to
\[
\begin{aligned}
\| N_{\lambda}^{B3} \bar u_\lambda^{x_0} \|_{L^1_{t,x}} 
&\lesssim \lambda^2    \Vert u_{\lambda}\Vert _{L^4_tL^8_x}^4 \Vert u_{\lambda}\Vert_{L^{\infty}_tL^2_x}\\
&\lesssim \epsilon^5 c_{\lambda}^5\lambda^{-5s+3} .
\end{aligned}
\]

\medskip

\noindent B3b(ii). We have quartic or higher order source terms with  exactly three factors at frequency $\lambda$ and at least one smaller frequency, e.g. 
\[
N_{\lambda}^{B3}=\lambda^2 L(\bfu_\lambda,\bfu_\lambda, \bfu_\lambda, \bfu_{\lambda_1}) , \quad \lambda_1<\lambda .
\]
We bound its contribution as follows
\[
\begin{aligned}
\| N_{\lambda}^{B3} \bar u_\lambda^{x_0} \|_{L^1_{t,x}} 
&\lesssim \lambda^2 \Vert u_\lambda u_{\lambda_1}\Vert_{L^2_{t,x}}\Vert u_{\lambda}\Vert_{L^4_tL^8_x}^2 \Vert u_{\lambda}\Vert_{L^{\infty}_tL^4_x} \\
&\lesssim \lambda^2 \epsilon^5c_{\lambda}^4c_{\lambda_{1}} \lambda^{-s-\frac{1}{2}}\lambda_{1}^{-s+\frac{1}{2}}
\lambda^{-2s+\frac{1}{2}}\lambda^{-s+\frac{1}{2}}
\\
&\lesssim  \epsilon^5c_{\lambda}^4c_{\lambda_{1}} \lambda^{-4s+\frac{5}{2}} \lambda_{1}^{-s+\frac{1}{2}}.
\end{aligned}
\]
\medskip

\noindent B3b(iii). We have exactly a cubic balanced form at frequency $\lambda$ which further satisfies the transversality condition.
Here we simply group the four entries into two groups with 
$\lambda$-frequency separated pairs, and use two bilinear $L^2_{t,x}$ bounds.

\noindent B3b(iv).
We have exactly three balanced terms at frequency $\lambda$,
possibly resonant. In this case we know in addition that $T<
\epsilon^{-6}$. We take
\[
N_{\lambda}^{B3} =\lambda^2L(\bfu_{\lambda}, \bfu_{\lambda}, \bfu_{\lambda})
\]
and its bound is given by
\[
\begin{aligned}
\| N_{\lambda}^{B3} \bar u_\lambda^{x_0} \|_{L^1_{t,x}} \lesssim \lambda^2\Vert u_{\lambda}\Vert^4_{L^4_{x,t}}\lesssim \lambda^2 T^{\frac{1}{3}}\Vert u_{\lambda}\Vert^4_{L^6_tL^4_x}\lesssim T^{\frac{1}{3}}\epsilon^4c_{\lambda}^4\lambda^{-4s+\frac{8}{3}}.
\end{aligned}
\]

\end{proof}

Given  the estimate \eqref{good-nl}, the $(u,v)$ bilinear bounds and the $(v,v)$ bilinear bounds are absolutely identical.
Hence in what follows we will just consider
the $(v,v)$ bounds. To prove the theorem it is convenient to make the following bootstrap assumptions: 

\begin{enumerate}
    \item Uniform energy bounds:
\begin{equation}\label{v-ee-boot}
\| v_\lambda \|_{L^\infty_t L^2_x} \leq C d_\lambda 
\end{equation}

 \item Unbalanced bilinear $(u,v)$-$L^2$ bound:
\begin{equation} \label{uv-bi-unbal-boot}
\| v_\lambda \bu_\mu^{x_0}  \|_{L^2_{t,x}} \leq C \epsilon d_{\lambda} c_\mu \mu^{-s+\frac{n-1}2} \lambda^{-\frac12} 
 \qquad \mu < \lambda 
 \end{equation}

\item Balanced bilinear $(v,v)$-$L^2$ bound:
\begin{equation} \label{vv-bi-bal-boot}
\| |D_x|^{\frac{3-n}2} (v_\lambda \bv_\mu^{x_0})  \|_{L^2_{t,x}} \leq   C^2 d_{\lambda} d_\mu  \lambda^{\frac12} (1+ \lambda |x_0|),
\qquad \mu \approx \lambda
\end{equation}

 \item Unbalanced bilinear $(v,v)$-$L^2$ bound:
\begin{equation} \label{vv-bi-unbal-boot}
\| v_\lambda \bv_\mu^{x_0}  \|_{L^2_{t,x}} \leq C^2 \epsilon d_{\lambda} d_\mu \mu^{\frac{n-1}2} \lambda^{-\frac12} 
 \qquad \mu < \lambda .
\end{equation}
\end{enumerate}
These are assumed to hold with a large universal constant $C$, and then will be proved to hold without it. The way $C$ will be handled is by always pairing it with $\epsilon$ factors in the estimates,
so that it can be eliminated simply by 
assuming that $\epsilon$ is sufficiently 
small.

Here we note that to prove \eqref{v-ee},
\eqref{uv-bi-bal} and \eqref{uv-bi-unbal} it suffices 
to work with a fixed $v_\lambda$, while for the bilinear
$v$ bounds we need to work with exactly two $v_\lambda$'s. Since the functions $v_\lambda$ are completely independent, without any restriction in generality we could assume that $d_\lambda = 1$ in what follows.

We first observe that the bound \eqref{v-ee} follows 
directly via a energy estimate, which corresponds to integrating \eqref{df-v-p}. Next we use the bootstrap assumptions to estimate the sources $F^{para}$ in the density flux energy identities for $v_\lambda$ in $L^1_{t,x}$,
\begin{lemma}\label{l:Fmp}
Assume that \eqref{uv-bi-unbal-boot} holds. Then we have
\begin{equation}\label{Fmp}
\| F^{para}_{\lambda,m}\|_{L^1_{t,x}} \lesssim  C^2 \epsilon^2 d_\lambda^2,
\qquad 
\|  F^{para}_{\lambda,p}\|_{L^1_{t,x}} \lesssim  C^2 \epsilon^2 \lambda  d_\lambda^2.
\end{equation}
\end{lemma}

\begin{proof} The only nontrivial part is the 
first term in $F^{para}_{\lambda,p}$, namely $G^{para}_{\lambda,p}$ given by the formula \eqref{G4},
which schematically has the form 
\[
\partial g_{[<\lambda]} \partial v_\lambda \partial \bv_\lambda,
\]
where we can use two bilinear $L^2$ bounds. Precisely, since $g$
is at least quadratic in $u$, it suffices to show that 
\begin{equation}
\| \partial u_{\ll\lambda} v_\lambda \|_{L^2_{t,x}}
\lesssim C \epsilon d_{\lambda} 
\lambda^{-\frac12},
\end{equation}
as well as its easier version
\begin{equation}
\| u_{\ll\lambda} v_\lambda \|_{L^2_{t,x}}
\lesssim C \epsilon d_{\lambda} 
\lambda^{-\frac12},
\end{equation}
where the $\epsilon$ factor serves to absorb the constant $C$.
But these are easily obtained from our bootstrap assumption \eqref{uv-bi-unbal-boot} after dyadic summation.

\end{proof}

Now we use interaction Morawetz identities to prove the bilinear bounds in the theorem. For each pair $v_\lambda, v_\mu$ we define the interaction Morawetz functional 
\begin{equation}\label{Ia-sharp-def-vv}
\bI^{x_0}(v_\lambda,v_\mu) :=
 \iint a_j(x-y) (M(v_\lambda)(x) P^j(v^{x_0}_\mu) (y) -  
P^j(v_\lambda)(x) M(v^{x_0}_\mu) (y)) \, dx dy,
\end{equation}
with $a_j(x) = \partial_j |x|$.
The time derivative of $\bI^{x_0}(v_\lambda,v_\mu)$ is
\begin{equation}\label{I-vv}
\frac{d}{dt} \bI^{x_0}(v_\lambda,v_\mu) =  \bJ^4(v_\lambda,v_\mu^{x_0}) +  \bK(v_\lambda,v_\mu^{x_0}), 
\end{equation}
where
\begin{equation}\label{J4}
\begin{aligned}
\bJ^4(v_\lambda,v_\mu^{x_0}) := \iint & a_{jk}(x-y) 
\left(M(v_\lambda) E^{jk}(v_\mu^{x_0}) - P^j (v_\lambda)  P^k(v_\mu^{x_0}) \right.
\\ & \left. + 
M(v_\lambda) E^{jk}(v_\mu^{x_0})
-  P^j(v_\lambda)  P^k(v_\mu^{x_0})\right)\, dxdy,
\end{aligned}
\end{equation}
respectively
\begin{equation}\label{K8-def-ab-vv}
\begin{aligned}
\bK (v_\lambda,v_\mu^{x_0}):= \iint a_j(x-y) & \ M(v_\lambda)(x)   F^{4,x_0}_{j,\lambda,p}(y) + P_j(v_\lambda)(y)  F^{4,x_0}_{\lambda,m}(x) 
\\ & \!\!\!\!\!\! - 
M(v_\mu^{x_0})(y)  F^{4}_{j,\lambda,p}
(x)  - P_j(v_\mu^{x_0})(x)  F^4_{\lambda,m}(y)\,
dx dy.
\end{aligned}
\end{equation}

We will use the (integrated) interaction Morawetz identity \eqref{I-vv} to produce  $L^2_{t,x}$ bilinear bounds as follows:
\begin{itemize}
    \item The spacetime term $\bJ^4$ contains at leading order the squared $L^2$ norm we aim to bound.
    \item The fixed time expression $\bI^{x_0}$ gives the primary bound for 
    $\bJ^4$.
    \item The space-time term $\bK$ will be estimated perturbatively. 
\end{itemize}
We successively consider each of these three
terms.

\bigskip

\textbf{I. The bound for $\bI^{x_0}(v_\lambda,v_\mu)$.}
For this we observe that
\begin{equation}\label{I-bound}
 |    \bI^{x_0}(v_\lambda,v_\mu) | \lesssim (\lambda+\mu) \, d_\lambda^2 d_\mu^2,
\end{equation}
which is a direct consequence of the 
straightforward fixed time density 
bounds
\begin{equation}\label{mpe-l1}
\| M(v_\lambda)\|_{L^1_{t,x}} \lesssim d_\lambda^2, \qquad 
\| P^j(v_\lambda)\|_{L^1_{t,x}}\lesssim \lambda d_\lambda^2, \qquad 
\| E^{jk}(v_\lambda)\|_{L^1_{t,x}}\lesssim \lambda^2 d_\lambda^2.
\end{equation}
These in turn follow from the uniform $L^2$ bound
from $v_\lambda$ and its frequency localization; the source terms $f_\lambda$
play no role.

\bigskip

\textbf{II. The bound for $\bK$.} This has the form 
\begin{equation}\label{K-bound}
 |    \int_0^T \bK(v_\lambda,v_\mu)\, dt | \lesssim (\lambda+\mu) \,d_\lambda^2 d_\mu^2,
\end{equation}
and is an immediate consequence of the bounds \eqref{Fmp} and \eqref{mpe-l1}.
\bigskip

\textbf{III. The contribution of $\bJ^4$.}
Here we begin with an algebraic computation
for $\bJ^4$, after which we specialize 
to the balanced, respectively the unbalanced case.

The expression $\bJ^4$ can be further written as 
\begin{equation}
\begin{aligned}    
\bJ^4(v_\lambda,v_\mu^{x_0}) = & \iint  a_{jk}(x-y) 
\left(|v_\lambda|^2  E^{jk}(v_\mu^{x_0}) 
+ |u_\mu^{x_0}|^2 E^{jk}(v_\lambda)
- 8 \Im(\bar v_\lambda  \partial^j v_\lambda) \Im(\bar v_\mu^{x_0}  \partial^j v_\mu^{x_0}) 
\right)\, dx dy,
\end{aligned}
\end{equation}
where
\[
\begin{aligned}
E^{jk}(v_\lambda) = & \   \partial^j v_\lambda 
\partial^k \bv_\lambda + \partial^k v_\lambda 
\partial^j \bv_\lambda -
v \partial^j   \partial^k   \bv -
\bv \partial^j   \partial^k   v 
\\
= & \   2\partial^j v_\lambda 
\partial^k \bv_\lambda + 2\partial^k v_\lambda 
\partial^j \bv_\lambda -
 \partial^j   (v \partial^k   \bv) -
 \partial^j   (\bv \partial^k   v) .
\end{aligned}
\]
At this point we integrate by parts the contributions of the last two terms in $E^{jk}$, noting that  we have two types of error terms, arising from 
\begin{itemize}
    \item the fact that $\partial^j$ is not skew-adjoint, and thus we have derivatives applied to the metric

    \item the fact that the metric $g$ in $\partial^j$ may be $g_{[<\lambda]}(x)$ or $g_{[<\mu]}^{x_0}$, depending on whether this operator is applied to $v_\lambda$ 
    or to $v_\mu^{x_0}$.
\end{itemize}
We obtain
\begin{equation}
\label{J4-expansion}
\begin{aligned}    
\bJ^4(v_\lambda,v_\mu^{x_0}) = & \ 4\iint  a_{jk}(x-y) 
\left(|v_\lambda|^2  \Re(\partial^j v_\mu^{x_0} 
\partial^k \bv_\mu^{x_0})
+ |u_\mu^{x_0}|^2  \Re(\partial^j v_\lambda 
\partial^k \bv_\lambda) \right.
\\
& \qquad \qquad \qquad \qquad \left.- 2 \Im(\bar v_\lambda  \partial^j v_\lambda) \Im(\bar v_\mu^{x_0}  \partial^j v_\mu^{x_0}) 
\right)\, dx dy
\\
& \ +  4\iint  a_{jk}(x-y) 
\left(\partial^j |v_\lambda|^2 
\partial^k |\bv_\mu^{x_0}|^2
+ \partial^k |u_\mu^{x_0}|^2  \partial^j |v_\lambda|^2 
\right)\, dx dy
\\
 & \ - \iint a_{jk}(x-y) |v_\lambda|^2 (g^{lj}_{[<\mu]} (y+x_0) - g^{lj}_{[<\lambda]}(x))
\partial_l \partial^j |v_\mu^{x_0}|^2 \, dx dy
\\
& \ + 2 \iint a_{jk}(x-y)(\partial_l g^{lj}_{[<\lambda]}(x)) \left(
|v_\lambda|^2 \partial^k |v_\mu^{x_0}|^2 
 + |u_\mu^{x_0}|^2  \partial^k |v_\lambda|^2  \right) \, dx dy
\\
:= & \ \bJ^4_{main} + \bR^4_1 + \bR^4_2.
\end{aligned}
\end{equation}
Here the leading term $J^4_{main}$ 
contains the contribution of the first two lines, and the last two lines represent 
the error terms $\bR^4_1$, respectively $\bR^4_2$. Just as in the constant coefficient case, $\bJ^4_{main}$ can be rewritten as
\begin{equation}\label{J4main}
\bJ^4_{main} = 2 \iint a_{jk}(x-y) F^j \bar F^k \, dxdy,
\end{equation}
where
\begin{equation}
F^j = \partial^j v_\lambda \bar v_\mu^{x_0} + v_\lambda \partial^j v_{\mu}^{x_0}    .
\end{equation}
At this point we separate the analysis 
into the balanced and the unbalanced case.
\medskip

\emph{III A. The balanced symmetric case, 
$\lambda = \mu$, $x_0= 0$.}
We use the interaction Morawetz identity 
\eqref{I-vv} with $a(x) = |x|$.
Then we will show that
\begin{lemma}\label{l:J4-bal}
Under our bootstrap assumptions on $u$ and our bilinear bootstrap assumptions \eqref{v-ee-boot}-\eqref{uv-bi-unbal-boot} on $v_\lambda$ we have 
\begin{equation}\label{eq:l2A}
\int_0^T \bJ^4(v_\lambda,v_\lambda)\, dt \gtrsim \| |D_x|^{\frac{3-n}2} |v_\lambda|^2\|_{L^2}^2
+ O(C^2\epsilon^2 \lambda d_\lambda^4 ).
\end{equation}
\end{lemma}
Combined with \eqref{I-bound} and \eqref{K-bound}, this implies 
the bound \eqref{vv-bi-bal} in this case.

\begin{proof}
We begin with the leading term $\bJ^{4}_{main}$, 
for which we have the algebraic inequality
\begin{equation}\label{J4-cov}
\bJ^4_{main}(v_\lambda,v_\lambda) \geq  2 \iint a_{jk}(x-y) \partial^j |v_\lambda(x)|^2 \partial^k |v_\lambda(y)|^2 \, dxdy.
\end{equation}
Here the derivatives are covariant, so we cannot directly use the identity \eqref{IM-clean}. We would like to replace $g$ by the identity matrix,
\[
\partial^j |v_\lambda(x)|^2 = \partial_j |v_\lambda(x)|^2
+ \partial_m [(g_{[<\lambda]}^{jm}-I) |v_\lambda(x)|^2] - 
(\partial_m g_{[<\lambda]}^{jm}) |v_\lambda(x)|^2.
\]
By \eqref{IM-clean}, the first term on the right applied to both entries in the right hand side of \eqref{J4-cov} yields exactly the desired expression on the right in \eqref{eq:l2A}. 

The contributions of the third term on the right,
involving a derivative of the metric, may be included in 
$R^4_2$ and will be estimated perturbatively there. It remains to 
consider the contributions of the second term on the right, which, after integrating by parts and excluding again the terms with a differentiated metric, have the form
\[
\begin{aligned}
 &  \iint a_{jk}(x-y) \partial ((g_{[<\lambda]}-I) |v_\lambda(x)|^2) \partial(g |v_\lambda(y)|^2) \, dxdydt  
  \\
& =  \iint \partial^2 a_{jk}(x-y) ((g_{[<\lambda]}-I) |v_\lambda(x)|^2)  (g_{[<\lambda]} |v_\lambda(y)|^2) \, dxdydt.
\end{aligned}
\]
In this expression $\partial^2 a_{jk}$ are smooth homogeneous distributions 
of order $-3$, so they act like $|D|^{3-n}$. Hence the above integral is estimated by 
\[
\lesssim \| (g_{[<\lambda]}-I) |v_\lambda(x)|^2\|_{L^2_t \dot H^\frac{3-n}2_x}  \| g_{[<\lambda]} |v_\lambda(y)|^2\|_{L^2_t \dot H^\frac{3-n}2_x}.
\]
Here we know that $\|g_{[<\lambda]} - I\|_{L^\infty_t H^s_x} \lesssim C^2 \epsilon^2$ so we can use a fixed time multiplicative bound of the form
\begin{equation}\label{Hs-bi}
\| (g_{[<\lambda]} -I)w\|_{\dot H^{\frac{3-n}2}} \lesssim \| (g_{[<\lambda]} -I)\|_{H^s} \|w\|_{\dot H^{\frac{3-n}2}},
\end{equation}
in order to estimate the above expression by 
\[
\lesssim C^2 \epsilon^2 \| |D_x|^{\frac{3-n}2} |v_\lambda|^2\|_{L^2_x}^2.
\]
But this can be absorbed into the leading term on the right in \eqref{eq:l2A}.
Hence we arrive at the bound \eqref{eq:l2A} for $\bJ^{4}_{main}$, modulo additional $\bR^4_2$ type terms. It remains to estimate the 
expressions $\bR^4_1$ and $\bR^4_2$ perturbatively.
\medskip

To bound $\bR^4_2$ we can start by estimating one of the factors by using two bilinear $L^2$
bounds to obtain
\blue{
\[
\|\partial g_{[<\lambda]} |v_{\lambda}|^2\|_{L^1_{t,x}} \lesssim 
 \| L(\partial \bfu_{\ll \lambda}, v_\lambda)\|_{L^2_{t,x}} \| L(\bfu_{\ll \lambda}, \bv_\lambda)\|_{L^2_{t,x}} 
\lesssim
C^2 \epsilon^2 d_{\lambda}^2\lambda^{-1}.
\]
}
This allows us to bound the entire $\bR^4_2$, using also the energy estimate, as follows
\begin{equation}\label{R42}
\left| \int_0^T \bR^4_2\, dt\right| \lesssim \lambda^2 \| |v_\lambda|^2 \|_{L^\infty_t L^1_x} \|\partial g_{[<\lambda]} |v_{\lambda}|^2\|_{L^1_{t,x}}  \lesssim C^2 \epsilon^2 d_{\lambda}^2\lambda,
\end{equation}
\blue{where the derivative on $|v_\lambda|^2$ yields a factor of $\lambda$, and one of the derivatives on $a$ is also integrated by parts and placed on $|v_\lambda|^2$, yielding the second $\lambda$ factor. This is exactly as needed.}
\medskip

Finally, we turn our attention to $\bR^4_1$, where we seek to prove a similar bound
\begin{equation}\label{R41}
\left| \int_0^T \bR^4_1\, dt\right|  \lesssim C^2 \epsilon^2 d_{\lambda}^2\lambda.
\end{equation}
We recall that the integrand has the form
\[
\bR^{4}_1 = a_{jk}(x-y) \lambda^2  (g_{[<\lambda]}(x) - g_{[<\lambda]}(y)) |v_\lambda(x)|^2 |v_\lambda(y)|^2.
\]
\blue{Here we begin with a paradifferential decomposition of $g_{[< \lambda]}$,
based on the highest entry frequency\footnote{This should be done using 
the continuous Littlewood-Paley decomposition \eqref{lp-cont}, but we write it schematically like this in order to avoid cluttering the notation.}
$\mu \ll \lambda $,
\[
g_{[< \lambda]}(u)- I_n = P_{<\lambda} g(u_{\ll \lambda}) = P_{<\lambda} \sum_{\mu \ll \lambda} g^1(u_{\leq \mu}) \, u_{\mu}: =
P_{<\lambda}\sum_{\mu \ll \lambda} g_{[\mu]},
\]
with $g^1$ smooth and vanishing at $u=0$. Then we estimate separately the contribution of each summand. The projector $P_{<\lambda}$ 
can be harmlessly discarded here.}
For each $\mu$, we split the kernel $a_{jk}$ on the dual scale, depending on the relative size of $x-y$ and $\mu^{-1}$.
For large $x-y$ we treat the two terms in the $g_{[\mu]}$ difference as separate, and estimate an expression of the form 
\[
 I_{>\mu} := \iint 1_{|x-y| >\mu^{-1}} a_{jk}(x-y) |v_\lambda(x)|^2 g^{lj}_{[\mu]}(x)
\partial_l \partial^j |v_\lambda(y)|^2 \, dx dy dt.
\]
Here we can use two bilinear $L^2_{t,x}$ bounds to write
\[
\| |v_\lambda(x)|^2 g^{lj}_{[\mu]}(x)\|_{L^1_{t,x}} \lesssim 
 C^2 \epsilon^2  d_\lambda^2\lambda^{-1}\mu^{\frac{n-1}2-s}
\]
and combine this with the energy bound for the second $v_\lambda$ 
factor and the uniform bound
\[
|1_{|x-y| >\mu^{-1}} a_{jk}(x-y)| \lesssim \mu,
\]
to arrive at the final bound 
\[
|I_{>\mu}| \lesssim C^2 \epsilon^2   \lambda d_\lambda^2  \mu^{\frac{n+1}2-s} ,
\]
which has a trivial $\mu$ summation for our choice of $s$.

\medskip

For small $x-y$ we need to exploit the difference instead, but 
the delicate matter is that we still want to be able to use bilinear $L^2$ estimates. To achieve this we use the following 

\begin{lemma}
There exists a  (matrix valued) kernel $K(w,z)$ in $B(0,1) \times \R^n$, uniformly smooth in $w$ and uniformly Schwartz in  $z$, so that for any function $u$ in $\R^n$ localized at frequency $\leq 1$ we have the representation 
\begin{equation}\label{eq:translations}
u(x) - u(y) = (x-y) \int K(x-y,z) \nabla u(x+z)\, dz , \qquad |x-y| \leq 1.   
\end{equation}
\end{lemma}    
To apply this to functions which are localized at frequency $\leq \mu$ we simply rescale and use the kernel
\[
K_\mu(x,z) = \mu^{n} K(\mu z, \mu z).
\]
Then we need the constraint $|x-y| \lesssim \mu^{-1}$.
\begin{proof}
We expand the difference on the left using a smooth projector 
$P_{\leq 1}$ to frequencies $\leq 1$ and its associated Schwartz kernel $K_{\leq 1}$ as
\[
\begin{aligned}
u(x)-u(y) = & \  \int_0^1 (x-y) \cdot \nabla u(x+ h(y-x)) \, dh
\\
& \ = \int_0^1 (x-y) \cdot \nabla P_{\leq 1} u(x+ h(y-x)) \, dh
\\ & \ = \int_0^1 \int (x-y)  \cdot   \nabla u(x+ h(y-x)-x_1) K_{\leq 1}(x_1) \, dh dx_1.
\end{aligned}
\]
Changing variables to $h(x-y)-x_1 = z$ we arrive at the desired formula  \eqref{eq:translations} with 
\[
K(x-y,z) = \int_0^1 K_{\leq 1} (h(y-x) - z) \, dh,
\]
which is uniformly smooth in $x-y$ and uniformly Schwartz in $z$
for $x-y$ inside the unit ball.
\end{proof}

Now we split the difference 
\[
g_{[\mu]}(y) - g_{[\mu]}(x) =  g^1(u_{\leq \mu}(y)) (u_{\mu}(y)-u_{\mu}(x)) +  (g^1(u_{\leq \mu})(y) - g^1(u_{\leq \mu})(x))   u_{\mu}(x),
\]
and apply the above lemma to the corresponding difference. 
For the first term above, for instance, we need to estimate the integral
\[
 \iint 1_{|x-y| < \mu^{-1}} a_{jk}(x-y) (x-y) K_{\mu}(x-y,z)  
 g^1(u_{\leq \mu}(y)) \nabla u_\mu(y+z)
  |v_\lambda(x)|^2 
\partial_l \partial^j |v_\lambda(y)|^2 \, dx dy dt dh,
\]
where we can use two bilinear $L^2_{t,x}$ estimates in $y$ for the pairs 
$(u_{\leq \mu},v_\lambda)$ and $(u_\mu,v_\lambda)$, an energy bound in $x$ and a trivial $z$ integration to arrive at the same bound as in \eqref{R41}, concluding 
the proof of the Lemma. 
\end{proof}

\bigskip

\emph{III B. The balanced non-symmetric case, 
$\mu \approx \lambda$, $x_0 \in \R$.} Here  we do not repeat the computation in the previous, symmetric case, but rather use the result there.  Precisely, we apply the symmetric diagonal bound 
to the functions 
\[
w_\lambda = a v_\lambda + b v_\mu^{x_0}, \qquad |a|, |b| \leq 1.
\]
To be able to do this, we need to estimate the source term 
in the paradifferential $w_\lambda$ equation
\[
(i \partial_t + \partial_{j} g_{[<\lambda]}^{jk} \partial_k) w_\lambda 
= g_\lambda,
\]
where 
\[
g_\lambda = a f_\lambda + b f_\mu^{x_0} + 
\partial_{j} (g_{[<\mu]}^{x_0} - g_{{[<\lambda]}}^{jk} \partial_k) v_\mu^{x_0}.
\]
For $g_\lambda$ we claim that
\begin{equation}\label{glambda}
\| g_\lambda w_\lambda^{y_0}\|_{L^1_{t,x}} \lesssim C^2 \epsilon^2 d_\lambda^2 (1+\lambda |x_0|)
\end{equation}
\blue{The first two terms in $g_\lambda$ are estimated exactly as before in terms of $d_\lambda$, so it remains to consider the third.} There we split the $g$ difference as
\[
g_{[<\mu]}^{x_0} - g_{[<\lambda]} = (g_{[<\mu]}^{x_0} - g_{[<\mu]})
+ (g_{[<\mu]} - g_{[<\lambda]}).
\]
The contribution of the second difference 
requires the bound
\[
\| \partial (g_{[<\mu]} - g_{[<\lambda]}) \partial v_\mu^{x_0} w_\lambda^{y_0}\|_{L^1_{t,x}} \lesssim C^2 \epsilon^2 d_\lambda^2, 
\]
which is obtained (with an extra $C^2 \epsilon^2$ factor) by applying twice the bootstrap assumption \eqref{uv-bi-unbal-boot}, since the $g$ difference is at least quadratic and contains at least one entry at frequency smaller but comparable to $\lambda$, as $\mu \approx \lambda$.

The contribution of the first difference 
requires the bound
\[
\| \partial (g_{[<\mu]} - g_{[<\mu]}^{x_0}) \partial v_\mu^{x_0} w_\lambda^{y_0}\|_{L^1_{t,x}} \lesssim C^2 \epsilon^2 \lambda |x_0| d_\lambda^2, 
\]
which is estimated similarly with two bilinear $L^2_{t,x}$ bounds after expanding
\[
g_{[<\mu]} - g_{[<\mu]}^{x_0} = \int_{0}^1 x_0 \cdot \nabla g^{hx_0}_{<\mu} \, dh.
\]

Applying the symmetric version of \eqref{vv-bi-bal}
to $w_\lambda$ we obtain the estimate 
\[
\| |D|^{\frac{3-n}2} \Re ( a \bar b v_\lambda v_\mu^{x_0} )\|_{L^2_{t,x}}^2 \lesssim C^2 \epsilon^2 \lambda d_\lambda^2 \, (1+\lambda |x_0|),
\]
from which the full bound \eqref{vv-bi-bal} follows by making appropriate choices for $a$ and $b$.

\bigskip

\emph{III C.  The unbalanced case $\mu \ll \lambda$.} 
The argument here is similar, but 
simpler because we will be able to directly treat $(g-I)$ perturbatively.

The advantage we have is that 
the dyadic annuli where $v_\mu$ and $v_\lambda$ are supported 
are separated. Here we are allowed to freely localize $v_\lambda$ 
with respect to angular directions by applying suitable multipliers:

\begin{lemma} 
Let $A(D)$ be a smooth, bounded multiplier at frequency $\lambda$.
If $v_\lambda$ satisfies \eqref{dl}, then so does $A(D) v_\lambda$.
\end{lemma} 

\begin{proof}
The source term in the paradifferential equation for $A(D)v_\lambda$
is 
\[
g_{[\lambda]} = A(D) f_\lambda + \partial_j [A(D),g^{jk}_{[<\lambda]}] \partial_k v_\lambda.
\]
The $d_\lambda$ bound corresponding to the first term on the right is trivial, while the second term can be written in the form 
\[
\lambda L( \partial g_{[<\lambda]}, v_\lambda)  .
\]
Then its contribution to $d_\lambda$ can also be easily estimated using two bilinear unbalanced $L^2$ bounds,
\[
\|\lambda L( \partial g_{[<\lambda]}, v_\lambda) v_\lambda^{x_0}\|_{L^1_{t,x}} \lesssim \lambda \| L(\partial \bfu_{<\lambda} v_\lambda)\|_{L^2_{t,x}}
\| L(\bfu_{<\lambda} v_\lambda)\|_{L^2_{t,x}} \lesssim C^2 \epsilon^2 d_\lambda^2,
\]
as needed.
\end{proof}

Frequency localizing $v_\lambda$ to smaller sets, we may assume that $v_\mu$
and $v_\lambda$ are $O(\lambda)$ uniformly separated in frequency in a well chosen open set $\Theta$ of angular directions.  We still use the weight $a(x)= |x|$ in our interaction Morawetz identity.

Here we start with $\bJ^4$ as in \eqref{J4}, and seek to show that 
\begin{equation}\label{J4-unbal}
\int_{0}^T \bJ^4(v_\lambda,v_\mu^{x_0})\, dt \gtrsim 
\lambda^2 \mu^{1-n}\|v_\lambda \bar v_\mu^{x_0}\|^2_{L^2_{t,x}} 
+ O(C^2 \epsilon^2 \lambda \,d_\lambda^2 d_\mu^2).
\end{equation}
Combined with \eqref{I-bound} and \eqref{K-bound}, this would 
imply the unbalanced bilinear bound \eqref{vv-bi-unbal}.

For this we use the earlier computation in \eqref{J4-expansion},
and begin by estimating perturbatively the terms $\bR^4_2$ 
and $\bR^4_1$. The bound for $\bR^4_2$ is identical to the earlier 
one in the proof of Lemma~\ref{l:J4-bal}, and is omitted.
However, the $\bR^4_1$  bound is more difficult now, because we want our estimate to be uniform in $x_0$ and so we have no control over
the difference of the metrics, other than the naive comparison with the  identity matrix. We will argue differently  depending on the size of $x-y$ \blue{relative to $\mu^{-1}$, using a smooth partition of unity. }

\medskip
a) If $x-y$ is small,
$|x-y| \lesssim \mu^{-1}$, then we rewrite the integral as
\[
\iint \chi_{|z| < \mu^{-1}} a_{jk}(z) |v_\lambda(x)|^2 (g^{lj}_{[<\mu]} (x+z+x_0) - g^{lj}_{[<\lambda]}(x))
\partial_l \partial^j |v_\mu(x+x_0+z)|^2 \, dx dz.
\]
Here we bound the metric difference uniformly by $C^2 \epsilon^2$, and 
use twice the bound \eqref{uab-bi-unbal-boot} for fixed $z$. Finally the $z$ integral yields another $\mu^{1-n}$ factor. So the 
above integral is estimated by
\[
\lesssim C^2 \epsilon^2 \mu^2 \lambda^{-1} \mu^{n-1} \mu^{1-n} d_\lambda^2 d_\mu^2=
C^2 \epsilon^2 \mu^2 \lambda^{-1} d_\lambda^2 d_\mu^2 ,
\]
which suffices.

\medskip
a) If $x-y$ is large, $|x-y| \gtrsim \mu^{-1}$, then we integrate 
by parts to move the two derivatives on $a$. If one derivative falls on the metric then we get an $\bR^4_2$ type term which is good. 
Else, we obtain an integral of the form
\[
\iint \partial^2(\chi_{ > \mu^{-1}} a_{jk})(x-y) |v_\lambda(x)|^2 (g^{lj}_{[<\mu]} (y+x_0) - g^{lj}_{[<\lambda]}(x))
g_{<\mu} (y+x_0) |v_\mu(y+x_0)|^2 \, dx dy.
\]
We separate the $(g-I)$ factors and regroup with either $|v_\lambda|^2$
or $|v_\mu|^2$, \blue{ using for them the uniform in time $C^2\epsilon^2$  bound in $H^s$. }
The Fourier transform of $\partial^2 (\chi_{ > \mu^{-1}} a_{jk})$
is bounded by $\lesssim |\xi|^{3-n}$, \blue{so using multiplicative bounds as in \eqref{Hs-bi}} we estimate the above integral by 
\[
\lesssim C^2 \epsilon^2 \| |D|^{\frac{3-n}2} |v_\lambda|^2 \|_{L^2_{t,x}} \| |D|^{\frac{3-n}2} |v_\mu|^2 \|_{L^2_{t,x}} ,
\]
where the $C^2 \epsilon^2$ factor comes from $(g - I)$. 
This we can estimate using \eqref{vv-bi-bal-boot}  twice.

We are now left with the leading term $J^4_{main}$, where 
we have the chief advantage that the integrand is nonnegative
pointwise. This allows us to localize it on the $\mu^{-1}$ scale
around the diagonal,  using a nonnegative Schwartz cutoff $\phi(\mu x)$ to be chosen later:
\[
\bJ^4_{main} \gtrsim  \iint \phi(\mu(x-y))  a_{jk}(x-y) F^j \bar F^k\, d\theta .
\]
At this point we can perturbatively replace the metric $g$ 
with the identity in $F^j$, with the same argument as in the case of $\bR^4_1$, arriving at 
\[
\bJ^4_{main} \gtrsim  \bJ^4_{flat} + O(C^2 \epsilon^2 d_\lambda^2 \, d_\mu^2\,), 
\qquad
\bJ^4_{flat}:= \iint \phi(\mu(x-y))  a_{jk}(x-y) F_j \bar F_k \, d\theta.
\]
It remains to show that we have the constant coefficient bound
\begin{equation}\label{J4-bound}
 \int_0^T  \bJ^4_{flat}\, dt \gtrsim \lambda^2 \mu^{1-n} \| v_\lambda \bv_\mu^{x_0}\|_{L^2_{t,x}}^2 .
\end{equation}

To simplify the analysis we decompose our weight $a$ as 
\[
a(x) = |x| = c_n \int_{\S^n} a^\theta(x) d\theta, \qquad a^\theta(x) = |x \cdot \theta|. 
\]
This leads to a corresponding nonnegative foliation of the Hessian
\[
D^2 a = c_n \int_{\S^n} D^2 a^\theta(x)\, d\theta,
\]
and further
\[
\bJ^4_{flat} = \int_{\S^n} \iint \phi(\mu(x-y))  a^\theta_{jk}(x-y) F_j \bar F_k\,  dx dy d\theta := \int_{\S^n} \bJ^4_\theta \, d \theta .
\]
It suffices now to prove that the bound \eqref{J4-bound} holds for 
$\bJ^4_\theta$ for an open set of directions $\theta$. Here we choose exactly the set of directions in which $v_\mu$ and $v_\lambda$ are separated in frequency. Without any loss of generality assume that $\theta = e_1$. \blue{Then, splitting coordinates $x = (x_1,x')$,  we are looking at}
\[
\bJ^4_{e_1}(v_\lambda,v_\mu^{x_0}) = \int_{x_1=y_1} \phi(\mu(x'-y'))|F_1|^2 \,dx dy', 
\]
where
\[
F_1(x,y) :=  \partial_1 v_\lambda(x) \bar v_\mu^{x_0}(y)+ v_\lambda \partial_1 \bar v_\mu^{x_0}(y).
\]
 The last form of the right hand side allows us 
to interpret $F$ as a bilinear form for a fixed translation parameter $z=x-y$, and the cutoff $\phi(\mu \cdot)$ allows us to estimate $z'$ 
integrals.

Given our angular frequency localizations, within the support of $v_\lambda$ we have $\xi_1 \approx \lambda$, while in the support of $v_\mu$ we have $|\xi_1| \lesssim \mu$. Hence the first term in $F_1$ is the leading one, while the second is smaller. Precisely, using the bootstrap bilinear $L^2_{t,x}$ bound we can write 
\blue{
\[
\begin{aligned}
\int \bJ^4_{e_1}(v_\lambda,v_\mu^{x_0}) dt = & \ 
\int \!\!\int_{z_1=0} \phi(\mu z') |F_1(x,x+z)|^2 \,dx dz'dt 
\\
= & \ \!\! \int\! \!\int_{z_1=0} \phi(\mu z') |\partial_1 v_\lambda(x)|^2 |v_\mu^{x_0}(x+z)|^2 
\, dx dz' dt + O( C^2 \mu d_\lambda^2 d_\mu^2).
\end{aligned}
\]
}
For a well-chosen bump function $\phi$, the frequency localization of $v_\mu$ allows us to collapse the integral on the right to
\[
\mu^{1-n} \int |\partial_1 v_\lambda(x)|^2 |u_\mu^{x_0}(x)|^2 \, dx dt
\gtrsim \mu^{1-n} \| \partial_1 (v_\lambda \bv_\mu^{x_0}) \|_{L^2_{t,x}}^2,
\]
where the $x_1$ derivative yields a $\lambda$ factor due to the $\lambda$ frequency separation between $v_\lambda$ and $v_\mu$ in the $x_1$ direction, and where the derivative 
applied to $v_\mu$ yields a smaller, perturbative  $\mu$ factor. 
This implies \eqref{J4-bound} and thus \eqref{J4-unbal}, concluding the proof of the theorem in the final unbalanced case.

\end{proof}


\section{Low regularity Strichartz bounds for the paradifferential flow}
\label{s:Str}

While the bilinear $L^2_{t,x}$ bounds play the leading role in estimating unbalanced interactions, the balanced interactions are instead controlled using Strichartz bounds. Since here we are considering cubic problems, the $L^4_{t,x}$ bound is the single most important one.
The way we handle the $L^4_{t,x}$ bounds depends somewhat on the dimension, where we distinguish two cases:

\begin{enumerate}
    \item high dimension, $n \geq 3$. There, on the one hand, $L^4_{t,x}$  is not a sharp Strichartz norm, and on the other hand the $L^4_{t,x}$ bound is obtained directly from the interaction Morawetz analysis, so this section is not playing any role in obtaining it.

    \item low dimension, $n = 2$. Here $L^4_{t,x}$  is a sharp Strichartz norm, but the $L^4_{t,x}$ bound cannot be obtained directly from the interaction Morawetz analysis. This is why this section is needed.
\end{enumerate}

Further elaborating on the case $n= 2$, we can make a further distinction in this case depending on whether we seek local or global  in time estimates: 

\begin{enumerate}[label=(2\roman*)]
    \item $n=2$, local in time bounds (needed for local well-posedness). The interaction Morawetz analysis yields a lossless
    $L^4_t L^8_x$ bound, which interpolated with the $L^\infty_t L^2_x$ energy estimate gives a lossless $L^6_t L^4_x$ bound. Applying Holder's inequality in time yields an $L^4_{t,x}$ bound with a $1/6$
    derivative loss, which suffices for the local result.

    \item $n=2$, global in time bounds (generically needed for global well-posedness). This is the single most interesting case, and also, the one case where the analysis here is 
    absolutely necessary, providing an $L^4_{t,x}$ bound with a $1/2$ derivative loss.
\end{enumerate}

Comparing the above discussion with our prior work \cite{IT-qnls} 
in one space dimension, we note that there the $L^6_{t,x}$ bound was the main one. For the local in time analysis, we proved a sharp $L^6_{t,x}$ bound, which by interpolation with energy bounds and Holder's inequality gave an $L^4_{t,x}$ bound with a $1/4$ derivative loss, which sufficed. For the global in time analysis on the other hand, we used the interaction Morawetz  analysis to prove a global $L^6_{t,x}$ bound with  a loss of derivatives; this in turn required the conservative and defocusing assumption, but avoided any need for an $L^4_{t,x}$ bound.

Instead, in the $n = 2$ case, no defocusing assumption is necessary. The conservative assumption does play a role though, in reducing the Sobolev regularity at which the global result holds;
this would have otherwise required an $L^4_{t,x}$ bound with a smaller $1/4$ derivative loss.

The result in this section provides a full range of Strichartz estimates in all dimensions, globally in time, but with a loss of derivatives: precisely one full derivative at the Pecher endpoint.
Strictly speaking we need it only in two space dimensions, but the proof is new and dimension independent, and we believe that the bounds it provides are of independent interest also in higher dimensions, and that the ideas here will also be useful for other
problems.

\blue{
For a broader perspective, one should note 
that the question of proving Strichartz estimates for linear Schr\"odinger evolutions with variable coefficients has been studied for more than 20 years. To start with, 
local in time full Strichartz estimates for compactly supported perturbations of the flat metric in $\R^n$ were considered in \cite{ST}, followed by \cite{RZ,TW,BT} for asymptotically flat metrics. Global in time results were then proved in \cite{T:Str} and then \cite{MMT-lin}, 
still in the asymptotically flat case. In the 
alternative setting of compact manifolds, \cite{BGT} proved local in time Strichartz estimates with derivative losses, obtained by adding up sharp Strichartz estimates on shorter, semiclassical time scales. For comparison purposes we note that all the metrics considered in these works have at least $C^2$ regularity, while the metrics in the present article are just above $C^\frac12$. On the other hand the estimates we prove in this section have losses of derivatives but are global in time,
so they cannot be obtained from lossless bounds in shorter time scales.}

\medskip

To review the set-up here, we seek $L^p_t L^q_x$ Strichartz bounds for the paradifferential problem
\begin{equation}\label{para-re}
i \partial_t v_\lambda + \partial_x g_{[<\lambda]} \partial_x v_\lambda 
= f_\lambda, \qquad v_\lambda(0) = v_{0,\lambda},
\end{equation}
where the pair of indices $(p,q)$ lies on the sharp Strichartz line
\begin{equation}
\frac{2}{p} + \frac{n}q = \frac{n}2, \qquad p \geq 2.    
\end{equation}
We will refer to the case $p=2$ as the Pecher endpoint, which is forbidden in dimension $n = 2$.

Classically one places the source term in dual Strichartz spaces.
While we certainly want to allow such sources, in our analysis 
we would like to allow for a larger class of functions which is
consistent with our needs for the quasilinear problem. One such class will be the space $DV^2_\Delta$ associated to the flat Schr\"odinger flow. Another class may be defined by  testing against solutions to the adjoint constant coefficient problem 
\begin{equation} \label{adjoint}
(i \partial_t + \Delta) w_\lambda = h_\lambda, \qquad w_{\lambda}(T) = w_{\lambda,T},
\end{equation}
defining a norm as follows:
\begin{equation}
\| f_\lambda \|_{\Ns} := \sup_{\| h_\lambda \|_{S'} + \|w_{\lambda,T}\|_{L^2_{t,x}} \leq 1} \left|\int f_\lambda w_\lambda \,dx dt\right| .
\end{equation}
For technical reasons we will use both $DV^2_\Delta$ and $N^\sharp$;
the first choice turns out to be better in the low dimension $n = 2$, as it allows us to bypass the forbidden endpoint,
and the second is better in the higher dimension $n \geq 3$, where the space $S'$ is well defined. Their common property is provided by the embeddings 
\begin{equation}\label{N-embed}
 N \subset \Ns, \qquad N \subset DV^2,
\end{equation}
which hold regardless of the dimension, where $N$ represents any of the admissible sharp dual Strichartz norms.

We will also need to measure 
$f_\lambda$ relative to $v_\lambda$, via the quantity 
\[
I(v_\lambda, f_\lambda) = \sup_{x_0} \| v_\lambda f_\lambda^{x_0}\|_{L^1_{t,x}} .
\]
One consequence of allowing such terms is that we can commute
the $x$ derivatives as we like in the paradifferential equation,
placing errors in the perturbative 
source term $f_\lambda$.

\begin{theorem}\label{t:para-se}
Assume that $u$ solves \eqref{qnls} in a time interval $[0,T]$, and satisfies the bootstrap bounds \eqref{uk-ee-boot},\eqref{uab-bi-unbal-boot},  \eqref{uab-bi-bal-boot}  and \eqref{uk-Str2-boot}  for some $s > \frac{n+1}2$. Then the following Strichartz estimates hold for the linear paradifferential equation \eqref{para-lin} for any pair of sharp Strichartz exponents $(p,q)$ and $(p_1,q_1)$:
\begin{equation}\label{Str-2}
\|v_\lambda\|_{L^\infty_t L^2_x} + \lambda^{-1} \| v_\lambda \|_{V^2_\Delta} \lesssim  \| v_{\lambda,0}\|_{L^2_x} 
 + \blue{\lambda^{-1}} \|f_\lambda\|_{DV^2_\Delta}
 + I(v_{\lambda},f_\lambda)^\frac12, \qquad n = 2,
\end{equation}
\begin{equation}\label{Str-3}
\|v_\lambda\|_{L^\infty_t L^2_x} + \lambda^{-1} \| v_\lambda \|_{L^2 L^\frac{2n}{n-2}} \lesssim  \| v_{\lambda,0}\|_{L^2_{x}} + 
 \blue{\lambda^{-1}} \|f_\lambda\|_{\Ns}
 + I(v_{\lambda},f_\lambda)^\frac12, \qquad n \geq 3.
\end{equation}
In particular we have
\begin{equation}\label{Str-all}
\lambda^{-\frac{2}p} \| v_\lambda \|_{L^p_t L^q_x} \lesssim  \| v_{\lambda,0}\|_{L^2} 
 + \lambda^{\frac{2}{p_1}} \|f_\lambda\|_{L^{p'_1}_tL^{q'_1}_x}.
\end{equation}
\end{theorem}

\blue{The first application of this result is to obtain Strichartz bounds for the solutions 
$u_\lambda$ to \eqref{qnls}:

\begin{corollary}\label{c:u-se}
  Assume that $u$ solves \eqref{qnls} in a time interval $[0,T]$, and satisfies the bootstrap bounds \eqref{uk-ee-boot},\eqref{uab-bi-unbal-boot},  \eqref{uab-bi-bal-boot}  and \eqref{uk-Str2-boot}. Then for all pairs $(p,q)$ of sharp  Strichartz exponents 
  $(p,q)$ we have:
  \begin{equation}
  \| u_\lambda \|_{L^p_t L^q_x}     \lesssim \epsilon c_\lambda \lambda^{-s+\frac{2}p}, 
  \qquad n \geq 3, \ T > 0.
  \end{equation}
  respectively
 \begin{equation}
  \| u_\lambda \|_{L^p_t L^q_x}     \lesssim \epsilon c_\lambda \lambda^{-s+\frac{2}p}, 
  \qquad n = 2, \ T\lesssim \epsilon^{-6}. 
  \end{equation}
  \end{corollary}

The case of the global in time bounds in dimension two in Theorem~\ref{t:local-fe2} (i),(iii) is discussed separately 
in Section~\ref{s:rough} for $s \geq 2$, respectively in Section~\ref{s:2low} for $s > \frac32$ under the conservative assumption.

}
\begin{proof}[Proof of Theorem~\ref{t:para-se}]
We first argue that \eqref{Str-all} follows from \eqref{Str-2}, 
respectively \eqref{Str-3}. On the left, we have a direct interpolation in the case of \eqref{Str-3}, while in the two dimensional case the interpolation is slightly more subtle,
\begin{equation}\label{interp2}
\lambda^{-\frac{2}p} \| v_\lambda \|_{L^p_t L^q_x} \lesssim 
\lambda^{-\frac{2}p} \| v_\lambda \|_{U^p_\Delta}
\lesssim \|v_\lambda\|_{L^\infty _tL^2_x} + \lambda^{-1} \| v_\lambda \|_{V^2_\Delta}.
\end{equation}
For the second term on the right in the case 
of \eqref{Str-3} we use the embedding $N \subset \Ns$. Similarly, in the case of \eqref{Str-2} we  have the embedding $L^{p'_1}_tL^{q'_1}_x \subset DV^{p'_1}_\Delta$.
Finally, the expression $I(v_\lambda,f_\lambda)$ 
can be estimated directly using H\"older's inequality.

Now we turn our attention to the proof of \eqref{Str-2} and \eqref{Str-3}.  We trivially have energy estimates
\[
\| v_\lambda \|_{L^\infty_t L^2_x} \lesssim \|v_\lambda(0)\|_{L^2_x} 
+ I(v_{\lambda},f_\lambda)^\frac12,
\]
so it remains to prove that 
\begin{equation}\label{Str-2a}
 \lambda^{-1} \| v_\lambda \|_{V^2_\Delta} \lesssim  \| v_{\lambda,0}\|_{L^2_x} 
 + \lambda^{-1} \|f_\lambda\|_{DV^2_\Delta}
 + I(v_{\lambda},f_\lambda)^\frac12, \qquad n = 2
\end{equation}
respectively
\begin{equation}\label{Str-3a}
 \lambda^{-1} \| v_\lambda \|_{L^2_t L^\frac{2n}{n-2}_x} \lesssim  \| v_{\lambda,0}\|_{L^2_x} +  \lambda^{-1} \|f_\lambda\|_{\Ns}
 + I(v_{\lambda},f_\lambda)^\frac12, \qquad n \geq 3.
\end{equation}
For this we reinterpret our equation as a flat Schr\"odinger equation by moving the $(g-I)$ contribution to the right, and then viewing it as a source term, which can be estimated as in $d_\lambda$ using two bilinear $L^2$ bounds. Precisely, we have
\begin{equation}\label{para-re-cc}
(i \partial_t + \Delta) v_\lambda 
= f_\lambda + f_\lambda^1, \qquad v_\lambda(0) = v_{0,\lambda},
\end{equation}
where
\begin{equation}
   f_\lambda^1 :=  \partial_x (g_{[<\lambda]}-I) \partial_x v_\lambda .
\end{equation}
\blue{Since $g-I$ is at least quadratic, we can express $f_\lambda^1$ in the form
\begin{equation}\label{fl1-rep}
  f_\lambda^1 = \lambda^2 L( h_0(\bfu_{\ll\lambda}), \bfu_{\ll\lambda}, \bfu_{\ll\lambda},v_\lambda)  
\end{equation}
with a smooth function $h_0$.
}

For \eqref{Str-2a} we estimate directly 
\[
\| v_\lambda \|_{V^2_\Delta} \lesssim \| v_\lambda(0)\|_{L^2_x}
+ \| f_\lambda\|_{DV^2_\Delta} 
+ \| f_\lambda^1\|_{DV^2_\Delta},
\]
where only the last term on the right still needs to be controlled,
\begin{equation}\label{DV2-bd}
\| f_\lambda^1\|_{DV^2_\Delta}^2 \lesssim \lambda ( \|v_\lambda(0)\|_{L^2_x}^2 + I(v_\lambda,f_\lambda)).
\end{equation}

We postpone the proof of the last bound, and turn our attention to \eqref{Str-3a}. There the argument goes by duality. Given $w_\lambda$ a solution to the constant coefficient adjoint  Schr\"odinger equation (which is the same) with source term $h_\lambda \in S'$  we write the duality relation
\[
\la v_\lambda,w_\lambda \ra|_0^T = \int^T_0\int v_\lambda h_\lambda - w_\lambda (f_\lambda +f_\lambda^1) \, dx dt,
\]
where we maximize the first term on the right over 
$\|h_{\lambda}\|_{S'} \leq 1$, which in turn  gives $\|w\|_{S} \lesssim 1$. 
This yields 
\[
\| v_\lambda\|_{S} \lesssim \|v_\lambda(0) \|_{L^2} + \|f_\lambda\|_{\Ns} + \|f_\lambda^1\|_{\Ns}.
\]
It remains to control the last term,
\begin{equation}\label{Ns-bd}
\| f_\lambda^1\|_{\Ns}^2 \lesssim \lambda ( \|v_\lambda(0)\|_{L^2_x}^2 + I(v_\lambda,f_\lambda)).
\end{equation}
This requires an $L^1_{t,x}$ bound for
$w_\lambda f_\lambda^1$, for which we use the representation \eqref{fl1-rep} to write 
\[
\| w_\lambda f_\lambda^1\|_{L^1_{t,x}} \lesssim \lambda^2 \sup_{x_0} \| v_\lambda u_{\ll \lambda}^{x_0}\|_{L^2_{t,x}}
 \sup_{x_0}\| w_\lambda u_{\ll \lambda}^{x_0} \|_{L^2_{t,x}}.
\]
Here we observe that for $w_\lambda$ we can control the associated $d_\lambda$ as in \eqref{dl} by $O(1)$, so we can use two unbalanced bilinear $L^2_{t,x}$ estimates from Theorem~\ref{t:para}, respectively Corollary~\ref{c:para} to bound this by
\[
\| w_\lambda f_\lambda^1\|_{L^1_{t,x}}  \lesssim  \lambda(\| v_\lambda(0)\|_{L^2_x}  + I(v_{\lambda},f_\lambda)^\frac12),
\]
as needed.

The proof of \eqref{DV2-bd} is somewhat similar. By duality 
it suffices to estimate in $L^1_{t,x}$ the product $f_\lambda^1 w_\lambda$
but this time for $w_\lambda \in U^2_\Delta$. It suffices to do this  for an $U^2_\Delta$ atom, 
\[
w_\lambda = \sum 1_{I_k} w_{\lambda,k},
\]
where $I_k$ are arbitrary disjoint time intervals and $w_{\lambda,k}$ are $L^2$ homogeneous Schr\"odinger waves.
Using Cauchy-Schwarz and two bilinear $L^2$ bounds as in Theorem~\ref{t:para} respectively Corollary~\ref{c:para} for each of the time intervals $I_k$ we have 
\[
\begin{aligned}
\| w_\lambda f_\lambda^1\|_{L^1_{t,x}} \lesssim & \  \lambda^2 
\sup_{x_0}\| v_\lambda  u_{\ll\lambda}^{x_0}
\|_{L^2} \sup_{x_0}(\sum_k \| w_{\lambda,k} u_{\ll \lambda}^{x_0} \|_{L^2_{t,x}(I_k)}^2)^\frac12 \\\lesssim & \ 
\lambda (\| v_\lambda(0)\|_{L^2_x}  + I(v_{\lambda},f_\lambda)^\frac12) ( \sum_k \| w_{\lambda,k}(0)\|_{L^2_x}^2)^\frac12, 
\end{aligned}
\]
as needed.

\end{proof}
\blue{
\begin{proof}[Proof of Corollary~\ref{c:u-se}]
To begin with, we recall that $u_\lambda$ solves the paradifferential equation 
with source term $f_\lambda = N_\lambda(u)$. Then in dimension $n \geq 3$ the bound for 
$I(u_\lambda,N_\lambda(u))$ is given by Proposition~\ref{p:N-lambda}. 
We also need the estimate 
\[
\|N_\lambda(u) \|_{N^\sharp} \lesssim \epsilon c_\lambda \lambda^{-s+1}  ,
\]
which is equivalent to
\[
\| N_\lambda(u) \bar w_\lambda\|_{L^1_{t,x}} \lesssim \epsilon c_\lambda \lambda^{-s} ( \|w_{\lambda,T}\|_{L^2_{t,x}}+ \|h_\lambda\|_{S'})
\]
for $w_\lambda$ solving \eqref{adjoint}, and is in turn proved by replacing one instance of \eqref{uab-bi-unbal-boot} by Corollary~\ref{c:para} in the 
proof of Proposition~\ref{p:N-lambda}.

In dimension $n=2$, the same applies but the last bound is now replaced by   
\[
\|N_\lambda \|_{DV^2_\Delta} \lesssim \epsilon c_\lambda \lambda^{-s}  .
\]
To see this again as a consequence of Corollary~\ref{c:para}, we use duality to test $N_\lambda(u)$ against $U^2_\Delta$ functions, so that the last bound is equivalent to 
\[
\| N_\lambda(u) \bar w_\lambda\|_{L^1_{t,x}} \lesssim \epsilon c_\lambda \lambda^{-s}  \|w_{\lambda}\|_{U^2_\Delta}.
\]
This is proved by replacing one instance of \eqref{uab-bi-unbal-boot} by Corollary~\ref{c:para1} in the proof of Proposition~\ref{p:N-lambda}.
\end{proof}
}


\section{The local/global well-posedness result}
\label{s:rough}

In this section we prove our main local well-posedness result in Theorems~\ref{t:local2} and \ref{t:local3+}. The same proof also yields the global well-posedness result in Theorem~\ref{t:global3} in high dimensions, 
as well as, with small adjustments, the two dimensional global well-posedness result in Theorem~\ref{t:global2}. For simplicity we will work with the problem \eqref{qnls}; the corresponding arguments for \eqref{dqnls} are identical, and are omitted.

Our starting point here is the local well-posedness result for regular data in Theorem~\ref{t:regular}. For these solutions we will prove the frequency envelope bounds in Theorems~\ref{t:local-fe3}, \ref{t:local-fe2}[(i),(ii)], which in turn allow us to continue the solutions, with 
uniform bounds, as follows:

\begin{enumerate}
    \item globally in time if $n \geq 3$,
    \item globally in time if $n=2$ for restricted $s$  i.e.,  $s \geq 2$.
    \item up to time $\epsilon^{-6}$ if $n=2$ at low regularity $s > \frac32$.
\end{enumerate}

We will then prove Theorems~\ref{t:local2},\ref{t:local3+},\ref{t:global3},\ref{t:global2}
by constructing rough solutions as unique limits of smooth solutions. We proceed in several steps:

\bigskip

\subsection{ A-priori bounds for regular solutions: Proof of Theorems~\ref{t:local-fe3},~\ref{t:local-fe2}[(i)(ii)].}
We recall that we can freely make the bootstrap assumption that 
the bounds \eqref{uk-ee-boot} and \eqref{uab-bi-unbal-boot} hold with a large universal constant $C$. Then we seek to apply Theorem~\ref{t:para} with $v_\lambda=u_\lambda$. To achieve this we write the equation for $u_\lambda$
in the paradifferential form \eqref{para-full}, using Proposition~\ref{p:N-lambda} for the source term bound. \blue{This suffices in dimension $n \geq 3$, and also in dimension $n = 2$ under the time constraint $T \ll \epsilon^{-6}$ corresponding to case (ii) of Theorem~\ref{t:local-fe2}.

It remains to consider the case $n=2$ for larger times, as in 
 case (i) of Theorem~\ref{t:local-fe2}. There, Proposition~\ref{p:N-lambda} yields
 the desired bound for the $F_\lambda$ component of $N_\lambda$, so it remains to estimate directly the contribution of the balanced cubic term $C_\lambda$. This we can do using an $L^4_{t,x}$ bound for $u_\lambda$, which then has to be inserted inside our bootstrap loop. Precisely, we claim that our solution $u$ also satisfies the dyadic $L^4_{t,x}$
 bounds \eqref{uk-Str2}. To prove this we add the corresponding bootstrap assumption \eqref{uk-Str2-boot}. This in particular allows us to estimate the $C_\lambda$ contribution  
as in Proposition~\ref{p:N-lambda} by
 \begin{equation}\label{Cl}
\| C_\lambda(\bfu,\bfu,\bfu) u_\lambda^{x_0}\|_{L^1_{t,x}} \lesssim \lambda^2 \|u_\lambda\|_{L^4_{t,x}}^4 \lesssim C^4 \epsilon^4 c_\lambda^4 \lambda^{-4s+4} \lesssim \epsilon^2 c_\lambda^2 \lambda^{-2s},
 \end{equation}
where at the last step we assumed $\epsilon \ll_C 1$ and, critically, the constraint $s \geq 2$.

 Finally, we need to close the $L^4_{t,x}$ part of the bootstrap loop and prove that \eqref{uk-Str2} holds. For this we use the $L^4_{t,x}$ Strichartz estimate \eqref{Str-2a} in Theorem~\ref{t:para-se}. By the interpolation inequality \eqref{interp2} this yields 
\[
\begin{aligned}
\lambda^{-\frac12} \|u_\lambda\|_{L^4_{t,x}} \lesssim & \ \|u_\lambda(0)\|_{L^2_{x}} 
+ \lambda^{-1} \|N_\lambda(u)\|_{DV_\Delta^2} + I(u_\lambda,N_\lambda(u))^\frac12
\\
\lesssim & \ \|u_\lambda(0)\|_{L^2_{x}} 
+ \lambda^{-1} \|F_\lambda(u)\|_{DV_\Delta^2} 
+  I(u_\lambda,F_\lambda(u))^\frac12
\\ & \ + \lambda^{-1} \|C_\lambda(\bfu,\bfu,\bfu)\|_{L^\frac43_{t,x}} 
+  I(u_\lambda,C_\lambda(\bfu,\bfu,\bfu))^\frac12.
\end{aligned}
\]
The $C_\lambda$ terms are estimated directly using the $L^4_{t,x}$ bootstrap bound 
\eqref{uk-Str2-boot}. The expression $I(u_\lambda,F_\lambda(u))$ has already been estimated in Proposition~\ref{p:N-lambda}. Lastly, the $DV^2_\Delta$ norm is estimated by 
\begin{equation}\label{Fl-DV2}
\lambda^{-1} \| F_\lambda(u)\|_{DV_\Delta^2} \lesssim \lambda^{-1} \| F_\lambda(u)\|_{L^\frac43_{t,x}} \lesssim C^3\epsilon^3 c_{\lambda}^2 \lambda^{-2s+\frac32}  \lesssim \epsilon c_{\lambda} \lambda^{-s},
\end{equation}
which suffices. Here, at the last step, we only used the fact that $F_\lambda$ is at least cubic and unbalanced, with the two highest frequencies comparable and $\gtrsim \lambda$,
therefore we can use a bilinear unbalanced bound \eqref{uab-bi-unbal-boot} and 
an $L^4_{t,x}$ bound \eqref{uk-Str2-boot}. 

Notably, we remark that our application of Theorem~\ref{t:para-se} only used the 
weaker constraint $s >\frac32$, while the stronger one $s \geq 2$ was only needed 
in \eqref{Cl}.

}

This concludes the proof of the above Theorems. For later use, we note that in particular the bounds for the paradifferential source terms in Proposition~\ref{p:N-lambda} are also valid for our solutions. 

\subsection{Higher regularity}
Here we consider regular solutions as above, with initial data $u_0$ which is not only $\epsilon$ small in $H^s$, but also belongs to a smaller space $H^{\sigma}$, say
\[
\|u_0\|_{H^s} \leq \epsilon, \qquad \| u_0\|_{H^{s_1}} \leq M, \qquad \sigma > s,
\]
where $M$ is possibly very large.
Then we claim that we have a uniform bound for the solution
in $H^{\sigma}$,  precisely
\begin{equation}
\| u\|_{L^\infty_t H^{\sigma}_x} \lesssim M.    
\end{equation}

To prove this,  for the initial data $u_0$ we consider a minimal $H^s$ frequency envelope $\epsilon c_\lambda$ with the unbalanced slowly varying condition as in Remark~\ref{r:unbal-fe}. Then by construction we must also have
\[
\sum_\lambda (\epsilon c_\lambda \lambda^{\sigma-s})^2 \lesssim M^2.
\]
By Theorems~\ref{t:local-fe3}, \ref{t:local-fe2} the frequency envelope $\epsilon c_\lambda$ is propagated along the flow, and we obtain
\[
\| u(t) \|_{H^{\sigma}}^2 \lesssim \sum_\lambda (\epsilon c_\lambda \lambda^{\sigma-s})^2 \lesssim M^2, \qquad t \in [0,T].
\]
We remark that not only the uniform $L^2$ higher regularity bounds are propagated along the flow, but also the corresponding bilinear $L^2_{t,x}$ and Strichartz bounds.

\subsection{ Continuation of regular solutions}  
Given a regular initial data $u_0 \in H^{\sigma}$ where $\sigma> \frac{n}2 +\frac32$, and  which is also small in $H^s$, we ask whether a regular solution exists
and for how long it can be continued:

\begin{proposition} \label{p:continuation}
For every regular initial data  $u_0 \in H^{\sigma}$  
and which also satisfies the smallness condition
\begin{equation}\label{small-data}
\| u_0 \|_{H^s} \leq \epsilon \ll 1,   
\end{equation}
there exists a local solution $u \in C_t H^{\sigma}_x$.
Furthermore, this local solution can be continued for as long as 
\[
\| u \|_{L^\infty_t H^s_x} \lesssim \epsilon.
\]
\end{proposition}

Combining this result with  Theorems~\ref{t:local-fe3},~\ref{t:local-fe2}[(i)(iii)], we obtain the following 
lifespan bounds:

\begin{corollary}\label{c:continuation}
Given initial data as in Proposition~\ref{p:continuation}, 
a regular solution exists 

(i) globally in time in dimension $n \geq 3$,

(ii) up to time $O(\epsilon^{-6})$ in dimension $n =2$.    
\end{corollary}

\begin{proof}[Proof of Proposition~\ref{p:continuation}]
For local solutions we want to apply Theorem~\ref{t:regular}. We remark that  even if we have the data $H^s$ smallness in \eqref{small-data}, there is no guarantee that the 
its $H^\sigma$ size is small. Hence we need the large data version 
of local well-posedness from \cite{MMT3}, which in turn requires verifying the  nontrapping assumption. Indeed, we claim that

\begin{lemma}\label{l:nontrap}
Let $\epsilon \ll 1$. Then the initial data $u_0$ satisfying the smallness condition \eqref{small-data} are uniformly nontrapping.
\end{lemma}

Here by uniformity we mean that the intersection of any geodesic 
for the metric $g$ with a ball $B(x_0,r)$ has length  $\lesssim r$,
with a uniform constant.

Assuming that the lemma holds, the large data local well-posedness theorem in \cite{MMT3} guarantees a local $H^\sigma$ solution. Let $[0,T)$ be the maximal interval on which an $H^\sigma$ solution exists.
Our apriori bounds guarantee that the $H^\sigma$ size of the solution remains bounded, 
\[
\|u(t) \|_{H^\sigma} \lesssim \|u(0) \|_{H^\sigma}, \qquad t \in [0,T),
\]
as well as that we have a  corresponding $H^\sigma$ frequency envelope bound, uniformly in time. This does not directly guarantee that the limit 
\[
u(T) = \lim_{t \to T} u(T)
\]
exists in $H^\sigma$.
However, by examining the \eqref{qnls} equation one easily sees that 
the limit exists in $H^{\sigma -2}$, which combined with the frequency envelope bound implies that the limit exists in $H^\sigma$.
This in turn contradicts the maximality of $T$, and concludes the proof of the proposition.
\medskip

It remains to prove the  nontrapping lemma:

\begin{proof}[Proof of Lemma~\ref{l:nontrap}]
The equations for the Hamilton flow have the form
\[
\dot x^j  = 2 g^{jk} \xi_k, \qquad \dot \xi_j = \partial_{x_j} g^{kl}
\xi_k \xi_l,
\]
where $g = g(u_0)$.
Here the full symbol $g^{jk} \xi_j \xi_k$ is a conserved quantity, 
so once we fix $|\xi(0)|=1$ by scaling, we have the global uniform bound $|\xi(t)|  \approx 1$.

We will prove that the bicharacteristics are nearly straight, in the sense that 
\begin{equation}\label{H-flow}
|\xi(t) - \xi(0)| \lesssim \epsilon, \qquad |\dot x(t) - \dot x(0)|
\lesssim \epsilon.
\end{equation}
Given that $|g - I| \lesssim \epsilon$,  the second bound clearly implies the first, so it remains to prove the second. For this we make a bootstrap assumption
\begin{equation}\label{H-flow-boot}
 |\xi(t) - \xi(0)|
\leq C\epsilon,
\end{equation}
with a large universal constant $C$. Then we also have 
\[
 |\dot x(t) - \dot x(0)| \lesssim C\epsilon \ll 1,
\]
which in turn indicates that our bicharacteristic, call it $\gamma$, is nearly straight.   It remains to estimate
\[
|xi(T) - \xi(0)| \lesssim \int_{0}^T |\nabla g(x(t)|\, ds
\lesssim \int_{0}^t |u_0(x(t))| |\nabla u_0(x(t))|\, ds.
\]
Since $\gamma$ is nearly straight and $s > \frac{n+1}{2}$, by the trace theorem it follows that 
\[
\| u\|_{L^2(\gamma)} + \|\nabla u\|_{L^2(\gamma)} \lesssim \| u_0\|_{H^s} \lesssim \epsilon.
\]
Then by Cauchy-Schwarz inequality we have 
\[
|\xi(T) - \xi(0)|\lesssim \epsilon^2,
\]
which closes our bootstrap and thus the proof of the nontrapping lemma.
\end{proof}

This in turn completes the proof of the continuation proposition.
\end{proof}

\subsection{A-priori $L^2$ bounds for the linearized equation: Proof of Theorem~\ref{t:linearize-fe}.}\label{s:lin-proof}

Here we seek again to apply Theorem~\ref{t:para} by writing the linearized equation in the form \eqref{para-lin}.  Correspondingly, we take $d_\lambda$ to be an $L^2$ frequency envelope for the initial data.

It suffices then to make a bootstrap assumption for $v_\lambda$,
which is the same as in the proof of Theorem~\ref{t:para},
and to prove the following bounds on the source terms $N_\lambda^{lin} v$, which represent the linearized equation counterpart of Proposition~\ref{p:N-lambda}:

\begin{proposition}
 Let $s \geq \frac{n+1}2$. Assume that the function $u$ satisfies   the bounds \eqref{uk-ee},\eqref{uab-bi-unbal} and \eqref{uab-bi-bal} in a time interval $[0,T]$, where $T \leq \epsilon^{-6}$
 in dimension $n=2$. Assume also that $v$ satisfies the bootstrap bounds \eqref{v-ee-boot}-\eqref{vv-bi-unbal-boot}. Then for $\epsilon$ small enough, the functions $N^{lin}_\lambda v$ satisfy 
 \begin{equation}\label{good-nl-lin}
\| N^{lin}_\lambda v  \cdot \bv_\lambda^{x_0}\|_{L^1_{t,x}} \lesssim \epsilon^2 d_\lambda^2 .
 \end{equation}
 The similar bound holds if we replace $\bv_\lambda^{x_0}$ by $\bar w_\lambda$, where $w_\lambda$ is a solution for the linear flat Schr\"odinger flow \eqref{adjoint}.
 
\end{proposition}

\begin{proof}
This proof is virtually identical to the proof of Proposition~\ref{p:N-lambda}, so the details are left for the 
interested reader.

\end{proof}

\bigskip

\subsection{Rough solutions as limits of smooth solutions}

Here we show that the $H^s$ energy estimates for the full equation, combined with the $L^2$ energy estimates for the linearized equation, both in the frequency envelope formulation, imply the local well-posedness results in 
Theorems~\ref{t:local2} and \ref{t:local3+}, as well as the global well-posedness results in Theorems \ref{t:global3} and \ref{t:global2}. A standard argument applies here, and for which we only outline the steps.  For more details we 
refer the reader to the same argument in our prior 1D paper \cite{IT-qnls},  the expository paper~\cite{IT-primer}, where the strategy of the proof is presented in detail,
\blue{as well as the more abstract presentation in \cite{ABITZ}.} Here it is more convenient to use the dyadic notation for frequencies, allowing also also for an easier comparison with \cite{IT-qnls}, \cite{IT-primer}.
The steps are as follows:

\begin{description}
    \item[I. Initial data regularization]
 given $u_0 \in H^s$, with frequency envelope $\epsilon c_k$, we consider 
 the regularized data 
\[
u_0^h = P_{<h} u_0,
\]
where $h \in \R^+$,
for which we can use an improved frequency envelope
\[
c^h_j = \left\{ \begin{array}{lc}
c_j & j \leq k,
\cr 
c_h \step^{-N(j-h)} & j > h,
\end{array} 
\right.
\]
with decay beyond frequency $2^h$.

For this data we consider the corresponding smooth solutions $u^h$. Using the continuation result for regular solutions in Proposition~\ref{p:continuation} and the  bounds in Theorems~\ref{t:local-fe3} and  \ref{t:local-fe2}, a continuity argument shows that these 
solutions extend as regular solutions
up to times $T$ as in  Theorems~\ref{t:local-fe3} and  \ref{t:local-fe2} with similar frequency envelope bounds
\begin{equation}\label{ul-fe}
 \|P_{j} u^h \|_{L^\infty_t L^2_x}
 \lesssim \epsilon c_j^h 
 \step^{-sj}.
\end{equation}
\medskip

\item[II. Difference bounds] Defining 
\[
v^h = \frac{d}{dh} u^h,
\]
 which solves the linearized equation around $u^h$,
with initial data
\[
v^h(0) = P_h u_0,
\]
satisfying
\[
\|v^h(0)\|_{L^2_x} \lesssim \epsilon \step^{-sh} c_h,
\]
we use the linearized energy estimates in Theorem~\ref{t:linearize-fe} to obtain the $L^2$ bounds
\begin{equation}\label{vl-fe}
\| v^h\|_{L^\infty_t L^2_x} \lesssim  \epsilon\step^{-sh} c_h.   
\end{equation}

\medskip

\item[III. Convergence] 
Here we define the rough solution  $u$ as the limit of $u^h$ via the telescopic sum
\[
u = u^0 + \sum_{k\geq 0} u^{k+1}-u^k.
\]
By \eqref{vl-fe} we have 
\begin{equation}\label{dul-fe}
\| u^{k+1}-u^k\|_{L^\infty_t L^2_x} \lesssim  \epsilon \step^{-sk} c_k,  
\end{equation}
which implies rapid convergence of the series in $L^2$, and 
\begin{equation}\label{dul-fe-u}
\| u-u^k\|_{L^\infty_t L^2_x} \lesssim  \epsilon \step^{-sk} c_k.   
\end{equation}

On the other hand by \eqref{ul-fe} we have 
\begin{equation}\label{dul-feh}
\| u^{k+1}-u^k\|_{L^\infty_t H^{s+m}_x} \lesssim  \epsilon \step^{sm} c_k.
\end{equation}
Combining the last two relations, it follows that the summands in the series are almost orthogonal in $H^s$, and we have
\begin{equation}\label{ddul-fe-diff}
\| u^{k}-u^j\|_{L^\infty_t H^{s}_x} \lesssim  \epsilon c_{[j,k]}, \qquad j < k.
\end{equation}
This shows convergence in $L^\infty_t H^s_x$, and 
\begin{equation}\label{ddul-fe-ldiff}
\| u-u^j\|_{L^\infty_t H^{s}_x} \lesssim  \epsilon c_{\geq j}. 
\end{equation}
Now it is easily verified that the limit $u$ solves the equation \eqref{qnls}, and also that it satisfies all the bounds in Theorems~\ref{t:local-fe3} and \ref{t:local-fe2}.

\medskip

\item[IV. Continuous dependence] This again follows 
the argument in \cite{IT-primer}. The weak Lipschitz
dependence implies convergence in a weaker topology, which is then upgraded to strong convergence using the uniform decay of the high frequency tails derived from frequency envelope bounds.  
\end{description}

\subsection{Scattering }
Our aim here is to prove the scattering part of our 
global well-posedness theorems. For clarity we begin with the simpler higher dimensional case in Theorem~\ref{t:global3}, and then discuss the differences in the 2D case.

We begin by writing the equation for each $u_\lambda$
as as constant coefficient Schr\"{o}dinger flow,
\[
(i\partial_t + \Delta) u_\lambda = f_{\lambda}:= N_\lambda(u) - (g_{[<\lambda]}- I) \partial^2 u_\lambda.
\]
Now we estimate the source term $f_\lambda$, taking our
cue from Section~\ref{s:Str}. For $N_\lambda$ we have 
the bounds from Proposition~\ref{p:N-lambda}, while for 
the second term on the right we repeat the proof of \eqref{Ns-bd}. This yields
\begin{equation}\label{Flambda}
\| f_{\lambda} \|_{\Ns} \lesssim \epsilon^3 c_\lambda^3 \lambda^{-s+1}    ,
\end{equation}
which has a one derivative loss. This bound holds globally in time, but if we restrict the time interval we further 
obtain the decay
\begin{equation}
\lim_{T \to \infty}    \| f_{\lambda} \|_{\Ns[T,\infty)} = 0.
\end{equation}
This is a consequence of the fact that the proof of 
\eqref{Flambda} uses at least one translation free bilinear $L^2_{t,x}$ bound, for which we can directly use the
absolute continuity of the $L^2$ norm.

We remark that we could also simply use $L^{\frac43}_{t,x}$ bounds, but that would be non-optimal as it would involve a larger loss of derivatives (two in 3D).
The bound \eqref{Flambda} and the above decay property allow us to estimate
\begin{equation}
\| u_\lambda (t) - e^{i(t-s) \Delta} u_\lambda(s)\|_{L^2_x} \lesssim   \epsilon^3 c_\lambda^3 \lambda^{-s+1},
\end{equation}
with the continuity property
\begin{equation}
   \lim_{t,s \to \infty}  \| u_\lambda(t) - e^{i(t-s) \Delta} u_{\lambda}(s)\|_{L^2_x} = 0.
\end{equation}
This implies that the limit
\[
u_{\lambda}^\infty = \lim_{t \to \infty} e^{-it\Delta} u_\lambda(t) 
\]
exists in $L^2$, and 
\begin{equation}
  \|  u_{\lambda}^\infty - u_{0\lambda}\|_{L^2_x}  \lesssim   \epsilon^3 c_\lambda^3 \lambda^{-s+1}.
\end{equation}
On the other hand, passing to the limit in the energy bounds for $u_\lambda$ we also get
\begin{equation}
   \|  u_{\lambda}^\infty\|_{L^2_x}  \lesssim   \epsilon c_\lambda \lambda^{-s}  .
\end{equation}
Then we can set
\[
u^\infty = \sum_\lambda u_\lambda^\infty \in H^s,
\]
where, combining the above properties, we have 
the strong limit
\begin{equation}
\lim_{t \to \infty}     e^{-it\Delta} u(t) = u^\infty \qquad \text{ in } H^s, 
\end{equation}
i.e., in the strong topology. This concludes the proof of the scattering result in dimensions three and higher.

The proof of the scattering result in two dimensions is similar, with only two differences:

\begin{itemize}
    \item The space $\Ns$ is now replaced by $DV^2_\Delta$, exactly as in the proof of the Strichartz estimates~\ref{s:Str}.

    \item The bound for $N_\lambda$ now separates into two parts: the unbalanced part $F_\lambda$, for which we can still use the estimates in Proposition~\ref{p:N-lambda}, and the cubic balanced part $C_\lambda$, for 
    which we simply use the $L^4$ bound \eqref{uk-Str2}.
    This allows us to estimate $C_\lambda$ in $L^\frac43 \subset DV^2$.
    We remark that the $L^4$ bound \eqref{uk-Str2} has a loss of $1/2$ derivative, but this is acceptable since our target is a bound with a one derivative loss, and we are at least one half derivative above scaling. 
    
\end{itemize}

\section{
The 2D global result at low regularity} \label{s:2low}

The aim of this final section is to complete the proof of the 
global well-posedness result in Theorem~\ref{t:global2c}. In view of the local well-posedness result in Theorem~\ref{t:local2} and its associated frequency envelope  bounds in Theorem~\ref{t:local-fe2}, it suffices to show that the frequency envelope bounds in Theorem~\ref{t:local-fe2} (iii) hold globally in time. Further, it suffices to do this under appropriate bootstrap assumptions, see Proposition~\ref{p:boot}.

For the proof of the local well-posedness result in Theorem~\ref{t:local2}, as well as the global result,
except for the low regularity regime in dimension $n=2$,
it was enough to interpret the (QNLS) evolution as a paradifferential equation with a perturbative source term. Here, on the other hand, we also single out the balanced cubic terms.
At higher regularity $s \geq 2$, these balanced cubic terms were estimated using the $L^4_{t,x}$ Strichartz estimates \eqref{uk-Str2}
with a half derivative loss. At lower regularity, this loss 
turns out to be too large, and instead we need a more complex argument. For clarity, we begin with a brief description of our strategy for balanced cubic terms:

\begin{enumerate}[label=(\roman*)]
    \item Low frequency terms: for these we simply establish and use an $L^4_{t,x}$ Strichartz bound perturbatively.
    \item High frequency terms without phase rotation symmetry. These may be further divided into two terms
\begin{itemize}
 \item transversal, where the three input frequencies and the output frequency may be split into two pairs of separate frequencies, and then we can perturbatively use two bilinear    
$L^2_{t,x}$ bounds.

\item nonresonant, which heuristically may be eliminated using a normal form correction, which in our setting corresponds to quartic density and flux corrections for the mass and the momentum. 
\end{itemize}    
\item High frequency terms with phase rotation symmetry, where we use the conservative assumption to construct a suitable flux correction modulo a perturbative part.
\end{enumerate}

Our goal is to prove the bounds \eqref{uk-ee}, \eqref{uab-bi-unbal}, \eqref{uab-bi-bal} and \eqref{uk-Str2}. In doing this, without any loss of generality we can make the bootstrap assumptions \eqref{uk-ee-boot}, \eqref{uab-bi-unbal-boot}, \eqref{uab-bi-bal-boot} and \eqref{uk-Str2-boot}. To accomplish this, we expand the analysis started in Section~\ref{s:df-2c}.

We consider the localized mass and momentum densities
\[
M_\lambda = M_\lambda(u,\bar u), 
\]
\[
P_{\lambda}^j = P_{\lambda}^j(u,\bar u) ,
\]
which satisfy the conservation laws
\[
\partial_t M_\lambda(u) = \partial_j  P^j_{\lambda}(u)
+ C^{4,res}_{\lambda,m}(u) + C^{4,nr}_{\lambda,m}(u) +F^4_{\lambda,m}(u),
\]
respectively
\[
\partial_t P^j_{\lambda}(u) = \partial_x E^{jk}_{\lambda}(u)
+ C^{4,res}_{\lambda,p}(u) + C^{4,nr}_{\lambda,p}(u)+F^4_{\lambda, p}(u),
\]
where $C^{4,res}$ represents the contribution of balanced cubic terms with phase rotation symmetry and $C^{4,nr}$ represent the 
contribution of balanced cubic nonresonant terms without the phase 
rotation symmetry. Here we remark that these terms are needed only at high frequency $\lambda \gg 1$, as at low frequency $\lambda \lesssim 1$ we can place all balanced cubic contributions in $F^4_{\lambda}$ due to the $L^4_{t,x}$ bootstrap bound \eqref{uk-Str2-boot}.

From Lemma~\ref{l:Fmp} we have bounds for the source terms $F^4_{\lambda}$,
namely
\begin{equation}\label{pert-flux-source}
 \|F^4_{\lambda,m}(u) \|_{L^1_{t,x}} \lesssim \epsilon^4 C^4 c_\lambda^4 \lambda^{-2s},
 \qquad
 \|F^4_{\lambda,p}(u) \|_{L^1_{t,x}} \lesssim \epsilon^4 C^4 c_\lambda^4
\lambda^{-2s+1}.
\end{equation}

However, the contributions of $C^{4,res}_{\lambda}$ and $C^{4,nr}_{\lambda}$
cannot be also estimated directly in $L^1_{t,x}$, and will be treated separately.

\medskip

\emph{ A. The $C^{4,res}_{\lambda}$ term.} Here we need to use the conservative assumption. This guarantees that the quartic forms $C^{4,res}$ have the additional structural property 
\begin{equation}\label{c4-conserv}
c^{4,res}_{\lambda,m}(\xi,\xi,\xi,\xi) = 0, 
\qquad 
c^{4,res}_{\lambda,p}(\xi,\xi,\xi,\xi) = 0, \quad \xi \in \R^2.
\end{equation}

At the symbol level, this allows us to obtain a decomposition 
\[
\lambda^{-1} c^{4,res}_{\lambda,m}(\xi_1,\xi_2,\xi_3,\xi_4)
= i(\xi_{odd} -\xi_{even})\cdot  r^{4,res}_{\lambda,m}(\xi_1,\xi_2,\xi_3,\xi_4) + i \Delta^4 \xi \cdot q^{4,res}_{\lambda,m}
(\xi_1,\xi_2,\xi_3,\xi_4),
\]
where all symbols on the right are smooth on the $\lambda$ scale and bounded.
Separating variables in $r^{4,res}_{\lambda, m}$, we arrive at the following 

\begin{lemma}
The quartic form $C^{4,res}_m$ admits a smooth representation
of the form  
\begin{equation}
\lambda^{-1} C^{4,res}_{\lambda,m}(u,\bu,u,\bu) = \partial_x
 Q^{4,res}_{\lambda,m}(u,\bu,u,\bu) + \sum_j \partial_x R^2_{j,a}(u,\bu)
 R^2_{j,b} (u,\bu),
\end{equation}
where the symbols $r_{j,a}$, $r_{j,b}$ are localized at frequency $\lambda$, bounded 
and smooth on the $\lambda$ scale, and sum is rapidly convergent in $j$.
\end{lemma}
We remark that one corollary of this is the bound 
\begin{equation} \label{en-bal}
\left| \int  C^{4,res}_m(u,\bu,u,\bu)\, dx dt \right|
\lesssim \lambda \sum_j \| |D|^\frac12 R^2_{j,a}(u,\bu)\|_{L^2_{t,x}}
\| |D|^\frac12 R^2_{j,b}(u,\bu)\|_{L^2_{t,x}},
\end{equation}
where the right hand side may be estimated using 
the bilinear balanced bound \eqref{uab-bi-bal-boot}.
This suffices for energy estimates, but not for interaction Morawetz, where we would like to have $L^1_{t,x}$ bounds instead.

To address this issue, there we use a paraproduct decomposition for the terms in the series, writing 
\[
\begin{aligned}
\partial_x R^2_{j,a}(u,\bu) R^2_{j,b} (u,\bu)
= & \ 
T_{\partial_x R^2_{j,a}(u,\bu)} R^2_{j,b} (u,\bu)
+ \Pi(\partial_x R^2_{j,a}(u,\bu), R^2_{j,b} (u,\bu))
\\ &\ - T_{\partial_x Q^2_{j,b}(u,\bu)} Q^2_{j,a} (u,\bu)
 + \partial_x T_{ Q^2_{j,b} (u,\bu)} Q^2_{j,a}(u,\bu).
\end{aligned}
\]
The first three terms can be estimated in $L^1$
by the right hand side in \eqref{en-bal}, while the 
last term together with $Q^{4,res}_{\lambda,m}$ go into a flux correction. Precisely, we have proved the following
\begin{proposition}\label{p:bal-sources}
Assume that $c^{4,res}_{\lambda, m}$ satisfies \eqref{c4-conserv}. 
Let $u$ satisfy \eqref{uk-ee-boot} and \eqref{uab-bi-bal-boot}.
Then there exists a decomposition
\begin{equation}
 C^{4,res}_{\lambda,m}(u) = \partial_x \tilde Q^{4,res}_{\lambda,m}(u) +  F^{4,res}_{\lambda,m}(u) , 
\end{equation}
so that 
\begin{equation}\label{flux-bal}
\|  \tilde Q^{4,res}_{\lambda,m}(u) \|_{L^\frac65_t L^\frac43_x} \lesssim C^4 \epsilon^4 c_\lambda^4 \lambda^{1+\frac56-4s}    ,
\end{equation}
while
\begin{equation}\label{error-cor-res}
\|  F^{4,res}_{\lambda,m}(u) \|_{L^1_{t,x}} \lesssim C^4 \epsilon^4 c_\lambda^4 \lambda^{3-4s}  .  
\end{equation} 
\end{proposition}

We note that in \eqref{flux-bal} we have in effect a full range
of estimates between $L^1_t L^2_x$ and $L^\frac32_t L^1_x$, of which the endpoints are not so useful to us, so we chose the middle.
To prove it, we interpolate between the energy bound \eqref{uk-ee-boot} and the balanced bilinear $L^2$ bound \eqref{uab-bi-bal-boot}.

\medskip{The $C^{4,nr}$ term.} For this term we use a normal form type correction to both the densities and the fluxes:

\begin{proposition}\label{p:nr-sources}
 
Let $u$ satisfy \eqref{uk-ee-boot} and \eqref{uab-bi-bal-boot}.
Then there exists a decomposition
\begin{equation}
 C^{4,nr}_{\lambda,m}(u) = - \partial_t B^{4,nr}_{\lambda,m}(u)
 + \partial_x Q^{4,nr}_{\lambda,m}(u) +  F^{6,nr}_{\lambda,m}(u)  ,
\end{equation}
so that we have the fixed time bound
\begin{equation}\label{energy-cor-nr}
\|  B^{4,nr}_{\lambda,m}(u) \|_{L^1_x} \lesssim C^4 \epsilon^4 c_\lambda^4 \lambda^{2-4s}    ,
\end{equation}
and the space-time bounds
\begin{equation}\label{flux-nr}
\|  Q^{4,nr}_{\lambda,m}(u) \|_{L^\frac65_t L^\frac43_x} \lesssim C^4 \epsilon^4 c_\lambda^4 \lambda^{1+\frac56-4s}    ,
\end{equation}
respectively
\begin{equation}\label{error-cor-nr}
\|  F^{6,nr}_{\lambda,m}(u) \|_{L^1_{t,x}} \lesssim C^4 \epsilon^4 c_\lambda^4 \lambda^{3-4s}.    
\end{equation}
\end{proposition}

\begin{proof}
The corrections are computed as if the linear part were the flat 
Schr\"odinger flow. This yields quartic forms with smooth symbols
\[
B^{4,nr}_{\lambda,m}(u) = L(u_\lambda,u_\lambda,u_\lambda,u_\lambda),
\qquad Q^{4,nr}_{\lambda,m}(u) = \lambda L(u_\lambda,u_\lambda,u_\lambda,u_\lambda).
\]
Then the bound \eqref{energy-cor-nr} follows from the energy estimate and 
Bernstein's inequality, while \eqref{flux-nr} is exactly as in the previous Proposition. It remains to estimate the error term, which has the form
\[
F^{6,nr}_{\lambda,m}(u) = \lambda^2 L(g_{[<\lambda]} - I, u_\lambda,u_\lambda,u_\lambda,u_\lambda) +  L(N_\lambda(u), u_\lambda,u_\lambda,u_\lambda):= F^{6,nr,1}_{\lambda,m}(u)+
F^{6,nr,2}_{\lambda,m}(u).
\]
In the first term we can use two bilinear $L^2_{t,x}$ bounds and two $L^\infty$ 
bounds from energy and Bernstein's inequality to arrive at
\[
\| F^{6,nr,1}_{\lambda,m}(u)\|_{L^1_{t,x}} \lesssim C^6 \epsilon^6 c_\lambda^4 
\lambda^{-4s+n+1},
\]
which suffices.

To obtain a similar estimate for the second term  $F^{6,nr,2}_{\lambda,m}(u)$ we separate $N_\lambda = F_\lambda+ C_\lambda$, 
where for $F_\lambda$ we use Proposition~\ref{p:N-lambda} and two $L^\infty$ bounds from energy and Bernstein's inequality, while for $C_\lambda$ 
we obtain 
\[
\lambda^2 L(u_\lambda,u_\lambda,u_\lambda,u_\lambda,u_\lambda,u_\lambda),
\]
which can be estimated directly by combining \eqref{uk-ee-boot} with 
\eqref{uab-bi-bal-boot} and Bernstein's inequality. This yields a frequency factor $\lambda^{-6s+4}$, which is more than sufficient.

\end{proof}

To summarize our findings, we have the 
modified densities 
\begin{equation}\label{ms}
\ms_\lambda(u) = M_\lambda(u) + B^{4,nr}_{\lambda,m}(u),    
\end{equation}
respectively
\begin{equation}\label{ps}
\ps_\lambda(u) = P_\lambda(u) + B^{4,nr}_{\lambda,p}(u),    
\end{equation}
and the associated density-flux identities
\begin{equation}\label{df2-lo-m}
\partial_t \ms_\lambda(u) = \partial_j  (P^j_{\lambda}(u)+ Q^{4,res,j}_{\lambda,m}+Q^{4,nr,j}_{\lambda,m}) +F^{6,nr}_{\lambda,m}(u)
+F^{4,res}_{\lambda,m}(u) +F^4_{\lambda,m}(u),
\end{equation}
respectively
\begin{equation}\label{df2-lo-p}
\partial_t {\ps}^j_{\lambda}(u) = \partial_k (E^{jk}_{\lambda}(u)
+ Q^{4,res,jk}_{\lambda,p}+Q^{4,nr,jk}_{\lambda,p}) +F^{6,nr,j}_{\lambda,m}(u)
+F^{4,res,j}_{\lambda,p}(u)  +F^{4,j}_{\lambda,p}(u),
\end{equation}
where the flux corrections $Q^{4,res}$,  $Q^{4,nr}$  are as in \eqref{flux-bal}, respectively \eqref{flux-nr},
and all the $F^4$ and $F^6$ terms are estimated in $L^1$ as in \eqref{pert-flux-source}, \eqref{error-cor-res} and \eqref{energy-cor-nr}.

\subsection{The uniform energy bounds}
Here we integrate the density-flux identity \eqref{df2-lo-m} using the bounds in Proposition~\ref{p:bal-sources} and Proposition~\ref{p:nr-sources} in order to prove the dyadic energy estimates \eqref{uk-ee} in our 
main result in Theorem~\ref{t:global2c}, under the appropriate bootstrap assumptions as stated in Proposition~\ref{p:boot}. \blue{ We remark that the mass and momentum corrections 
play a perturbative role at fixed time in this computation,
\begin{equation}\label{B4-pert}
\begin{aligned}
\|   B^{4,nr}_{\lambda,m}(u) \|_{L^1_{t,x}} \lesssim C^4 \epsilon^4 c_\lambda^4 \lambda^{-4s+2}
\lesssim \epsilon^2 c_\lambda^2 \lambda^{-2s},
\\
\|   B^{4,nr}_{\lambda,p}(u) \|_{L^1_{t,x}} \lesssim C^4 \epsilon^4 c_\lambda^4 \lambda^{-4s+3}
\lesssim \epsilon^2 c_\lambda^2 \lambda^{-2s+1},
\end{aligned}
\end{equation}
provided we are above scaling $s \geq 1$.
}

\subsection{ The interaction Morawetz identities}

Here we prove the bilinear $L^2_{t,x}$ bounds \eqref{uab-bi-unbal}
and \eqref{uab-bi-bal}.

Overall, we will seek to pair a frequency $\lambda$ 
portion of one solution $u$ with a frequency $\mu$ portion 
of another solution $v$, which will eventually be taken to be 
$v = u^{x_0}$, i.e. a translate of $u$. The argument follows
the computations done earlier in Section~\ref{s:para},
with the only difference that we need to add in the contribution of our density and flux corrections.

We will write here 
all identities in the general case. Then, in order to 
prove the global result in Theorem~\ref{t:global2c} we will specialize to three cases:
\begin{enumerate}
    \item The diagonal case $\lambda = \mu$, $u=v$ (and thus $x_0=0$).
    \item The balanced shifted case $\lambda \approx \mu$, 
    with  $x_0$ arbitrary.
\item The unbalanced case $ \mu < \lambda$, 
    with  $x_0$ arbitrary.
\end{enumerate}

We define the interaction Morawetz functional for the functions 
$(u,v)$ and the associated pair of frequencies $(\lambda,\mu)$ as
\begin{equation}\label{IM}
\bI_{\lambda\mu}(u,v) :=   \iint a_j(x-y) \left(\ms_\lambda(u)(x) {\ps}^j_{\mu}(v) (y) -  
{\ps}^j_{\lambda}(u)(x) \ms_{\mu}(v) (y)\right) \, dx dy.
\end{equation}
We use the modified density-flux identities for the mass and for the momentum in order to compute the time derivative of $\bI_{\lambda\mu}$ as
\begin{equation}\label{interaction-xilm}
\frac{d}{dt} \bI_{\lambda\mu} =  \bJ^4_{\lambda\mu} + \bJ^{6}_{\lambda\mu} + \bK_{\lambda\mu}. 
\end{equation}
Compared to the computations done in the case of the paradifferential equation, \blue{ 
these terms are mostly similar with two exceptions:
\medskip

(i) the mass and momentum corrections in $\bI_{\lambda\mu}$ whose contributions  are easily controlled using \eqref{B4-pert}.

\medskip
(ii) the expression  $\bJ^{6}_{\lambda \mu}$, which contains the contribution 
of our density and flux corrections, namely it has the form
}
{\small
\[
\begin{aligned}
\bJ^{6}_{\lambda \mu} := & \ \iint a_{jk}(x-y) \left(Q^{4,res,k}_{\lambda,m}(u) +Q^{4,nr,k}_{\lambda,m}(u))(x){\ps}^j_{\mu}(v) (y) -  
(Q^{4,res,jk}_{\lambda,p}(u) +Q^{4,nr,jk}_{\lambda,p}(u))(x) \ms_{\mu}(v) (y) \, \right. 
\\ & \left. - \ms_\lambda(u)(x) (Q^{4,res,jk}_{\mu,p}(v)+ Q^{4,nr,jk}_{\mu,p}(v)) (y) +  
{\ps}^j_{\lambda}(u)(x) (Q^{4,res,k}_{\mu,m}(v) +Q^{4,nr,k}_{\mu,m}(v))(y)\right) \, dx dy.
\end{aligned}
\]
}
To obtain the same outcome as in the paradifferential case, 
it will suffice to estimate the contribution of $\bJ^{6}_{\lambda \mu}$ perturbatively, i.e., to show that 
\begin{equation}\label{J6-bal-want}
 \left| \int_0^T    \bJ^{6}_{\lambda \mu} \, dt \right|
 \lesssim \epsilon^4 c_\lambda^2 c_\mu^2 \lambda^{1-2s}\mu^{-2s}.
\end{equation}
For simplicity we consider one of the second term expressions, which is the worst if $\mu \lesssim \lambda$. Since $|a_{jk}| \leq \dfrac{1}{|x-y|} $, we can use Young's inequality to estimate
\[
 \left| \int_0^T  \iint Q^{4,res,jk}_{\lambda,p}(u)(x) \ms_{\mu}(v) (y)  \, dxdydt \right| \lesssim \| Q^{4,res,jk}_{\lambda,p}(u)\|_{L^\frac65_t L^\frac43_x} \| M_{\mu}(v)\|_{L^6_t L^\frac43_x},
\]
and similarly for $Q^{4,nr,jk}_{\lambda,p}$.
For $\ms_\mu(v)$ we can interpolate the $L^\infty_t L^1_x$ bound provided by \eqref{uk-ee-boot}
and the $L^2_t L^4_x$ bound derived from \eqref{uab-bi-bal-boot} to 
obtain
\[
\| \ms_{\mu}(v)\|_{L^6_t L^\frac43_x} \lesssim C^2 \epsilon^2 c_\mu^2 \mu^{\frac16-2s},
\]
where the quartic correction is better than the leading quadratic term.

Combining this with \eqref{flux-bal} or \eqref{flux-nr} we get 
\[
\lesssim C^6 \epsilon^6 c_\lambda^4 c_\mu^2 \lambda^{2+\frac56-2s}
\mu^{\frac16-2s}
\]
which is better than \eqref{J6-bal-want} provided that $s > 1$. This concludes the proof of \eqref{J6-bal-want}, and thus the proof of  \eqref{uab-bi-bal}
and \eqref{uab-bi-unbal}.

\blue{
\subsection{ The $L^4_{t,x}$ Strichartz bounds}
Here we prove \eqref{uk-Str2}, thereby concluding the proof of Theorem~\ref{t:global2c}.
With minor changes, this is as in the proof of Theorem~\ref{t:para-se}. One difference is that here we have already established the uniform energy bounds for $u_\lambda$. By the interpolation inequality \eqref{interp2}, it remains to show the $V^2_\Delta$ bound with a derivative loss, 
\begin{equation}\label{ul-V2}
\|u_\lambda \|_{V^2_\Delta} \lesssim \epsilon c_\lambda \lambda^{-s+1}    
\end{equation}
For this we 
rewrite the $u_\lambda$ equation as
\begin{equation}\label{para-re-cc+}
(i \partial_t + \Delta) u_\lambda 
= N_\lambda + f_\lambda^1, \qquad u_\lambda(0) = u_{0,\lambda},
\end{equation}
where
\begin{equation}
   f_\lambda^1 :=  \partial_x (g_{[<\lambda]}-I) \partial_x u_\lambda,
\end{equation}
and estimate directly 
\[
\| u_\lambda \|_{V^2_\Delta} \lesssim \| v_\lambda(0)\|_{L^2_x}
+ \| N_\lambda\|_{DV^2_\Delta} 
+ \| f_\lambda^1\|_{DV^2_\Delta},
\]
Here the second term on the right is estimated as in \eqref{Fl-DV2}
while the last term on the right is estimated by duality, testing against a function $w \in U^2_\Delta$,
\[
\left| \int f_\lambda^1 \bar w_\lambda dx dt \right|
\lesssim \| f_\lambda^1 \bar w_\lambda \|_{L^1_{t,x}}  \lesssim \lambda^2 \| L(\bfu_{\ll\lambda}, u_\lambda)\|_{L^2_{t,x}} \| L(\bfu_{\ll \lambda}, w_\lambda)\|_{L^2_{t,x}}
\lesssim C^3 \epsilon^3 \lambda c_\lambda \lambda^{-s+1},
\]
where at the last step we have used two bilinear unbalanced $L^2_{t,x}$ bounds, namely 
one instance of \eqref{uab-bi-unbal-boot} and one application of Corollary~\ref{c:para1}
( repeated in the context of this section, using the modified density flux identities 
for $u_\lambda$).

This concludes the proof of \eqref{ul-V2} and thus the proof of the theorem.

}

 \bibliographystyle{plain}


\begin{thebibliography}{10}

\bibitem{ABITZ}
Thomas {Alazard}, Nicolas {Burq}, Mihaela {Ifrim}, Daniel {Tataru}, and Claude
  {Zuily}.
\newblock {Nonlinear interpolation and the flow map for quasilinear equations}.
\newblock {\em arXiv e-prints}, page arXiv:2410.06909, October 2024.

\bibitem{AD}
Thomas Alazard and Jean-Marc Delort.
\newblock Global solutions and asymptotic behavior for two dimensional gravity
  water waves.
\newblock {\em Ann. Sci. \'{E}c. Norm. Sup\'{e}r. (4)}, 48(5):1149--1238, 2015.

\bibitem{MR2443925}
Ioan Bejenaru and Daniel Tataru.
\newblock Large data local solutions for the derivative {NLS} equation.
\newblock {\em J. Eur. Math. Soc. (JEMS)}, 10(4):957--985, 2008.

\bibitem{MR2811056}
Fernando Bernal-V\'{\i}lchis, Nakao Hayashi, and Pavel~I. Naumkin.
\newblock Quadratic derivative nonlinear {S}chr\"{o}dinger equations in two
  space dimensions.
\newblock {\em NoDEA Nonlinear Differential Equations Appl.}, 18(3):329--355,
  2011.

\bibitem{BT}
Jean-Marc Bouclet and Nikolay Tzvetkov.
\newblock Strichartz estimates for long range perturbations.
\newblock {\em Amer. J. Math.}, 129(6):1565--1609, 2007.

\bibitem{Bourgain-nf}
Jean Bourgain.
\newblock A remark on normal forms and the ``{$I$}-method'' for periodic {NLS}.
\newblock {\em J. Anal. Math.}, 94:125--157, 2004.

\bibitem{BGT}
N.~Burq, P.~G\'erard, and N.~Tzvetkov.
\newblock Strichartz inequalities and the nonlinear {S}chr\"odinger equation on
  compact manifolds.
\newblock {\em Amer. J. Math.}, 126(3):569--605, 2004.

\bibitem{MR3336355}
Jean-Yves Chemin and Delphine Salort.
\newblock Wellposedness of some quasi-linear {S}chr\"{o}dinger equations.
\newblock {\em Sci. China Math.}, 58(5):891--914, 2015.

\bibitem{MR1344627}
Hiroyuki Chihara.
\newblock Local existence for semilinear {S}chr\"{o}dinger equations.
\newblock {\em Math. Japon.}, 42(1):35--51, 1995.

\bibitem{Co}
M.~Colin.
\newblock On the local well-posedness of quasilinear {S}chr\"{o}dinger
  equations in arbitrary space dimension.
\newblock {\em Comm. Partial Differential Equations}, 27(1-2):325--354, 2002.

\bibitem{MR1886962}
M.~Colin.
\newblock On the local well-posedness of quasilinear {S}chr\"{o}dinger
  equations in arbitrary space dimension.
\newblock {\em Comm. Partial Differential Equations}, 27(1-2):325--354, 2002.

\bibitem{I-method}
J.~Colliander, M.~Keel, G.~Staffilani, H.~Takaoka, and T.~Tao.
\newblock Almost conservation laws and global rough solutions to a nonlinear
  {S}chr\"{o}dinger equation.
\newblock {\em Math. Res. Lett.}, 9(5-6):659--682, 2002.

\bibitem{MR2415387}
J.~Colliander, M.~Keel, G.~Staffilani, H.~Takaoka, and T.~Tao.
\newblock Global well-posedness and scattering for the energy-critical
  nonlinear {S}chr\"{o}dinger equation in {$\mathbb R^3$}.
\newblock {\em Ann. of Math. (2)}, 167(3):767--865, 2008.

\bibitem{BNPS}
Anne de~Bouard, Nakao Hayashi, Pavel~I. Naumkin, and Jean-Claude Saut.
\newblock Scattering problem and asymptotics for a relativistic nonlinear
  {S}chr\"{o}dinger equation.
\newblock {\em Nonlinearity}, 12(5):1415--1425, 1999.

\bibitem{D}
Jean-Marc Delort.
\newblock Semiclassical microlocal normal forms and global solutions of
  modified one-dimensional {KG} equations.
\newblock {\em Ann. Inst. Fourier (Grenoble)}, 66(4):1451--1528, 2016.

\bibitem{D-focus}
Benjamin Dodson.
\newblock Global well-posedness and scattering for the mass critical nonlinear
  {S}chr\"{o}dinger equation with mass below the mass of the ground state.
\newblock {\em Adv. Math.}, 285:1589--1618, 2015.

\bibitem{MR1284428}
Shin-ichi Doi.
\newblock On the {C}auchy problem for {S}chr\"{o}dinger type equations and the
  regularity of solutions.
\newblock {\em J. Math. Kyoto Univ.}, 34(2):319--328, 1994.

\bibitem{GMS}
P.~Germain, N.~Masmoudi, and J.~Shatah.
\newblock Global solutions for 2{D} quadratic {S}chr\"{o}dinger equations.
\newblock {\em J. Math. Pures Appl. (9)}, 97(5):505--543, 2012.

\bibitem{MR1340853}
J.~Ginibre and N.~Hayashi.
\newblock Almost global existence of small solutions to quadratic nonlinear
  {S}chr\"{o}dinger equations in three space dimensions.
\newblock {\em Math. Z.}, 219(1):119--140, 1995.

\bibitem{TW}
Andrew Hassell, Terence Tao, and Jared Wunsch.
\newblock A {S}trichartz inequality for the {S}chr\"odinger equation on
  nontrapping asymptotically conic manifolds.
\newblock {\em Comm. Partial Differential Equations}, 30(1-3):157--205, 2005.

\bibitem{HN}
Nakao Hayashi and Pavel~I. Naumkin.
\newblock Asymptotics for large time of solutions to the nonlinear
  {S}chr\"{o}dinger and {H}artree equations.
\newblock {\em Amer. J. Math.}, 120(2):369--389, 1998.

\bibitem{MR1856255}
Nakao Hayashi and Pavel~I. Naumkin.
\newblock Global existence of small solutions to the quadratic nonlinear
  {S}chr\"{o}dinger equations in two space dimensions.
\newblock {\em SIAM J. Math. Anal.}, 32(6):1390--1403, 2001.

\bibitem{HN1}
Nakao Hayashi and Pavel~I. Naumkin.
\newblock Large time asymptotics for the fractional nonlinear {S}chr\"{o}dinger
  equation.
\newblock {\em Adv. Differential Equations}, 25(1-2):31--80, 2020.

\bibitem{MR1255899}
Nakao Hayashi and Tohru Ozawa.
\newblock Remarks on nonlinear {S}chr\"{o}dinger equations in one space
  dimension.
\newblock {\em Differential Integral Equations}, 7(2):453--461, 1994.

\bibitem{HTT}
Sebastian Herr, Daniel Tataru, and Nikolay Tzvetkov.
\newblock Global well-posedness of the energy-critical nonlinear
  {S}chr\"{o}dinger equation with small initial data in {$H^1(\mathbb T^3)$}.
\newblock {\em Duke Math. J.}, 159(2):329--349, 2011.

\bibitem{HLT}
Jiaxi Huang, Ze~Li, and Daniel Tataru.
\newblock Global {R}egularity of {S}kew {M}ean {C}urvature {F}low for {S}mall
  {D}ata in {$d\geq 4$} {D}imensions.
\newblock {\em Int. Math. Res. Not. IMRN}, (5):3748--3798, 2024.

\bibitem{HT1}
Jiaxi Huang and Daniel Tataru.
\newblock Local well-posedness of skew mean curvature flow for small data in
  {$d\geq4$} dimensions.
\newblock {\em Comm. Math. Phys.}, 389(3):1569--1645, 2022.

\bibitem{HT2}
Jiaxi Huang and Daniel Tataru.
\newblock Local well-posedness of the skew mean curvature flow for small data
  in {$d\geqq 2$} dimensions.
\newblock {\em Arch. Ration. Mech. Anal.}, 248(1):Paper No. 10, 79, 2024.

\bibitem{hi}
John~K. Hunter and Mihaela Ifrim.
\newblock Enhanced life span of smooth solutions of a {B}urgers-{H}ilbert
  equation.
\newblock {\em SIAM J. Math. Anal.}, 44(3):2039--2052, 2012.

\bibitem{BH}
John~K. Hunter, Mihaela Ifrim, Daniel Tataru, and Tak~Kwong Wong.
\newblock Long time solutions for a {B}urgers-{H}ilbert equation via a modified
  energy method.
\newblock {\em Proc. Amer. Math. Soc.}, 143(8):3407--3412, 2015.

\bibitem{MR0948533}
Wataru Ichinose.
\newblock On {$L^2$} well posedness of the {C}auchy problem for
  {S}chr\"{o}dinger type equations on the {R}iemannian manifold and the
  {M}aslov theory.
\newblock {\em Duke Math. J.}, 56(3):549--588, 1988.

\bibitem{IT-NLS}
Mihaela Ifrim and Daniel Tataru.
\newblock Global bounds for the cubic nonlinear {S}chr\"{o}dinger equation
  ({NLS}) in one space dimension.
\newblock {\em Nonlinearity}, 28(8):2661--2675, 2015.

\bibitem{IT-g}
Mihaela Ifrim and Daniel Tataru.
\newblock Two dimensional water waves in holomorphic coordinates {II}: {G}lobal
  solutions.
\newblock {\em Bull. Soc. Math. France}, 144(2):369--394, 2016.

\bibitem{IT-c}
Mihaela Ifrim and Daniel Tataru.
\newblock The lifespan of small data solutions in two dimensional capillary
  water waves.
\newblock {\em Arch. Ration. Mech. Anal.}, 225(3):1279--1346, 2017.

\bibitem{IT-BO}
Mihaela Ifrim and Daniel Tataru.
\newblock Well-posedness and dispersive decay of small data solutions for the
  {B}enjamin-{O}no equation.
\newblock {\em Ann. Sci. \'{E}c. Norm. Sup\'{e}r. (4)}, 52(2):297--335, 2019.

\bibitem{IT-global}
Mihaela Ifrim and Daniel Tataru.
\newblock Global solutions for 1{D} cubic defocusing dispersive equations:
  {P}art {I}.
\newblock {\em Forum Math. Pi}, 11:Paper No. e31, 46, 2023.

\bibitem{IT-qnls}
Mihaela Ifrim and Daniel Tataru.
\newblock Global solutions for 1d cubic dispersive equations, part iii: the
  quasilinear schr\"odinger flow, 2023.

\bibitem{IT-primer}
Mihaela Ifrim and Daniel Tataru.
\newblock Local well-posedness for quasi-linear problems: a primer.
\newblock {\em Bull. Amer. Math. Soc. (N.S.)}, 60(2):167--194, 2023.

\bibitem{IT-conjecture}
Mihaela {Ifrim} and Daniel {Tataru}.
\newblock {The global well-posedness conjecture for 1D cubic dispersive
  equations}.
\newblock {\em arXiv e-prints}, page arXiv:2311.15076, November 2023.

\bibitem{IT-focusing}
Mihaela Ifrim and Daniel Tataru.
\newblock Long time solutions for 1d cubic dispersive equations, {P}art {II}:
  {T}he focusing case.
\newblock {\em Vietnam J. Math.}, 52(3):597--614, 2024.

\bibitem{IT-packet}
Mihaela Ifrim and Daniel Tataru.
\newblock Testing by wave packets and modified scattering in nonlinear
  dispersive {PDE}'s.
\newblock {\em Trans. Amer. Math. Soc. Ser. B}, 11:164--214, 2024.

\bibitem{2024arXiv240206278J}
In-Jee {Jeong} and Sung-Jin {Oh}.
\newblock {Wellposedness of the electron MHD without resistivity for large
  perturbations of the uniform magnetic field}.
\newblock {\em arXiv e-prints}, page arXiv:2402.06278, February 2024.

\bibitem{KP}
Jun Kato and Fabio Pusateri.
\newblock A new proof of long-range scattering for critical nonlinear
  {S}chr\"{o}dinger equations.
\newblock {\em Differential Integral Equations}, 24(9-10):923--940, 2011.

\bibitem{Ke-Ta}
Markus Keel and Terence Tao.
\newblock Endpoint {S}trichartz estimates.
\newblock {\em Amer. J. Math.}, 120(5):955--980, 1998.

\bibitem{KPRV1}
C.~E. Kenig, G.~Ponce, C.~Rolvung, and L.~Vega.
\newblock Variable coefficient {S}chr\"{o}dinger flows for ultrahyperbolic
  operators.
\newblock {\em Adv. Math.}, 196(2):373--486, 2005.

\bibitem{KPRV2}
Carlos~E. Kenig, Gustavo Ponce, Christian Rolvung, and Luis Vega.
\newblock The general quasilinear ultrahyperbolic {S}chr\"{o}dinger equation.
\newblock {\em Adv. Math.}, 206(2):402--433, 2006.

\bibitem{MR1230709}
Carlos~E. Kenig, Gustavo Ponce, and Luis Vega.
\newblock Small solutions to nonlinear {S}chr\"{o}dinger equations.
\newblock {\em Ann. Inst. H. Poincar\'{e} C Anal. Non Lin\'{e}aire},
  10(3):255--288, 1993.

\bibitem{MR1660933}
Carlos~E. Kenig, Gustavo Ponce, and Luis Vega.
\newblock Smoothing effects and local existence theory for the generalized
  nonlinear {S}chr\"{o}dinger equations.
\newblock {\em Invent. Math.}, 134(3):489--545, 1998.

\bibitem{KPV}
Carlos~E. Kenig, Gustavo Ponce, and Luis Vega.
\newblock The {C}auchy problem for quasi-linear {S}chr\"{o}dinger equations.
\newblock {\em Invent. Math.}, 158(2):343--388, 2004.

\bibitem{KT}
Herbert Koch and Daniel Tataru.
\newblock Energy and local energy bounds for the 1-d cubic {NLS} equation in
  {$H^{-\frac14}$}.
\newblock {\em Ann. Inst. H. Poincar\'{e} Anal. Non Lin\'{e}aire},
  29(6):955--988, 2012.

\bibitem{Lin-Ponce}
Wee~Keong Lim and Gustavo Ponce.
\newblock On the initial value problem for the one dimensional quasi-linear
  {S}chr\"{o}dinger equations.
\newblock {\em SIAM J. Math. Anal.}, 34(2):435--459, 2002.

\bibitem{LLS}
Hans Lindblad, Jonas L\"{u}hrmann, and Avy Soffer.
\newblock Asymptotics for 1{D} {K}lein-{G}ordon equations with variable
  coefficient quadratic nonlinearities.
\newblock {\em Arch. Ration. Mech. Anal.}, 241(3):1459--1527, 2021.

\bibitem{LS}
Hans Lindblad and Avy Soffer.
\newblock Scattering and small data completeness for the critical nonlinear
  {S}chr\"{o}dinger equation.
\newblock {\em Nonlinearity}, 19(2):345--353, 2006.

\bibitem{MMT-lin}
Jeremy Marzuola, Jason Metcalfe, and Daniel Tataru.
\newblock Strichartz estimates and local smoothing estimates for asymptotically
  flat {S}chr\"odinger equations.
\newblock {\em J. Funct. Anal.}, 255(6):1497--1553, 2008.

\bibitem{MMT1}
Jeremy~L. Marzuola, Jason Metcalfe, and Daniel Tataru.
\newblock Quasilinear {S}chr\"{o}dinger equations {I}: {S}mall data and
  quadratic interactions.
\newblock {\em Adv. Math.}, 231(2):1151--1172, 2012.

\bibitem{MMT2}
Jeremy~L. Marzuola, Jason Metcalfe, and Daniel Tataru.
\newblock Quasilinear {S}chr\"{o}dinger equations, {II}: {S}mall data and cubic
  nonlinearities.
\newblock {\em Kyoto J. Math.}, 54(3):529--546, 2014.

\bibitem{MMT3}
Jeremy~L. Marzuola, Jason Metcalfe, and Daniel Tataru.
\newblock Quasilinear {S}chr\"{o}dinger equations {III}: {L}arge data and short
  time.
\newblock {\em Arch. Ration. Mech. Anal.}, 242(2):1119--1175, 2021.

\bibitem{MR0860041}
Sigeru Mizohata.
\newblock {\em On the {C}auchy problem}, volume~3 of {\em Notes and Reports in
  Mathematics in Science and Engineering}.
\newblock Academic Press, Inc., Orlando, FL; Science Press Beijing, Beijing,
  1985.

\bibitem{2023arXiv231019221P}
Ben {Pineau} and Mitchell~A. {Taylor}.
\newblock {Low regularity solutions for the general Quasilinear ultrahyperbolic
  Schr{\"o}dinger equation}.
\newblock {\em arXiv e-prints}, page arXiv:2310.19221, October 2023.

\bibitem{PV}
Fabrice Planchon and Luis Vega.
\newblock Bilinear virial identities and applications.
\newblock {\em Ann. Sci. \'{E}c. Norm. Sup\'{e}r. (4)}, 42(2):261--290, 2009.

\bibitem{Pop}
Markus Poppenberg.
\newblock On the local well posedness of quasilinear {S}chr\"{o}dinger
  equations in arbitrary space dimension.
\newblock {\em J. Differential Equations}, 172(1):83--115, 2001.

\bibitem{RZ}
Luc Robbiano and Claude Zuily.
\newblock Strichartz estimates for {S}chr\"odinger equations with variable
  coefficients.
\newblock {\em M\'em. Soc. Math. Fr. (N.S.)}, (101-102):vi+208, 2005.

\bibitem{2023arXiv231102556S}
Jie {Shao} and Yi~{Zhou}.
\newblock {Local wellposedness for the quasilinear Schr{\"o}dinger equations
  via the generalized energy method}.
\newblock {\em arXiv e-prints}, page arXiv:2311.02556, November 2023.

\bibitem{Shatah}
Jalal Shatah.
\newblock Normal forms and quadratic nonlinear {K}lein-{G}ordon equations.
\newblock {\em Comm. Pure Appl. Math.}, 38(5):685--696, 1985.

\bibitem{ST}
Gigliola Staffilani and Daniel Tataru.
\newblock Strichartz estimates for a {S}chr\"odinger operator with nonsmooth
  coefficients.
\newblock {\em Comm. Partial Differential Equations}, 27(7-8):1337--1372, 2002.

\bibitem{MR1191488}
Jiro Takeuchi.
\newblock Le probl\`eme de {C}auchy pour certaines \'{e}quations aux
  d\'{e}riv\'{e}es partielles du type de {S}chr\"{o}dinger. {VIII}.
  {S}ym\'{e}trisations ind\'{e}pendantes du temps.
\newblock {\em C. R. Acad. Sci. Paris S\'{e}r. I Math.}, 315(10):1055--1058,
  1992.

\bibitem{Tao-WM}
Terence Tao.
\newblock Global regularity of wave maps. {II}. {S}mall energy in two
  dimensions.
\newblock {\em Comm. Math. Phys.}, 224(2):443--544, 2001.

\bibitem{Tao-BO}
Terence Tao.
\newblock Global well-posedness of the {B}enjamin-{O}no equation in {$H^1({\bf
  R})$}.
\newblock {\em J. Hyperbolic Differ. Equ.}, 1(1):27--49, 2004.

\bibitem{T-unpublished}
Daniel Tataru.
\newblock On global existence and scattering for the wave maps equation,
  unpublished original version, 1999.

\bibitem{T:Str}
Daniel Tataru.
\newblock Parametrices and dispersive estimates for {S}chr\"odinger operators
  with variable coefficients.
\newblock {\em Amer. J. Math.}, 130(3):571--634, 2008.

\end{thebibliography}


\end{document}